\documentclass[12pt,reqno,a4paper]{amsart}
\usepackage[margin=25truemm]{geometry}
\usepackage{amssymb}
\usepackage{stmaryrd}
\usepackage{graphicx,xcolor}
\usepackage{tikz}

\usepackage{amsfonts}
\usepackage{mathrsfs}
\usepackage{bm}
\usepackage{bbm}
\usepackage{xcolor}

\allowdisplaybreaks[4]

\hypersetup{% option list for hyperref
 setpagesize=false,
 hypertexnames=false,
 hyperfootnotes=true,
 bookmarksnumbered=true,%
 bookmarksopen=true,%
 colorlinks=true,%
 linkcolor=blue,
 citecolor=blue,
}
\usepackage[shortlabels]{enumitem}
\theoremstyle{plain}
    \newtheorem{theorem}{Theorem}[section]
    \newtheorem*{theorem*}{Theorem}
    \newtheorem{lemma}[theorem]{Lemma}
    \newtheorem{proposition}[theorem]{Proposition}
    \newtheorem{corollary}[theorem]{Corollary}

\theoremstyle{definition}
    \newtheorem{definition}{Definition}[section]
    \newtheorem{remark}{Remark}[section]
    
\numberwithin{equation}{section}

\newcommand{\even}{\mathord{\mathrm{even}}}
\newcommand{\odd}{\mathord{\mathrm{odd}}}

\DeclareMathOperator{\Span}{Span}
\DeclareMathOperator{\diag}{diag}

\let\Re\relax%これでamsmathで定義されている\Reを消している.(\let\Re=\undefined としてもよい)
\DeclareMathOperator{\Re}{Re}%これで\Reを再定義している.
\let\Im\relax%これでamsmathで定義されている\Imを消している.(\let\Im=\undefinedとしてもよい)
\DeclareMathOperator{\Im}{Im}%これで\Imを再定義している.

\DeclareMathOperator{\sech}{sech}

\newcommand{\eps}{\varepsilon}
\newcommand{\C}{\mathbb{C}}
\newcommand{\N}{\mathbb{N}}

\newcommand{\R}{\mathbb{R}}

\newcommand{\boD}{\mathcal{D}}

\newcommand{\boL}{\mathcal{L}}

\newcommand{\boR}{\mathcal{R}}

\newcommand{\boV}{\mathcal{V}}

\newcommand{\boY}{\mathcal{Y}}
\newcommand{\boZ}{\mathcal{Z}}

\newcommand{\la}{\ensuremath{\lambda}}

\newcommand{\te}{\ensuremath{\theta}}
\newcommand{\al}{\ensuremath{\alpha}}
\newcommand{\be}{\ensuremath{\beta}}
\newcommand{\gam}{\ensuremath{\gamma}}

\newcommand{\ka}{\ensuremath{\kappa}}
\newcommand{\si}{\ensuremath{\sigma}}
\newcommand{\Si}{\ensuremath{\Sigma}}
\newcommand{\om}{\ensuremath{\omega}}

\makeatletter												%
  \@addtoreset{equation}{section}						%
\makeatother

\newcommand{\nor}[2]{\left\| {#1} \right\|_{#2}}		%	Norm
\newcommand{\ovl}[1]{\overline{#1}}					%	Complex conjugate
		
\newcommand{\inp}[2]{\langle {#1} , {#2} \rangle }	%	inner product

\newcommand{\Del}{{\Delta}}							%	Laplacian
\newcommand{\del}{{\delta}}								%	Small delta
\newcommand{\rd}{{\partial}}								%	Round d
\newcommand{\nab}{{\nabla}}							%	Nabla

\newcommand{\wto}{\rightharpoonup}

\newcommand{\boa}{{\bm{a}}}
\newcommand{\boy}{{\bm{y}}}
\newcommand{\boz}{{\mathbf{z}}}
\newcommand{\bal}{{\bm{\al}}}
\newcommand{\bbe}{{\bm{\be}}}
\newcommand{\bla}{{\bm{\la}}}

\newcommand{\bmu}{{\bm{\mu}}}
\newcommand{\bnu}{{\bm{\nu}}}

\newcommand{\brho}{{\bm{\rho}}}

\begin{document}
%%%%%%%%%%%%%%%%%%%%%%%%%%%%%%%%%%%%%%%%%%%%%%%%%%%%%%%%

\title[Multi-solitons for delta NLS]{Multi-solitons for the nonlinear Schr\"{o}dinger equation with repulsive Dirac delta potential}

\author[S. Gustafson]{Stephen Gustafson}
\address[S. Gustafson]{University of British Columbia, 1984 Mathematics Rd., Vancouver, V6T1Z2, Canada.}
\email{gustaf@math.ubc.ca}

\author[T. Inui]{Takahisa Inui}
\address[T. Inui]{Department of Mathematics, Graduate School of Science, Osaka University, Toyonaka, Osaka, 560-0043, Japan.
%\newline 
%University of British Columbia, 1984 Mathematics Rd., Vancouver,  V6T1Z2, Canada.
}
\email{inui@math.sci.osaka-u.ac.jp}

\author[I. Shimizu]{Ikkei Shimizu}
\address[I. Shimizu]{Department of Mathematics, Graduate School of Science, Osaka University, Toyonaka, Osaka, 560-0043, Japan.}
\email{shimizu@cr.math.sci.osaka-u.ac.jp}

\date{\today}
\maketitle

\begin{abstract}
We prove the existence of multi-soliton solutions for the nonlinear Schr\"{o}dinger equation with repulsive Dirac delta potential and $L^2$-supercritical focusing nonlinear term. Our main contribution is to treat the unmoving part of the multi-solitons, which is the ground state of the equation. The linearized operator around it has two unstable eigenvalues. This is the main difference from NLS without potential, whose existence of multi-solitons is investigated by C\^{o}te, Martel, and Merle (2011). 
\end{abstract}

\tableofcontents

%%%%%%%%%%%%%%%%%%%%%%%%%%%%%%%%%%%%%%%%%%

\section{Introduction}

We consider the following nonlinear Schr\"{o}dinger eqaution with repulsive Dirac delta potential in one dimension: 
\begin{align}
\label{deltaNLS}
\tag{$\delta$NLS}
	i \partial_{t} u + \partial_{x}^{2} u + \gamma \delta u +|u|^{p-1}u=0,
	\quad (t,x) \in \mathbb{R} \times \mathbb{R},
\end{align}
where $\gamma <0$, $\delta$ denotes the Dirac delta at the origin, and $p>1$. The condition $\gamma<0$ means that the potential is repulsive. From the view point of the scaling, we say that the equation is $L^2$-subcritical when $1<p<5$, $L^2$-critical when $p=5$, and $L^2$-supercritical when $p>5$. We mainly consider the $L^2$-supercritical case in the present paper. 
%though our results in the present paper are expected to be true even in the $L^2$-subcritical and $L^2$-critical. 
Here, the Schr\"{o}dinger operator $-\Delta_{\gamma}:=-\partial_{x}^{2} -\gamma \delta $ is defined by 
\begin{align*}
	- \Delta_{\gamma}f &:=- \partial_{x}^{2}f \text{ for } f \in \mathcal{D}(-\Delta_{\gamma}),
	\\
	\mathcal{D}(- \Delta_{\gamma}) &:= \{ f \in H^{1}(\mathbb{R}) \cap H^{2}(\mathbb{R}\setminus\{0\}): f'(0+) - f'(0-)=- \gamma f(0)\}. 
\end{align*}
%We assume here that $\gamma<0$, meaning that the potential is repulsive. In this case, 
Since $\gamma<0$, $-\Delta_{\gamma}$ is a nonnegative self-adjoint operator on $L^{2}(\mathbb{R})$ (see \cite{AGHKH88} for more details), which implies that the linear Schr\"{o}dinger propagator $e^{it\Delta_{\gamma}}$ is well-defined on $L^{2}(\mathbb{R})$ by Stone's theorem. The operator $-\Delta_{\gamma}$ may also be defined by its quadratic form:
\begin{align*}
	q( f , g) := \int_{\mathbb{R}} f'(x)\overline{g'(x)} dx - \gamma f(0)\overline{g(0)} 
\end{align*} 
for $f,g \in H^{1}(\mathbb{R})$. See e.g. \cite{Kat95,Sch12}. 

It is known that the equation \eqref{deltaNLS} is locally well-posed in the energy space $H^{1}(\mathbb{R})$. Moreover, the energy $E_\gamma$ and mass $M$, which are defined by
\begin{align*}
	E_{\gamma}(f)&:= \frac{1}{2}\|\partial_{x} f\|_{L^{2}}^{2} - \frac{\gamma}{2} |f(0)|^{2}  - \frac{1}{p+1}\|f\|_{L^{p+1}}^{p+1},
	\\
	M(f)&:= \|f\|_{L^{2}}^{2},
\end{align*}
are conserved (see \cite{GHW04,FOO08}). 

It is also known that there exists a ground state standing wave solution $e^{i\omega t}Q_{\omega,\gamma}$, where $Q_{\omega,\gamma}$ is the positive function defined by 
\begin{align*}
	Q_{\omega,\gamma}(x):= \left[ \frac{(p+1)\omega}{2} \sech^{2}\left\{ \frac{(p-1)\sqrt{\omega}}{2}|x| + \tanh^{-1}\left( \frac{\gamma}{2\sqrt{\omega}} \right) \right\} \right]^{\frac{1}{p-1}}
\end{align*}
for $\omega>\gamma^2/4$. See e.g. \cite{FuJe08}. 
This function $Q_{\omega,\gamma}$ is a unique positive solution of 
\begin{align*}
	-\partial_{x}^{2}Q_{\omega,\gamma} -\gamma \delta Q_{\omega,\gamma} + \omega Q_{\omega,\gamma} =Q_{\omega,\gamma}^{p}.
\end{align*}

We are interested in the existence of multi-soliton solutions for \eqref{deltaNLS}. First, we recall results for NLS without potential. 
We consider the following nonlinear Schr\"{o}dinger equation:
\begin{align}
\label{NLS}
	i\partial_t u +\partial_x^2 u +|u|^{p-1}u=0, \quad (t,x) \in \mathbb{R} \times  \mathbb{R},
\end{align}
where $p>5$, which is \eqref{deltaNLS} with $\gamma=0$. 
The function $e^{i\omega t}Q_{\omega,0}$ for $\omega>0$ is a ground state standing wave solution of the equation \eqref{NLS}. The equation \eqref{NLS} is gauge invariant, translation invariant, and Galilean invariant. That is, if $u$ is a solution, then 
\begin{align*}
	u_1(t,x) := u(t, x-v_1 t -x_1)e^{i\left(\frac{1}{2}v_1 x - \frac{1}{4}v_1^2 t + \theta_1  \right)}
\end{align*}
is also a solution for $v_1,x_1,\theta_1 \in \mathbb{R}$. 
Therefore, for given $v_1,x_1,\theta_1 \in \mathbb{R}$ and $\omega_1 >0$, the function 
\begin{align*}
	\mathcal{R}_1(t,x):= Q_{\omega_1,0}(x-v_1 t -x_1)e^{i\left(\frac{1}{2}v_1 x - \frac{1}{4}v_1^2 t +\omega_1 t + \theta_1  \right)}
\end{align*}
is a one-soliton solution of the equation \eqref{NLS}. The following theorem showed existence of a multi-soliton solution. 

\begin{theorem*}[\cite{CMM11}]
Assume $p>5$. Let $K\geq 2$, $v_1,...,v_K,x_1,...,x_K,\theta_1,...,\theta_K \in \mathbb{R}$, and $\omega_1,...,\omega_K >0$. We assume that $v_1<v_2<...<v_K$. Define 
\begin{align*}
	\mathcal{R}(t,x)&:=\sum_{k=1}^{K}\mathcal{R}_k(t,x),
	\\
	\mathcal{R}_k(t,x)&:= Q_{\omega_k,0}(x-v_k t -x_k)e^{i\left(\frac{1}{2}v_k x - \frac{1}{4}v_k^2 t +\omega_k t + \theta_k  \right)}.
\end{align*}
Then there exists $T_0 \in \mathbb{R}$, $C,c_0>0$, and a solution $u\in C([T_0,\infty);H^1(\mathbb{R}))$ to \eqref{NLS} such that 
\begin{align*}
	\|u(t) - \mathcal{R}(t)\|_{H^1} \leq C e^{-c_0 t}
\end{align*}
for $t \in [T_0,\infty)$. 
\end{theorem*}

\begin{remark}
In fact, for NLS with $H^1$-subcritical nonlinearity in general dimensions, the existence of the multi-solitons is known. See \cite{Mer90,MaMe06,CMM11}. 
\end{remark}

We turn now to equation \eqref{deltaNLS}. We will construct multi-soliton solutions for \eqref{deltaNLS}. If $v \neq 0$, then the interaction between the Dirac delta potential at the origin and the moving one-soliton 
\begin{align*}
	Q_{\omega_1,0}(x-v_1 t -x_1)e^{i\left(\frac{1}{2}v_1 x - \frac{1}{4}v_1^2 t +\omega_1 t + \theta_1  \right)}
\end{align*}
are exponentially decaying in time. Thus we expect that we can construct multi-soliton for \eqref{deltaNLS} if their solitons have non-zero velocity. Moreover, if one of them has zero velocity, then we expect that the multi-soliton with $Q_{\omega,\gamma}$ replacing $Q_{\omega,0}$ can be constructed.

%%%%%%%%%%%%%%%%%%%%%

\subsection{Main results}

Before giving our results, we introduce some notations. 
Let $K$ be a positive integer. 
We set $\llbracket 1, K\rrbracket := \{1,2,...,K\}$. For a given $K$, we let 
\begin{align*}
	v_1, ....,v_K \in \mathbb{R},
	\quad
	\omega_{1},,...,\omega_{K}\in (0,\infty),
	\quad
	x_{1},,...,x_{K}\in \mathbb{R},
	\quad 
	\theta_{1},...,\theta_{K}\in \mathbb{R}.
\end{align*}
We assume that $v_k \neq v_j$ for any distinct $j,k \in \llbracket 1,K \rrbracket$. Therefore, in what follows, we assume that 
\begin{align*}
	v_1<v_2< ....<v_K 
\end{align*}
by renumbering. 
We additionally assume that $\omega_{k_0} > \gamma^2/4$ and $x_{k_0} =0$ if $k_0 \in \llbracket 1, K \rrbracket$ satisfies $v_{k_0} =0$. 
For simplicity of the notation, we set 
\begin{align*}
	\gamma_k:=
	\begin{cases}
	\gamma & \text{ for } k \text{ such that }v_k=0,
	\\
	0 & \text{ for } k \text{ such that }v_k\neq 0. 
	\end{cases}
\end{align*}
We define 
\begin{align*}
	\mathcal{R}(t,x)&:=\sum_{k=1}^{K}\mathcal{R}_k(t,x),
	\\
	\mathcal{R}_k(t,x)&:= Q_{\omega_k,\gamma_k}(x-v_k t -x_k)e^{i\left(\frac{1}{2}v_k x - \frac{1}{4}v_k^2 t +\omega_k t + \theta_k  \right)}.
\end{align*}
We note that if $k_0 \in \llbracket 1, K \rrbracket$ satisfies $v_{k_0} =0$ then $\mathcal{R}_{k_0}(t,x)= Q_{\omega_{k_0},\gamma}(x)e^{i(\omega_{k_0}t+\theta_{k_0})}$. 

The following  is our main theorem. 

\begin{theorem}
\label{thm1.1}
Let $p>5$ and $K\geq 2$. We assume that there exists $k_{0} \in \llbracket 1,K\rrbracket$ such that $v_{k_{0}}=0$ (and additionally that $\omega_{k_{0}}>\gamma^{2}/4$ and $x_{k_{0}} =0$). 
Then, there exist $T_{0}\in \mathbb{R}$, $C,c_{0}>0$ and a solution $u \in C([T_{0},\infty): H^{1}(\mathbb{R}))$ to \eqref{deltaNLS} such that 
\begin{align*}
	\| u(t) -\mathcal{R}(t)\|_{H^{1}} \leq C e^{-c_{0}t}
\end{align*}
for all $t\in [T_{0},\infty)$. 
\end{theorem}

\begin{remark}
We do not need to consider the case $K=1$ here since the unmoving soliton $\mathcal{R}_1$ is a ground state standing wave solution to \eqref{deltaNLS}. 
\end{remark}

The following theorem is also obtained by a similar argument.

\begin{theorem}
\label{thm1.2}
Let $p>5$ and $K\geq 1$. We assume that $v_{k}\neq 0$ for any $k\in \llbracket 1,K\rrbracket$. 
Then, there exist $T_{0}\in \mathbb{R}$, $C,c_{0}>0$ and a solution $u \in C([T_{0},\infty): H^{1}(\mathbb{R}))$ to \eqref{deltaNLS} such that 
\begin{align*}
	\| u(t) -\mathcal{R}(t)\|_{H^{1}} \leq C e^{-c_{0}t}
\end{align*}
for all $t\in [T_{0},\infty)$. 
\end{theorem}

\begin{remark}
Unlike Theorem \ref{thm1.1}, the case $K=1$ with non-zero velocity is of interest. Indeed, since the equation \eqref{deltaNLS} is not Galilean invariant, the moving soliton $\mathcal{R}_1$ is not a solution. Thus, the existence of moving one-soliton is non-trivial. 
\end{remark}

\begin{remark}
Theorem \ref{thm1.2} in the case of $K=1$ shows that the previous global dynamics results in \cite{IkIn17, ArIn22, Inu23p}  is optimal, where it was proved that solutions to \eqref{deltaNLS} with $E_{\gamma}(u_{0})M(u_{0})^{\frac{1-s_{c}}{s_{c}}} \leq E_{0}(Q_{1,0})M(Q_{1,0})^{\frac{1-s_{c}}{s_{c}}} $ either scatter or blow up (possibly grow up) in both time directions, where $s_{c}:=1/2-2/(p-1)$. Theorem \ref{thm1.2} means that there exists non-scattering, global (forward-in-time), and bounded solution whenever $E_{\gamma}(u_{0})M(u_{0})^{\frac{1-s_{c}}{s_{c}}} > E_{0}(Q_{1,0})M(Q_{1,0})^{\frac{1-s_{c}}{s_{c}}}$. 
Moreover, in a similar way, Theorem \ref{thm1.2} with $K=2$ also shows that the scattering and blow-up dichotomy result for even solutions obtained recently in \cite{GuIn23p2} is optimal. 
\end{remark}

\begin{remark}
In the subcritical case $1<p< 5$, we expect that Theorem \ref{thm1.2} can be shown in a similar way to \cite{MaMe06}. However, for multi-solitons with one zero velocity as in Theorem \ref{thm1.1}, we need to pay more attention. This is because the ground state standing wave solution $e^{i\omega t}Q_{\omega,\gamma}$ is unstable for any $1< p < 5$ in $H^1(\mathbb{R})$ (see \cite{LCFFKS08}). Thus we need to control the unstable direction even in the subcritical case unlike the usual NLS. 
However, we expect that Theorem \ref{thm1.1} can be shown by the similar method in the subcritical case except for the degenerate case $3<p<5$ and $\omega=\omega_*(p) \in (\gamma^2/4,\infty)$ satisfying $\partial_\omega \|Q_{\omega,\gamma}\|_{L^2}=0$ (see \cite[Section 8.2]{Oht11} for the instability result). 
%We do not pursue this since it can be treated in the similar to ours.
\end{remark}

\begin{remark}
Combet \cite{Com14} showed the existence of a family of multi-soliton solutions for the usual NLS with $L^2$-supercritical nonlinearity in one dimension.  We can expect that a similar result holds for our equation but we do not develop this in the paper. 
\end{remark}

%%%%%%%%%%%%%%%%%%%%%

\subsection{Idea of proofs}

We use the method developed by C\^{o}te, Martel, and Merle \cite{CMM11}. 

Our main interest is the case that one soliton has zero velocity. In this case, we should consider $Q_{\omega,\gamma}(x)e^{i\omega t +i \theta}$ instead of $Q_{\omega,0}(x -x_1)e^{i\omega_1 t +i \theta_1}$.
To construct the multi-soliton for NLS without potential \eqref{NLS}, the linearized operator around $Q_{\omega,0}$ plays an important role. In the $L^2$-supercritical case, the linearized operator has one unstable eigenfunction. 
On the other hand, for \eqref{deltaNLS}, the linearized operator around $Q_{\omega,\gamma}$ has two unstable eigenfunctions, one of which is even and the other is odd. The other solitons $\mathcal{R}_k$ move away from the origin and thus their interaction with the delta at the origin exponentially decays and we can ignore it. The method of \cite{CMM11} to control the unstable direction is still applicable, taking into account the two unstable directions for the linearized operator around $Q_{\omega,\gamma}$. According to the method, local well-posedness in $H^s$ for some $s<1$ is also important for the compactness argument. We are not aware of a proof that the equation \eqref{deltaNLS} is locally well-posed in $H^s$ for some $s<1$, so we provide one by showing the equivalence of the Sobolev type norm related to $(-\Delta_\gamma + 1)^s$ and the usual fractional Sobolev $H^s$ for $s \in (0,3/2)$. 

%We only consider the case that $v_{k_{0}}=0$ for some $k_{0} \in \llbracket 1, K \rrbracket$. See Appendix  for non-zero velocities' case. 

%%%%%%%%%%%%%%%%%%%%%

%\subsection{Other related studies}

%%%%%%%%%%%%%%%%%%%%%

\subsection{Construction of the paper}

Section \ref{sec2} is devoted to properties of the linearized operator. 
In Section \ref{sec2.1}, we investigate the spectrum of the linearized operator and show that it has two unstable directions when $\gamma<0$. In Section \ref{sec2.2}, we show coercivity results. Section \ref{sec3} is devoted to the main result, Theorem \ref{thm1.1}. In Section \ref{sec3.1}, we give the statement of a uniform backward estimate and show the main theorem by assuming the uniform backward estimate. We give modulation arguments in Section \ref{sec3.2}. In Section \ref{sec3.3}, we give a proof of the uniform backward estimate. In Appendix \ref{appB}, we give a proof of the local well-posedness of \eqref{NLS} in $H^s$ for $s\in (1/2,1)$, which is used in the proof of the main result by a compactness argument. To show it, we give the equivalency of the Sobolev norm related to $(-\Delta_\gamma + 1)^s$ and the usual fractional Sobolev norm with order $s$ for $s \in (0,3/2)$. In Appendix \ref{appC}, we give the modifications required for the proof of Theorem \ref{thm1.2}. 

%%%%%%%%%%%%%%%%%%%%%%%%%%%%%%%%%%%%%%%%%%

\section{Linearized operator}
\label{sec2}

%%%%%%%%%%%%%%%%%%%%%

\subsection{Spectrum of the linearized operator}

\label{sec2.1}

Throughout this section, let $\gamma \leq 0$ and $\omega >\gamma^2/4$. 
We also use the notation
\begin{align*}
	\langle f , g \rangle_{L^2}= \langle f , g \rangle :=\Re \int f\overline{g}dx.
\end{align*}

We consider the operator
\begin{align*}
	\boL_{\om,\gam} f := -iL^+_{\om,\gam} f_1 +  L^-_{\om,\gam} f_2,\qquad f=f_1+if_2 \in \boD(-\Del_{\gam}),
\end{align*}
where
\begin{align*}
	L^+_{\om,\gam} := - \Del_\gam+ \om - p Q_{\om,\gam}^{p-1}, &&
%	\\
	L^-_{\om,\gam} := - \Del_\gam + \om - Q_{\om,\gam}^{p-1}.
\end{align*}
%
%\begin{align*}
%	L^+_{\om,\gam} &:= - \Del_\gam+ \om - p Q_{\om,\gam}^{p-1},
%	\\
%	L^-_{\om,\gam} &:= - \Del_\gam + \om - Q_{\om,\gam}^{p-1}
%\end{align*}
%for $\gamma \leq 0$ and $\omega >\gamma^2/4$. 
%\begin{align*}
%	\boL_{\om,\gam} f := -L^+_{\om,\gam} f_1 + i L^-_{\om,\gam} f_2,\qquad f=f_1+if_2 \in \boD(\Del_{\om,\gam}).
%\end{align*}
%This 
$\boL_{\om,\gam}$ corresponds to the linear part of the equation \eqref{deltaNLS} for perturbations around $Q_{\om,\gam}$. Indeed, 
if we write $u (t,x) = e^{i\om t} \left( Q_{\om,\gam} + v(t,x) \right)$, then \eqref{deltaNLS} is equivalent to
\begin{align}
\label{eq2.1}
\rd_t v = \boL_{\om,\gam} v + O(v^2).
\end{align}
Note that $\boL_{\om,\gam}$ is just $\R$-linear, not $\C$-linear. 
For this reason, we sometimes write $\boL_{\om,\gam}$ in the matrix form
\begin{align*}
	\boL_{\om,\gam} =
	\begin{pmatrix}
	0 & L^-_{\om,\gam} 
	\\
	-L^+_{\om,\gam} & 0
	\end{pmatrix}
	,\qquad \boD(\boL_{\om,\gam}) = \boD(-\Del_{\gam})^2
\end{align*}
by the identification $\C\sim \R^2$. \par
%
%
%\subsubsection{Spectral properties}
The first part of the main claim in this section is the spectral properties of $\boL_{\om,\gam}$ as follows. The proof is given in Section \ref{sec2.1.2}. 
\begin{proposition}
\label{xP2.1}
The following are true. 
%$$
%\si (\boL_{\omega,0}) = \{  \}
%$$
\begin{enumerate}
\item If $\gam=0$, then% $\boL_{\omega,0}$
\begin{align*}
\si (\boL_{\omega,0}) \setminus i\R = \{ \pm \mathfrak{y}_{\om,0} \} 
%\cup \{ is : |s|\ge \om \}
, &&
\mathfrak{y}_{\om,0} >0.
\end{align*}
\item 
If $\gamma<0$, then
\begin{align*}
\si (\boL_{\omega,\gam}) \setminus i\R = \{ \pm \mathfrak{y}_{\om,\gam} ,\pm \mathfrak{z}_{\om,\gam} \} 
%\cup \{ is : |s|\ge \om \}
, &&
\mathfrak{y}_{\om,\gam},\ \mathfrak{z}_{\om,\gam} >0.
\end{align*}
\item For $\gam\le 0$, the eigenspace associated with $\pm \mathfrak{y}_{\om,\gam}$ is one-dimensional, spanned by an even function $Y^\pm_{\om,\gam} \in \boD(-\Del_\gam)$ with 
$\inp{Y^\pm_{\om,\gam}}{Q_{\om,\gam}}_{L^2} = 0$. 
\item For $\gam < 0$, the eigenspace associated with $\pm \mathfrak{z}_{\om,\gam}$ is one-dimensional, spanned by an odd function $Z^\pm_{\om,\gam} \in \boD(-\Del_\gam)$. 
%$\boL_{\omega,\gamma}$ has exactly two positive eigenvalues. (These may coincide.)
\end{enumerate}
\end{proposition}
%
%We identify $\boL_{\om,\gam}$ as the matrix form: 
%\begin{align*}
%	\boL_{\om,\gam} =
%	\begin{pmatrix}
%	0 & L^-_{\om,\gam} 
%	\\
%	-L^+_{\om,\gam} & 0
%	\end{pmatrix}
%	,\qquad \boD(\boL_{\om,\gam}) = \boD(\Del_{\om,\gam})^2,
%\end{align*}
%where now $\boD(\Del_{\om,\gam})$ is restricted to the real-valued functions. Under this identification, \eqref{eq2.1} is written by
%\begin{align*}
%	\begin{pmatrix}
%	\rd_t v_1
%	\\
%	\rd_t v_2
%	\end{pmatrix}
%	= 
%	\begin{pmatrix}
%	0 & 1\\
%	-1 & 0
%	\end{pmatrix}
%	\begin{pmatrix}
%	L^+_{\om,\gam} & 0
%	\\
%	0 & L^-_{\om,\gam}
%	\end{pmatrix}
%	\begin{pmatrix}
%	v_1\\
%	v_2
%	\end{pmatrix}
%	+ O(v^2),
%\end{align*}
%where $v_1=\Re v$ and $v_2=\Im v$. 
%
%
%Moreover, the bilinear form with respect to $L_{\omega,\gamma}^+$ and $L_{\omega,\gamma}^-$ are respectively given by
%\begin{align*}
%	B^+_{\om,\gam} (f,g) &:= \int_{\R} \rd_x f \rd_x g dx + \om \int_{\R} fg dx - \gam f(0)g(0) 
%	-\int_{\R} p Q_{\om,\gam}^{p-1} fg dx,
%	\\
%	B^-_{\om,\gam} (f,g) &:= \int_{\R} \rd_x f \rd_x g dx + \om \int_{\R} fg dx - \gam f(0)g(0) 
%	-\int_{\R} Q_{\om,\gam}^{p-1} fg dx,
%\end{align*}
%for $f,g \in H^1(\mathbb{R})$. 
%
%We use the notation
%\begin{align*}
%	\langle f , g \rangle_{L^2}= \langle f , g \rangle :=\Re \int f\overline{g}dx.
%\end{align*}
%
%
The new assertion of this paper is the case $\gam<0$, while 
the case $\gam=0$ has been already proven by \cite[Corollary 3.1]{Gri88} (see also \cite{Sch06}). 
In this case, the new unstable mode emerges due to the eigenvalues branching from the kernel of $\boL_{\om,0}$.

\subsubsection{Lemmas}%Spectral properties of $L^\pm_{\om,\gam}$}

We only consider $\gam<0$. 
We first recall the spectral properties of $L_{\omega,\gamma}^{\pm}$ obtained in \cite{LCFFKS08}. 
%
%Let us recall that 
%We know the following result for the spectrum of $L_{\omega,\gamma}^{\pm}$. 
%
%
\begin{lemma}[The spectrum of $L^\pm_{\om,\gam}$]
\label{xL2.2}
The following properties hold:
\begin{enumerate}
\item We have
\begin{align*}
	\si (L^-_{\om,\gam}) = \{ 0 \} \cup \Si^- \cup [\om,\infty),\qquad 
	\Si^- \subset [ \rho^- , \om)
	\quad \text{for some } \rho^->0
\end{align*}
where $\Si^-$ consists of %discrete 
point spectrum. %($\Si$ is usually called the set of \textit{internal modes}.) 
Moreover, 
\begin{align*}
	\ker (L^-_{\om,\gam}) = \Span \{ Q_{\om,\gam} \}.
\end{align*}
\item 
We have
\begin{align*}
	&\si (L^+_{\om,\gam}) = \{ \la_0(\gam) ,\la_1(\gam) \} \cup \Si^+(\gam) \cup [\om,\infty),
\quad \la_0(\gam) < \la_1(\gam) <0,
	\\
	&\Si^+(\gam) \subset [ \rho^+ , \om)
	\quad \text{for some } \rho^+>0,
\end{align*}
where $\Si^+$ consists of discrete points. %, and 
%\begin{align*}
%	\la_0(\gam) < \la_1(\gam) <0\qquad & \text{if } \gam<0.
%\end{align*}
Moreover, 
\begin{align*}
	\dim \ker (L^+_{\om,\gam} - \la_0(\gam)) = \dim \ker (L^+_{\om,\gam} - \la_1(\gam)) =1, 
\end{align*}
and the eigenfunction associated with $\la_0(\gam)$ (resp. $\la_1(\gam)$) is even (resp. odd).
\end{enumerate}
\end{lemma}
%
%\begin{proposition}[The spectrum of $L^+_{\om,\gam}$]
%\label{P2.4}
%We have
%\begin{align*}
%	&\si (L^+_{\om,\gam}) = \{ \la_0(\gam) ,\la_1(\gam) \} \cup \Si^+(\gam) \cup [\om,\infty),
%	\\
%	&\Si^+(\gam) \subset [{}^\exists \rho^+ , \om)
%	\quad \text{for some } \rho^+>0
%\end{align*}
%where $\Si^+$ consists of discrete points, and 
%\begin{align*}
%	\la_0(\gam) < \la_1(\gam) <0\qquad & \text{if } \gam<0.
%\end{align*}
%Moreover, 
%\begin{align*}
%	\dim \ker (L^+_{\om,\gam} - \la_0) = \dim \ker (L^+_{\om,\gam} - \la_1) =1. 
%\end{align*}
%Furthermore, the eigenfunction of $\la_0(\gam)$ (resp. $\la_1(\gam)$) is even (resp. odd).
%\end{proposition}
%
We also recall the properties of $Q_{\om,\gam}$.  
%which involves the Vakitov--Kolokolov condition in our case. 
%
%
\begin{lemma}
\label{L2.5}
The following are true.
\begin{enumerate}
\item $L^+_{\om,\gam} Q_{\om,\gam} = - (p-1) Q_{\om,\gam}^p$.

\item $\rd_\om Q_{\om,\gam} \in \boD(-\Del_\gam)_{\even}$, and $L^+_{\om,\gam} \rd_\om Q_{\om,\gam} = -Q_{\om,\gam}$.

\item (Vakitov--Kolokolov condition) 
\begin{align*}
\rd_{\om} \nor{Q_{\om,\gam}}{L^2(\R)}^2 = 
\frac 1{2\om} \left( \frac{1}{p-1} - \frac 14 \right) \nor{Q_{\om,\gam}}{L^2(\R)}^2 
-\frac{1}{p-1} \frac{\rd_\om \al_\om}{\sqrt{\om}} Q_{\om,\gam}^2 (0),
\end{align*}
where 
$
\al_\om := \tanh^{-1} \left( \frac{\gam}{2\sqrt{\om}} \right)
$.
%\begin{align*}
%\al_\om := \tanh^{-1} \left( \frac{\gam}{2\sqrt{\om}} \right),\qquad 
%\rd_\om \al_\om = - \frac{\gam}{4\om^{3/2}} \cosh^2(\al_\om).
%\end{align*}
In particular, if $p>5$ and $\gam\le 0$, then $\rd_{\om} \nor{Q_{\om,\gam}}{L^2(\R)}^2 <0$.
\item If $p>5$ and $\gam \le 0$, then 
\begin{align*}
\inf_{\substack{f \in \boD(-\Del_\gam) \\ f\perp Q_{\om,\gam} ,\ \nor{f}{L^2} =1}} 
\inp{L^+_{\om,\gam} f}{f}_{L^2} < 0.
\end{align*}
\end{enumerate}
\end{lemma}

\begin{proof}
See \cite{FuJe08} for (1) and (3). (2) follows from direct calculations. Substituting for $f$ the normalization of $\phi - \frac{\langle Q_{\omega,\gamma}, \phi \rangle_{L^2}}{ \langle Q_{\omega,\gamma}, \partial_\omega Q_{\omega,\gamma} \rangle_{L^2}} \partial_\omega Q_{\omega,\gamma}$, where $\phi$ is an even eigenfunction of $L_{\omega,\gamma}^+$ or substituting a normalized odd eigenfunction $\psi$ of $L_{\omega,\gamma}^+$, we get (4). 
\end{proof}

\subsubsection{Proof of Proposition \ref{xP2.1}}
\label{sec2.1.2}

%\begin{proof}
We now prove Proposition \ref{xP2.1}. 
We apply the argument of %Chang%, Gustafson, Nakanishi, and Tsai 
\cite{CGNT07} under even/odd restrictions. %The existence of one eigenvalue 
Note that the even case is already investigated in \cite{GuIn22p}. 
For the reader's convenience, however, we give a proof including both cases. 
%In the proof, we write
For a subspace $X\subset L^2(\R)$, we define
$$
X_\perp := 
\{f\in X\ |\ \inp{f}{Q_{\omega,\gamma}}_{L^2} =0  \}. 
$$
We divide the proof into several steps.\medskip\par
\textbf{Step 1.} Consider
\begin{equation}\label{x2.3}
\mu_1 := \inf_{
%\substack{u\in H^1_{\even}\setminus 0 \\ \inp{u}{Q_{\omega,\gamma}}
(H^1_{\even})_\perp \setminus \{ 0\}
} 
\frac{\inp{L^+_{\omega,\gamma} f}{f}_{L^2} }{\inp{(L^-_{\omega,\gamma})^{-1} f}{f}_{L^2} }.
\end{equation}
This is well-defined since Lemma \ref{xL2.2} implies that $L^-_{\om,\gam}$ is bijective mapping from $\boD (-\Del_\gam)_\perp \to L^2_\perp$. 
%
%\textit{Step 1-1.} 
%We first recall $(\boD(\Del_\gamma) , \nor{\cdot}{\boD(\Del_\gamma)})$ is a Banach space with 
%$\nor{f}{\boD(\Del_\gamma)} = \nor{f}{L^2} + \nor{\Del_\gamma f}{L^2}$. (See \cite[Corollary 2.2.8]{CaHa98}.) 
%Let us define
%$$
%L^2_{\perp} := \{f\in L^2(\R)\ |\ \inp{f}{Q_{\omega,\gamma}}_{L^2} =0  \},
%$$
%$$
%\boD(\Del_\gamma)_\perp := 
%\{f\in \boD(\Del_\gamma)\ |\ \inp{f}{Q_{\omega,\gamma}}_{L^2} =0  \}.
%%\qquad \nor{f}{\boD(\Del_\gamma)} := \nor{f}{L^2} + \nor{\Del_\gamma f}{L^2}.
%$$
%We first show $L^-_{\omega,\gamma} : \boD(\Del_\gamma)_\perp \to L^2_{\perp}$ is invertible. 
%Obviously $L^-_{\omega,\gamma}$ is bounded. 
%$L^-_{\omega,\gamma}$ is injective by Proposition \ref{P2.3}. 
%Since $L^-_{\omega,\gamma}$ is Fredholm, $R(L^-_{\omega,\gamma})$ is closed, and thus
%$$
%%L^2(\R) = \Ker L^-_{\omega,\gamma} \oplus 
%R(L^-_{\omega,\gamma}) = (\Ker L^-_{\omega,\gamma})^\perp = L^2_\perp.
%$$
%Hence $L^-_{\omega,\gamma}$ is bijective. By the inverse mapping theorem, it follows that $(L^-_{\omega,\gamma})^{-1}: L^2_\perp \to \boD(\Del_\gamma)_\perp$ is well-defined and bounded.
%\medskip\par
%
%\textit{Step 1-2.} 
We claim that $-\infty < \mu_1 < 0$. First note that for $f\in (H^1_{\even})_\perp$, we have
$$
\inp{(L^-_{\omega,\gamma})^{-1}f}{f}_{L^2} = \inp{g}{L^-_{\omega,\gamma}g}_{L^2} \ge \rho^- \nor{g}{L^2}^2,\qquad g:=(L^-_{\omega,\gamma})^{-1}f.
$$
%$\mu_1$ is well-defined. 
%the denominator is always positive. 
Hence, by Lemma \ref{L2.5} (4), we have $\mu_1<0$. 
We next show $\mu_1>-\infty$. For this, 
let $f\in H^1_\perp$, and write $g= (L^-_{\omega,\gamma})^{-1}f$. 
% \in \boD(\Del_\gamma)_\perp$. 
First note that $f\in \boD((L^-_{\omega,\gamma})^{1/2})$, since 
Friedrichs' second representation theorem implies 
%Since 
$\boD((L^-_{\omega,\gamma})^{1/2}) = H^1(\R)$. 
% by Friedrichs' second representation theorem, we have
Then for $\varepsilon_1>0$, we have
$$
\begin{aligned}
\nor{f}{L^2}^2 &= \inp{f}{L^-_{\omega,\gamma} g}_{L^2} = \inp{(L^-_{\omega,\gamma})^{1/2} f}{(L^-_{\omega,\gamma})^{1/2} g} 
\le \nor{(L^-_{\omega,\gamma})^{1/2} f}{L^2} \nor{(L^-_{\omega,\gamma})^{1/2} g}{L^2} \\
&\le \varepsilon_1 \nor{(L^-_{\omega,\gamma})^{1/2} f}{L^2}^2 + C_{\varepsilon_1} \nor{(L^-_{\omega,\gamma})^{1/2} g}{L^2}^2\\
&= \varepsilon_1 \inp{L^-_{\omega,\gamma}f}{f}_{L^2} + C_{\varepsilon_1} \inp{L^-_{\omega,\gamma}g}{g}_{L^2} \\
&= \varepsilon_1 \inp{L^+_{\omega,\gamma}f}{f}_{L^2} + (p-1)\varepsilon_1 \int_\R Q_{\omega,\gamma}^{p-1} |f|^2 dx
 + C_{\varepsilon_1} \inp{f}{(L^-_{\omega,\gamma})^{-1} f}_{L^2}\\
&\le \varepsilon_1 \inp{L^+_{\omega,\gamma}f}{f}_{L^2} + C_0 (p-1)\varepsilon_1 \nor{f}{L^2}^2
 + C_{\varepsilon_1} \inp{f}{(L^-_{\omega,\gamma})^{-1} f}_{L^2}.
\end{aligned}
$$
%for any $\varepsilon_1>0$. 
Now we take $\varepsilon_1$ such that $C_0 (p-1) \varepsilon_1 = \frac 12$. Then we obtain
\begin{equation}\label{2.8}
0\le \frac{1}{C} \nor{f}{L^2}^2 \le \inp{L^+_{\omega,\gamma}f}{f}_{L^2} + C \inp{f}{(L^-_{\omega,\gamma})^{-1} f}_{L^2},
\end{equation}
which yields
$$
-C \le \frac{\inp{L^+_{\omega,\gamma}f}{f}_{L^2}}{ \inp{(L^-_{\omega,\gamma})^{-1} f}{f}_{L^2} }.
$$
Hence the claim follows. 
%which concludes the boundedness from below.
\medskip\par
\textbf{Step 2.} We next claim that the infimum in \eqref{x2.3} is attained. Let $\{f_n \}$ be a sequence with 
$$
\inp{(L^-_{\omega,\gamma})^{-1} f_n}{f_n}_{L^2} =1 ,\qquad \text{and}\qquad 
\inp{L^+_{\omega,\gamma}f_n}{f_n}_{L^2} \xrightarrow{n\to\infty} \mu_1.
$$
Then by \eqref{2.8}, $\{f_n\}$ is bounded in $L^2$. %by \eqref{2.8}. 
Since $\mu_1<0$, we may assume that $\inp{L^+_{\omega,\gamma}f_n}{f_n}_{L^2}<0$ for all $n$, which yields
$$
\nor{\rd_x f_n}{L^2}^2 - \gamma |f_n(0)|^2 + \omega \nor{f_n}{L^2}^2 < p \int_\R Q_{\omega,\gamma}^{p-1} |f_n|^2 dx \le C \nor{f_n}{L^2}^2.
$$
Thus $\{f_n\}$ is, in fact, bounded in $H^1(\R)$. Therefore, 
there exists $f\in H^1(\R)$ such that $f_n \wto f$ in $H^1(\R)$ by taking subsequence if necessary. 
Clearly $f\in  (H^1_{\even})_\perp$. %Since $f\mapsto f(0)$ is bounded linear functional in $H^1$, we have
Also, we have
$$
f_n(0) \to f(0)\qquad\text{and}\qquad 
\int_\R Q_{\omega,\gamma}^{p-1} |f_n|^2 dx \to \int_\R Q_{\omega,\gamma}^{p-1} |f|^2 dx,
$$
where the latter follows from the Rellich--Kondrachov theorem, after the restriction to a uniformly bounded domain in $n$. Therefore, by lower semi-continuity of norms, we have
$$
\inp{L^+_{\omega,\gamma}f}{f}_{L^2} \le \liminf_{n\to\infty} \inp{L^+_{\omega,\gamma}f_n}{f_n}_{L^2} = \mu_1.
$$
This especially implies $f\neq 0$. 
Moreover, 
\begin{align*}
	0 &\le \inp{(L^-_{\omega,\gamma})^{-1} (f-f_n)}{(f-f_n)}_{L^2} 
	\\
	&= 
	\inp{(L^-_{\omega,\gamma})^{-1} f}{f}_{L^2} -2 \inp{(L^-_{\omega,\gamma})^{-1} f}{f_n}_{L^2} 
	+ \inp{(L^-_{\omega,\gamma})^{-1} f_n}{f_n}_{L^2},
\end{align*}
and taking $\liminf_{n\to\infty}$ gives
$$
\begin{aligned}
0\le - \inp{(L^-_{\omega,\gamma})^{-1} f}{f}_{L^2} + \liminf_{n\to\infty}\inp{(L^-_{\omega,\gamma})^{-1} f_n}{f_n}_{L^2} \\
\Longleftrightarrow \qquad 
\inp{(L^-_{\omega,\gamma})^{-1} f}{f}_{L^2} \le \liminf_{n\to\infty}\inp{(L^-_{\omega,\gamma})^{-1} f_n}{f_n}_{L^2}  = 1.
\end{aligned}
$$
Therefore, since $\mu_1<0$, we get
$$
\inp{L^+_{\omega,\gamma}f}{f}_{L^2} \le \mu_1 \le \mu_1 \inp{(L^-_{\omega,\gamma})^{-1} f}{f}_{L^2}, %\qquad (\text{Note } \mu_1<0),
$$
and thus $f$ attains the minimum.\medskip\par
\textbf{Step 3.} By the Lagrange multiplier theorem, the minimum $\xi_1 \in (H_{\even}^1)_{\perp}\setminus\{0\}$ should satisfy
$$
L^+_{\omega,\gamma} \xi_1 = \al (L^-_{\omega,\gamma})^{-1} \xi_1 + \be Q_{\omega,\gamma}
$$
for some $\al,\be\in\R$. 
Then, taking inner product with $\xi_1$ implies $\al= \mu_1$. 
Letting $%\eta_{\om,\gam}
\mathfrak{y}_{\om,\gam}
:= \sqrt{-\mu_1} >0$, we have
$$
\begin{pmatrix}
0 & L^-_{\omega,\gamma} \\
- L^+_{\omega,\gamma} & 0
\end{pmatrix}
\begin{bmatrix}
\xi_1\\
\pm \mathfrak{y}_{\om,\gam} (L^-_{\omega,\gamma})^{-1} \xi_1 
\mp \mathfrak{y}_{\om,\gam}^{-1} \be Q_{\omega,\gamma}
\end{bmatrix}
=
\pm \mathfrak{y}_{\om,\gam} 
\begin{bmatrix}
\xi_1\\
\pm \mathfrak{y}_{\om,\gam} (L^-_{\omega,\gamma})^{-1} \xi_1 
\mp \mathfrak{y}_{\om,\gam}^{-1} \be Q_{\omega,\gamma}
\end{bmatrix}
.
$$
Thus $\pm \mathfrak{y}_{\om,\gam}$ are eigenvalues, with even eigenfunctions 
$Y^\pm_{\om,\gam} := \xi_1 \pm i (\mathfrak{y}_{\om,\gam} (L^-_{\omega,\gamma})^{-1} \xi_1 - \mathfrak{y}_{\om,\gam}^{-1} \be Q_{\omega,\gamma})$. \medskip\par
\textbf{Step 4.} We next consider
$$
\mu_2 := \inf_{u\in H^1_{\odd}\setminus \{0\}} 
\frac{\inp{L^+_{\omega,\gamma} f}{f}_{L^2} }{\inp{(L^-_{\omega,\gamma})^{-1} f}{f}_{L^2} }.
$$
%Note that for all $f\in H^1_{\odd}$, it holds that $\inp{f}{Q_{\omega,\gamma}}_{L^2}=0$. 
Since $H^1_{\odd}\subset H^1_\perp$, $(L^-_{\omega,\gamma})^{-1}$ is well-defined on $H^1_{\odd}$. 
%Since the second eigenfunction $\la_1(\gamma)$ of $L^+_{\omega,\gamma}$ is negative, and the corresponding eigenfunction is odd, we have $\mu_2< 0$. 
Moreover, we have $\mu_2 <0$ by Lemma \ref{xL2.2} (2). 
Hence, by the same argument 
as above, we have $-\infty < \mu_2 < 0$, and also the infimum is attained by some $\xi_2\in H^1_{\odd}\setminus \{0\}$. By variational argument, we have
$$
L^+_{\omega,\gamma} \xi_2 = \al (L^-_{\omega,\gamma})^{-1} \xi_2.
$$
Taking inner product with $\xi_2$, we have $\al=\mu_2<0$. 
Define $\mathfrak{z}_{\om,\gam} = \sqrt{-\mu_2}$. Then 
$$
\begin{pmatrix}
0 & L^-_{\omega,\gamma} \\
- L^+_{\omega,\gamma} & 0
\end{pmatrix}
\begin{bmatrix}
\xi_2\\
\pm \mathfrak{z}_{\om,\gam} (L^-_{\omega,\gamma})^{-1} \xi_2 
\end{bmatrix}
=
\pm \mathfrak{z}_{\om,\gam} 
\begin{bmatrix}
\xi_2\\
\pm \mathfrak{z}_{\om,\gam}  (L^-_{\omega,\gamma})^{-1} \xi_2
\end{bmatrix}
.
$$
Thus $\pm \mathfrak{z}_{\om,\gam} $ is an eigenvalue, with odd eigenfunctions $Z^\pm_{\om,\gam} 
:= \xi_2 \pm i \mathfrak{z}_{\om,\gam}  (L^-_{\omega,\gamma})^{-1} \xi_2$. \medskip\par
\textbf{Step 5.} Finally, we claim that no other spectrum exist in $\C\setminus i\R$. Since $Y^+_{\om,\gam} \in H^1_{\even}$ and $Z^+_{\om,\gam} \in H^1_{\odd}$, we have 
$\inp{Y^+_{\om,\gam}}{Z^+_{\om,\gam}} = 0$, 
%$$
%%\inp{
%\left\langle
%\begin{bmatrix}
%\xi_1\\
%e_1 (L^-_{\omega,\gamma})^{-1} \xi_1 + e_1^{-1} \be Q_{\omega,\gamma}
%\end{bmatrix}
%,
%\begin{bmatrix}
%\xi_2\\
%e_2 (L^-_{\omega,\gamma})^{-1} \xi_2
%\end{bmatrix}
%\right\rangle
%_{L^2(\R)^2}
%= 0.
%$$
%In particular, 
and thus %the eigenfunctions to $$ above 
$Y^+_{\om,\gam}, Z^+_{\om,\gam}$ are linearly independent. 
On the other hand, by Lemma \ref{xL2.2} and the Grillakis--Shatah--Strauss theorem (see \cite[Theorem 5.8]{GSS90}), the dimension of eigenspace for eigenvalues with positive real part should be no more than $2$. Thus, $\mathfrak{y}_{\om,\gam}, \mathfrak{z}_{\om,\gam}$ are the all such eigenvalues. 
The negative side follows by symmetry. Hence the proof is complete.
%\end{proof}

%\subsection{Application to multi-soliton setting}

%When $\gamma=0$, it is known that only one positive eigenvalue exists. See \cite{}. 
%We denote eigenvalues $\pm \mathfrak{y}_{\omega,0}$, whose eigenfunctions are denoted by $Y_{\omega,\gamma}^{\pm}$. That is, we have
%\begin{align*}
%	\mathcal{L}_{\omega,0}Y_{\omega,0}^{\pm}
%	=\pm \mathfrak{y}_{\omega,0} Y_{\omega,0}^{\pm}.
%\end{align*}
%
%When $\gamma<0$, as shown before, we have two positive eigenvalues. We denote eigenvalues $\pm \mathfrak{y}_{\omega,\gamma}$ and $\pm \mathfrak{z}_{\omega,\gamma}$, whose even and odd eigenfunctions are denoted by $Y_{\omega,\gamma}^{\pm}$ and $Z_{\omega,\gamma}^{\pm}$, respectively. That is, we have
%\begin{align*}
%	\mathcal{L}_{\omega,\gamma}Y_{\omega,\gamma}^{\pm}
%	=\pm \mathfrak{y}_{\omega,\gamma} Y_{\omega,\gamma}^{\pm},
%	\quad
%	\mathcal{L}_{\omega,\gamma}Z_{\omega,\gamma}^{\pm}
%	=\pm \mathfrak{z}_{\omega,\gamma}  Z_{\omega,\gamma}^{\pm}.
%\end{align*}
%
%When we consider the construction of multi-soliton solutions, we use the following notation for simplicity: 
%\begin{align*}
%	&\mathcal{L}_{k}:=\mathcal{L}_{\omega_{k},\gamma_{k}},
%	\quad 
%	\mathfrak{y}_{k}:=\mathfrak{y}_{\omega_{k},\gamma_{k}}, 
%	\quad 
%	\mathfrak{z}_{k}:=\mathfrak{z}_{\omega_{k},\gamma_{k}},
%	\\
%	&Y_{k}^{\pm}(x):=Y_{\omega_{k},\gamma_{k}}^{\pm}(x),
%	\quad 
%	Z_{k}^{\pm}(x):=Z_{\omega_{k},\gamma_{k}}^{\pm}(x).
%\end{align*}

\subsection{Coercivity}
\label{sec2.2}

%\subsubsection{Coercivity from spectral gap}
Next we consider the quadratic form%discuss the coercive subspace of the bilinear form
%$$
%B_{\omega,\gamma} (f,g) = B^+_{\omega,\gamma} (f_1,g_1) + B^-_{\omega,\gamma} (f_2,g_2)
%$$
$$
B_{\om,\gam} (u,w) := 
%\frac 12 
B^+_{\om,\gam} (\Re u,\Re w) + %\frac 12 
B^-_{\om,\gam} (\Im u,\Im w),\quad 
u,w\in H^1(\R),
$$
where for $f,g\in H^1(\R)$, we define
%defined by
%at the beginning. 
\begin{align*}
	B^+_{\om,\gam} (f,g) :=& \inp{L^+_{\om,\gam} f}{g}_{L^2} \\
=&\int_{\R} \rd_x f \rd_x g dx + \om \int_{\R} fg dx - \gam f(0)g(0) 
	-\int_{\R} p Q_{\om,\gam}^{p-1} fg dx,
	\\
	B^-_{\om,\gam} (f,g) :=& 
\inp{L^-_{\om,\gam} f}{g}_{L^2} \\
=& 
\int_{\R} \rd_x f \rd_x g dx + \om \int_{\R} fg dx - \gam f(0)g(0) 
	-\int_{\R} Q_{\om,\gam}^{p-1} fg dx.
\end{align*}
$B_{\om,\gam}$ corresponds to the Hessian of the action functional around $Q_{\om,\gam}$; namely,
\begin{equation*}%\label{x2.2}
E_\gam (u) + \frac \om 2 M (u)= 
E_\gam (Q_{\om,\gam}) + \frac \om 2 M (Q_{\om,\gam}) 
+\frac 12 B_{\om,\gam} (u,u) + o(|u|^2).
\end{equation*}
By this, 
since $E_\gam$ and $M$ are conserved quantities for \eqref{deltaNLS}, 
the coercive subspace of $B_{\om,\gam}$ can be controlled by conserved quantities. 
In this section, we investigate the coercivity of $B_{\om,\gam}$ when $\gam<0$, 
in the same spirit as in \cite{DuRo10,CMM11} where the case $\gam=0$ is considered. \par
The main claim of this section is the following:

\begin{proposition}
\label{xP2.4}
Let $\gamma<0$. Set
$$
G := \Span \{
iQ_{\om,\gam},\ i\Re Y^+_{\om,\gam}, \Im Y^+_{\om,\gam}, 
i\Re Z_{\om,\gam}^+, \Im Z_{\om,\gam}^+
\}.
$$
Then there exists $c>0$ such that
\begin{equation}\label{2.11}
B_{\om,\gam} (f,f) \ge c \nor{f}{H^1}^2,\qquad \forall f\in G^\perp.
\end{equation}
\end{proposition}

% 
%(see Lemma \ref{lem3.12}). 
%, Lemma \ref{lem3.12}. 
%This structure encourages us to investigate the coercivity of $B_{\om,\gam}$. 
%Indeed, 
%since $E_\gam$ and $M$ are conserved quantities for \eqref{deltaNLS}, 
%\eqref{x2.2} gives the control of solutions (see Lemma \ref{lem3.12}).

\subsubsection{Bound of unstable dimension} 
%We confine ourselves to $\gamma<0$.\bigskip\par
%
%We first introduce the eigenfunctions of $L^+_{\omega,\gamma}$ w.r.t. $\la_0(\gamma)$, $\la_1(\gamma)$ as $g_0$, $g_1$, respectively;
%$$
%L^+_{\omega,\gamma} g_0 = \la_0 g_0,\qquad L^+_{\omega,\gamma} g_1 = \la_1 g_1,\qquad (\la_0<\la_1<0).
%$$
We first prepare an auxiliary lemma which gives a bound of the unstable dimension for $B_{\om,\gam}$. To begin with, we write 
the eigenfunctions of $L^+_{\omega,\gamma}$ associated with $\la_0(\gamma)$, $\la_1(\gamma)$ as $g_0$, $g_1$, respectively. 
Then we claim the following.

\begin{lemma}
\label{P2.8}
Set
$$G_1 := \Span \{ iQ_{\omega,\gamma} , \ g_0 ,\ g_1 \}.$$
Then there exists $c>0$ such that
\begin{equation}\label{2.9}
B_{\omega,\gamma} (f,f) \ge c \nor{f}{H^1}^2,\qquad \forall f\in G_1^\perp.
\end{equation}
\end{lemma}

\begin{proof}
%By Propositions \ref{P2.3} and \ref{P2.4}, $L^\pm_{\omega,\gamma}$ enjoys spectral gap 
%between $0$. Thus 
By Lemma \ref{xL2.2}, there exists $c>0$ such that
\begin{equation}\label{2.10}
B_{\omega,\gamma}(f,f) \ge c \nor{f}{L^2}^2 \qquad \forall f\in G_1^\perp.
\end{equation}
To upgrade it into $H^1$-coercivity, 
we %derive \eqref{2.9} 
argue by contradiction. Suppose \eqref{2.9} is not ture, then 
%we have
%$$
%\inf_{f\in G_1^\perp\setminus 0} \frac{B_{\omega,\gamma} (f,f)}{\nor{f}{H^1(\R)}^2} =0.
%$$
%Thus 
there exists $\{f_n\}\subset G_1^\perp$ such that
$$
\nor{f_n}{H^1} =1\qquad \text{and}\qquad B_{\omega,\gamma} (f_n,f_n) \xrightarrow{n\to\infty} 0.
$$
By compactness, taking subsequence if necessary, we have $f_n \wto  f_*$ for some $f_*\in H^1(\R)$. %in $H^1(\R)$. 
By \eqref{2.10}, we have $f_n\to 0$ in $L^2(\R)$, and thus $f_*=0$. 
%Now 
%recalling the definition , we obtain
Hence
%$$
%\begin{aligned}
%&B_{\omega,\gamma}^+ (f_{1n},f_{1n}) %= \inp{L^+_{\omega,\gamma} f_{1n}}{f_{1n}}_{L^2} 
%=
%\nor{\rd_x f_{1n}}{L^2}^2 - \gamma |f_{1n}(0)|^2 + \omega \nor{f_{1n}}{L^2}^2 - p\int_\R Q_{\omega,\gamma}^{p-1} |f_{1n}|^2 dx,\\
%&
%B_{\omega,\gamma}^- (f_{2n},f_{2n}) %= \inp{L^-_{\omega,\gamma} f_{2n}}{f_{2n}}_{L^2} 
%=
%\nor{\rd_x f_{2n}}{L^2}^2 - \gamma |f_{2n}(0)|^2 + \omega \nor{f_{2n}}{L^2}^2 - \int_\R Q_{\omega,\gamma}^{p-1} |f_{2n}|^2 dx,
%\end{aligned}
%$$
%and thus
$$\begin{aligned}
\nor{\rd_x f_n}{L^2}^2 
= &B_{\omega,\gamma} (f_n,f_n)  +\gamma |f_n(0)|^2 -\omega \nor{f_n}{L^2}^2 \\
&\hspace{10pt} +p\int_\R Q_{\omega,\gamma}^{p-1} |f_{1n}|^2 dx 
+\int_\R Q_{\omega,\gamma}^{p-1} |f_{2n}|^2 dx 
\xrightarrow{n\to\infty} 0,
\end{aligned}
$$
%where we used $f_n \wto 0$ in $H^1$ and $f_n\to 0$
%Hence it follows that 
which implies $f_n \to 0$ in $H^1(\R)$. This contradicts $\nor{f_n}{H^1}=1$.
\end{proof}
%\end{myleftbar}

%\subsubsection{Coercivity associated to the eigenfunctions of $\boL_{\omega,\gamma}$}
\subsubsection{Proof of Proposition \ref{xP2.4}}

%By imitating Duyckaerts and Merle \cite{DuMe08} and C\^ote, Martel, and Merle \cite{CMM11}, we introduce the space associated with eigenfunctions of $\boL_{\omega,\gamma}$. We omit the subscription $\omega,\gamma$ in this section. 
%We set $Y^+= Y_1 + i Y_2$, $Z^+ = Z_1 + i Z_2$. Thus we have
%$$
%\begin{pmatrix}
%0 & L^-_{\omega,\gamma}\\
%-L^+_{\omega,\gamma} & 0
%\end{pmatrix}
%\begin{pmatrix}
%Y_1 \\ Y_2
%\end{pmatrix}
%=
%\mathfrak{y}
%\begin{pmatrix}
%Y_1 \\ Y_2
%\end{pmatrix}
%,\qquad
%\begin{pmatrix}
%0 & L^-_{\omega,\gamma}\\
%-L^+_{\omega,\gamma} & 0
%\end{pmatrix}
%\begin{pmatrix}
%Z_1 \\ Z_2
%\end{pmatrix}
%=
%\mathfrak{z}
%\begin{pmatrix}
%Z_1 \\ Z_2
%\end{pmatrix}
%.
%$$
%We recall that $Y_1$, $Y_2$ are even functions, while $Z_1$, $Z_2$ are odd functions. Moreover,
%$$
%\inp{Y_1}{Q_{\omega,\gamma}} =0,\qquad 
%Y_2 = \mathfrak{y}(L_{\omega,\gamma}^-)^{-1} Y_1 + c Q_{\omega,\gamma},\qquad
%Z_2 = \mathfrak{z} (L_{\omega,\gamma}^-)^{-1} Z_1.
%$$
%
%
%\begin{proposition}
%\label{P2.9}
%Let $\gamma<0$. Set
%$$
%G_2 := \Span \{
%iQ_{\om,\gam},\ i\boY_1, \boY_2, i\boZ_1, \boZ_2
%\}.
%$$
%Then there exists $c>0$ such that
%\begin{equation}\label{2.11}
%B_{\om,\gam} (f,f) \ge c \nor{f}{H^1(\R)}^2,\qquad \forall f\in G_2^\perp.
%\end{equation}
%\end{proposition}

%\begin{proof}
We divide the proof into several steps.\medskip\par
\textbf{Step 1.} Let us denote $B[f] := B_{\omega,\gamma} (f,f)$, and $Y_{\om,\gam}^\pm = 
Y_1\pm iY_2$, $Z_{\om,\gam}^\pm = Z_1 \pm iZ_2$. % for the real and imaginary parts. 
We omit subscripts $\omega,\gamma$ and just write $Y^\pm,Z^\pm$.  We first show
\begin{equation}\label{x2.8}
B [f] >0,\qquad \forall f\in G^\perp \setminus \{0\}.
\end{equation}
%
%\textit{Step 1-1.} 
Suppose the contrary. Then there exists $h\in G^\perp\setminus \{0\}$ such that $B[h] \le 0$. 
%When we summarize all the conditions, we have
Then $h$ satisfies
\begin{equation}\label{2.12}
\inp{Q_{\omega,\gamma}}{h_2}_{L^2} = \inp{Y_1}{h_2}_{L^2} = 
\inp{Y_2}{h_1}_{L^2} = \inp{Z_1}{h_2}_{L^2} =
\inp{Z_2}{h_1}_{L^2} = 0,\qquad B[h] \le 0.
\end{equation}
%
%\textit{Step 1-2.} 
Now we claim
\begin{equation}\label{2.13}
\forall f \in E := \Span \{ iQ_{\omega,\gamma} ,\ Y^+,\  Z^+,\ h\} ,\qquad 
B[f] \le 0.
\end{equation}
First note that 
$$
B_{\omega,\gamma} (iQ_{\omega,\gamma} , f) = 0,\qquad \forall f \in H^1(\R).
$$
Moreover, for $f=f_1+if_2\in H^1(\R)$, we have
$$
\begin{aligned}
B_{\omega,\gamma} (Y^+ , f) 
= \inp{L^+_{\omega,\gamma}Y_1}{f_1}_{L^2} +  \inp{L^-_{\omega,\gamma}Y_2}{f_2}_{L^2} 
= -\mathfrak{y}_{\om,\gam} \inp{Y_2}{f_1}_{L^2} + \mathfrak{y}_{\om,\gam} \inp{Y_1}{f_2}_{L^2}, %= 0,
\end{aligned}
$$
$$
\begin{aligned}
B_{\omega,\gamma} (Z^+ , f) 
= \inp{L^+_{\omega,\gamma}Z_1}{f_1}_{L^2} +  \inp{L^-_{\omega,\gamma}Z_2}{f_2}_{L^2} 
= -\mathfrak{z}_{\om,\gam} \inp{Z_2}{f_1}_{L^2} + \mathfrak{z}_{\om,\gam} \inp{Z_1}{f_2}_{L^2}. %= 0,
\end{aligned}
$$
Therefore
$$
\left\{
\begin{aligned}
&B_{\omega,\gamma} (Y^+ , h) =  -\mathfrak{y}_{\om,\gam} \inp{Y_2}{h_1}_{L^2} + 
\mathfrak{y}_{\om,\gam} \inp{Y_1}{h_2}_{L^2} =0,\\
&B_{\omega,\gamma} (Z^+ , h) = -\mathfrak{z}_{\om,\gam} \inp{Z_2}{h_1}_{L^2} + \mathfrak{z}_{\om,\gam} \inp{Z_1}{h_2}_{L^2} =0,\\
&B_{\omega,\gamma} (Y^+ , Z^+) = -\mathfrak{y}_{\om,\gam} \inp{Y_2}{Z_1}_{L^2} + 
\mathfrak{y}_{\om,\gam} \inp{Y_1}{Z_2}_{L^2} =0,
\end{aligned}
\right.
$$
where first two follow from \eqref{2.12}, and the last one follows from
$$
\inp{Y_j}{Z_k}_{L^2} =0 \qquad ({}^\forall j,k=1,2)
$$
by $H^1_{\even}\perp H^1_{\odd}$
%\footnote{
%In fact, $B_{\omega,\gamma} (\boY_+ , \boZ_+)= 0$ can be shown without using $H^1_{\even}\perp H^1_{\odd}$. We can use $B_{\omega,\gamma} (\boY_+ , \boZ_+) = B_{\omega,\gamma} (\boZ_+, \boY_+)$ to conclude $\inp{\boY_1}{\boZ_2}_{L^2} = \inp{\boY_2}{\boZ_1}_{L^2}  =0$.
%}
. Also,
$$
\left\{
\begin{aligned}
&B[Y^+] = -\mathfrak{y}_{\om,\gam} \inp{Y_2}{Y_1}_{L^2} + \mathfrak{y}_{\om,\gam} \inp{Y_1}{Y_2}_{L^2} =0,\\
&B[Z^+] = -\mathfrak{z}_{\om,\gam} \inp{Z_2}{Z_1}_{L^2} + \mathfrak{z}_{\om,\gam} \inp{Z_1}{Z_2}_{L^2} =0.
\end{aligned}
\right.
$$
Hence, combining these with the assumption $B[h]\le 0$, we have \eqref{2.13}.\medskip\par
%$$
%\begin{aligned}
%B_{\omega,\gamma} (\boY_+ , \boZ_+) 
%&= \inp{L^+_{\omega,\gamma}\boZ_1}{h_1}_{L^2} +  \inp{L^-_{\omega,\gamma}\boZ_2}{h_2}_{L^2} \\
%&= -e_0 \inp{\boZ_2}{h_1}_{L^2} + e_0 \inp{\boZ_1}{h_2}_{L^2} = 0,
%\end{aligned}
%$$
%
\textbf{Step 2.} Next we claim $\dim_\R E= 4$, 
which leads to contradiction to Lemma \ref{P2.8}, since $E\cap G_1^\perp$ 
cannot be empty by $\dim_\R G_1 \le 3$. 
It suffices to show that $iQ_{\om,\gam}, Y^+, Z^+ ,h$ are linearly independent. 
Suppose that for some real numbers $\al_1,\al_2,\al_3,\al_4\in\R$, we have
$$
\al_1 iQ_{\omega,\gamma} + \al_2 Y^+ + \al_3 Z^+ + \al_4 h =0. 
$$
First we take the $B_{\om,\gam}$-coupling with $Y^-$. Then
$$
\begin{aligned}
B_{\omega,\gamma} (Y^+, Y^-) 
& 
=  -\mathfrak{y}_{\om,\gam} \inp{Y_2}{Y_1}_{L^2} - \mathfrak{y}_{\om,\gam} \inp{Y_1}{Y_2}_{L^2} 
= -2\mathfrak{y}_{\om,\gam}  \inp{Y_1}{Y_2}_{L^2} \\
%&= 
%-2 \mathfrak{y}_{\om,\gam} \inp{\boY_1}{\mathfrak{y}_{\om,\gam} (L^-_{\omega,\gamma})^{-1} Y_1 + cQ_{\omega,\gamma}}_{L^2} \\
&= 
-2 \mathfrak{y}_{\om,\gam}^2 \inp{Y_1}{ (L^-_{\omega,\gamma})^{-1} Y_1}_{L^2} \neq 0
\end{aligned}
$$
since $Y_1\in H^1_{\even} \setminus 0$ with $Y_1\perp Q_{\omega,\gamma}$. 
Moreover, 
%\footnote{
%Here we need $H^1_{\odd}\perp H^1_{\even}$ to show $B_{\omega,\gamma}(\boZ_+, \boY_-)=0$.
%}
%,
$$
B_{\omega,\gamma} (Z^+, Y^-) 
=  -\mathfrak{z}_{\om,\gam} \inp{Z_2}{Y_1}_{L^2} - \mathfrak{z}_{\om,\gam} \inp{Z_1}{Y_2}_{L^2} =0,
$$
$$
\begin{aligned}
B_{\omega,\gamma} (h, Y^-) &= 
\inp{L^+_{\omega,\gamma} Y_1}{h_1}_{L^2} - \inp{L^-_{\omega,\gamma} Y_2}{h_2}_{L^2} \\
&=
-\mathfrak{y}_{\om,\gam} \inp{Y_2}{h_1}_{L^2} - \mathfrak{y}_{\om,\gam} \inp{Y_1}{h_2}_{L^2} =0.
\end{aligned}
$$
Hence we have $\al_2=0$. Next we consider the coupling with $Z^-$. Then
$$
\begin{aligned}
B_{\omega,\gamma} (Z^+, Z^-) 
&
=  -\mathfrak{z}_{\om,\gam} \inp{Z_2}{Z_1}_{L^2} - \mathfrak{z}_{\om,\gam} \inp{Z_1}{Z_2}_{L^2} 
= -2\mathfrak{z}_{\om,\gam} \inp{Z_1}{Z_2}_{L^2} \\
%&= 
%-2 \mathfrak{z}_{\om,\gam} \inp{Z_1}{e_1 (L^-_{\omega,\gamma})^{-1} Z_1}_{L^2} \\
&= 
-2 \mathfrak{z}_{\om,\gam}^2 \inp{Z_1}{ (L^-_{\omega,\gamma})^{-1} Z_1}_{L^2} \neq 0,
\end{aligned}
$$
$$
\begin{aligned}
B_{\omega,\gamma} (h, Z^-) &= 
\inp{L^+_{\omega,\gamma} Z_1}{h_1}_{L^2} - \inp{L^-_{\omega,\gamma} Z_2}{h_2}_{L^2} \\
&=
-\mathfrak{z}_{\om,\gam} \inp{Z_2}{h_1}_{L^2} - \mathfrak{z}_{\om,\gam} \inp{Z_1}{h_2}_{L^2} =0.
\end{aligned}
$$
Thus we have $\al_3=0$. Inner product with $iQ_{\omega,\gamma}$ in $L^2$ yields 
$\al_1=0$. Since $h\neq 0$, we have $\al_4=0$, which concludes the claim.\medskip\par

\textbf{Step 3.} We next show the coercivity \eqref{2.11}. 
%
%\textit{Step 2-1.} We first claim the $L^2$-coercivity:
%$$
%B_{\omega,\gamma} (f,f) \le {}^\exists c \nor{f}{H^1}^2,\qquad \forall f\in G_2^\perp.
%$$
Suppose the contrary. Since $B[f]>0$ in $G^\perp$ by Step 1, we have
%there exists $\{f_n\}\subset G_2^\perp$ s.t.
$$
\inf_{f\in G^\perp\setminus \{0\}} \frac{B_{\omega,\gamma} (f,f)}{\nor{f}{H^1(\R)}^2} =0.
$$
Thus there exists $\{f_n\}\subset G^\perp$ s.t.
$$
\nor{f_n}{H^1(\R)} =1\qquad \text{and}\qquad B_{\omega,\gamma} (f_n,f_n) \xrightarrow{n\to\infty} 0.
$$
By compactness, taking subsequence if necessary, we have $f_n \wto  f_*$ for some $f_*$ in $H^1(\R)$. 
Then the same argument as in the proof of Lemma \ref{P2.8} yields
%we have
%$$
%f_n(0) \xrightarrow{n\to\infty} f_*(0) ,\qquad \text{and}\qquad 
%\int_{\R} Q_{\omega,\gamma}^{p-1} |f_n|^2 \xrightarrow{n\to\infty} \int_{\R} Q_{\omega,\gamma}^{p-1} |f_*|^2.
%$$
%Thus by weak lower semi-continuity of norms, we have
$$
B_{\omega,\gamma} (f_*,f_*) \le \liminf_{n\to\infty} B_{\omega,\gamma} (f_n,f_n) = 0.
$$
On the other hand, since $G^\perp$ is weakly closed, we have $f_*\in G^\perp$. 
%which implies $B_{\omega,\gamma} (f_*,f_*) \ge 0$ by Step 1. Hence 
Hence \eqref{x2.8} implies $B_{\omega,\gamma} (f_*,f_*)=0$, and thus $f_*=0$. However,
$$
\begin{aligned}
&\hspace{20pt}\nor{\rd_x f_n}{L^2}^2 + \omega \nor{f_n}{L^2}^2 \\
&= B_{\omega,\gamma} (f_n,f_n)  +\gamma |f_n(0)|^2 
+p\int_\R Q_{\omega,\gamma}^{p-1} |f_{1n}|^2 dx
+\int_\R Q_{\omega,\gamma}^{p-1} |f_{2n}|^2 dx
 \xrightarrow{n\to\infty} 0,
\end{aligned}
$$
which contradicts the assumption $\nor{f_n}{H^1(\R)}=1$. 
Hence the proof is complete.
%\end{proof}

\subsubsection{Equivalent formulation of coercivity}

For later application, we derive an equivalent formulation of 
Proposition \ref{xP2.4}. 

\begin{proposition}[Coercivity for $\gamma<0$]
\label{P2.11}
There exist $c>0$ such that the following is true:
\begin{equation}\label{x2.11}
\begin{aligned}
c\nor{f}{H^1}^2 \le B_{\om,\gam} (f,f) + 
c^{-1} &\left(
\inp{f}{iQ_{\om,\gam}}^2 + \inp{f}{iY^+_{\om,\gam}}^2 
+ \inp{f}{iY^-_{\om,\gam}}^2 \right.\\
&\hspace{80pt} \left. + \inp{f}{iZ^+_{\om,\gam}}^2
+ \inp{f}{iZ^-_{\om,\gam}}^2 \right) .
\end{aligned}
\end{equation}
\end{proposition}

\begin{proof}
For simplicity of notation, we write 
$iQ_{\om,\gam}, iY^+_{\om,\gam}, iY^-_{\om,\gam}, iZ^+_{\om,\gam}, iZ^-_{\om,\gam}$ as $g_1,\cdots ,g_5$. 
We argue by contradiction.  
Suppose %Proposition \ref{P2.10}
\eqref{x2.11} is not true. Then for any $n\in\N$, there exists $f_n\in H^1$ such that
$$
\frac{1}{n} \nor{f_n}{H^1}^2 \ge 
B_{\om,\gam} (f_n,f_n) + 
n \sum_{j=1}^5 \inp{f_n}{g_j}_{L^2}^2.
%\inp{f_n}{iQ_{\om,\gam}}^2 + \inp{f_n}{i\boY_1}^2 
%+ \inp{f_n}{\boY_2}^2 + \inp{f_n}{i\boZ_1}^2
%+ \inp{f_n}{\boZ_2}^2 .
$$
By normalization, we may assume $\nor{f_n}{H^1}=1$ for all $n$. Then by taking subsequence, we have $f_n\wto f_*$ for some $f_*\in H^1(\R)$. 
Then we have%the same argument as above yields 
%By the Rellich-Kondrachov theorem, we have
%$$
%f_n(0) \xrightarrow{n\to\infty} f_*(0) ,\qquad \text{and}\qquad 
%\int_{\R} Q_{\om,\gam}^{p-1} |f_n|^2 \xrightarrow{n\to\infty} \int_{\R} Q_{\om,\gam}^{p-1} |f_*|^2.
%$$
%Thus by weak lower semi-continuity of norms, we have
$$
B_{\om,\gam} (f_*,f_*) \le \liminf_{n\to\infty} B_{\om,\gam} (f_n,f_n) \le 
\liminf_{n\to\infty} \frac 1n = 0.
$$
%where the first inequality follows similarly as above. 
On the other hand, noting that $B_{\om,\gam} (f_n,f_n)$ 
is bounded from below, we have
%Moreover, we have
$$
\sum_{j=1}^5 \inp{f_*}{g_j}_{L^2}^2
=
\lim_{n\to\infty} \sum_{j=1}^5 \inp{f_n}{g_j}_{L^2}^2 
\le 
\liminf_{n\to\infty}
\frac 1n \left( \frac 1n  - B_{\om,\gam}(f_n,f_n) \right) = 0.
$$
Hence $f_*\in G^\perp$, and thus $f_*=0$ by Proposition \ref{xP2.4}. 
%Proposition \ref{P2.9}, we have $f_*=0$. 
However, %for all $n$ we have
$$
\begin{aligned}
&\hspace{20pt}\nor{\rd_x f_n}{L^2}^2 + \om \nor{f_n}{L^2}^2 \\
&= B_{\om,\gam} (f_n,f_n)  +\gam |f_n(0)|^2 
+p\int_\R Q_{\om,\gam}^{p-1} |f_{1n}|^2 dx
+\int_\R Q_{\om,\gam}^{p-1} |f_{2n}|^2 dx
 \xrightarrow{n\to\infty} 0,
\end{aligned}
$$
which contradicts that $\nor{f_n}{H^1(\R)}=1$ for all $n$. Hence the proof is complete.
\end{proof}

When $\gamma=0$, we have the following coercivity result. 

\begin{proposition}[Coercivity for $\gamma=0$]
\label{P2.10}
There exists $c>0$ such that 
\begin{align*}
	c\|f\|_{H^1}^2 
	&\leq B_{\omega,0}(f,f)
	+c^{-1} \left( \langle \partial_x Q_{\omega,0} , f \rangle^2
	+\langle i Q_{\omega,0} , f \rangle^2 \right.
	\\
	&\hspace{130pt} \left. +\langle i Y_{\omega,0}^+ , f \rangle^2
	+\langle i Y_{\omega,0}^- , f \rangle^2 \right)
\end{align*}
for any $f \in H^1(\mathbb{R})$. 
\end{proposition}

\begin{proof}
This is already obtained in \cite[(3.6)]{CMM11}. See \cite{DuMe08,DuRo10} for the proof. 
\end{proof}

%\subsection{Notations and Summary}
%
%We will apply the theorems in this section to our setting. 
%For $k\in \llbracket 1,K\rrbracket$, we denote
%%When we consider the construction of multi-soliton solutions, we use the following notation for simplicity: 
%\begin{align*}
%	&\mathcal{L}_{k}:=\mathcal{L}_{\omega_{k},\gamma_{k}},
%	\quad 
%	\mathfrak{y}_{k}:=\mathfrak{y}_{\omega_{k},\gamma_{k}}, 
%	\quad 
%%	\mathfrak{z}_{k}:=\mathfrak{z}_{\omega_{k},\gamma_{k}},
%%	\\
%	Y_{k}^{\pm}(x):=Y_{\omega_{k},\gamma_{k}}^{\pm}(x),
%%	\quad 
%%	Z_{k}^{\pm}(x):=Z_{\omega_{k},\gamma_{k}}^{\pm}(x).
%\end{align*}
%For $k=k_0$, where $\gam_{k}\neq 0$, we set
%\begin{align*}
%\mathfrak{z}_{k}:=\mathfrak{z}_{\omega_{k},\gamma_{k}},\quad
%Z_{k}^{\pm}(x):=Z_{\omega_{k},\gamma_{k}}^{\pm}(x).
%\end{align*}

%%%%%%%%%%%%%%%%%%%%%%%%%%%%%%%%%%%%%%%%

%%%%%%%%%%%%%%%%%%%%%%%%%%%%%%%%%%%%%%%%
\section{Proofs}
\label{sec3}

In what follows, we assume that there exists $k_0 \in \llbracket 1, K \rrbracket$ such that $v_{k_0}=0$. 
We write $Q_k := Q_{\om_k,\gam_k}$. 
We will apply the theorems in Section \ref{sec2} to our setting. 
For $k\in \llbracket 1,K\rrbracket$, we denote
%When we consider the construction of multi-soliton solutions, we use the following notation for simplicity: 
\begin{align*}
	&\mathcal{L}_{k}:=\mathcal{L}_{\omega_{k},\gamma_{k}},
	\quad 
	\mathfrak{y}_{k}:=\mathfrak{y}_{\omega_{k},\gamma_{k}}, 
	\quad 
%	\mathfrak{z}_{k}:=\mathfrak{z}_{\omega_{k},\gamma_{k}},
%	\\
	Y_{k}^{\pm}(x):=Y_{\omega_{k},\gamma_{k}}^{\pm}(x).
%	\quad 
%	Z_{k}^{\pm}(x):=Z_{\omega_{k},\gamma_{k}}^{\pm}(x).
\end{align*}
For $k=k_0$, where $\gam_{k_0}=\gam< 0$, we set
\begin{align*}
\mathfrak{z}_{k}:=\mathfrak{z}_{\omega_{k},\gamma_{k}},\quad
Z_{k}^{\pm}(x):=Z_{\omega_{k},\gamma_{k}}^{\pm}(x).
\end{align*}

%%%%%%%%%%%%%%%%%%%%%
\subsection{Uniform backward estimate}

\label{sec3.1}

In this section, we introduce a uniform backward estimate, which implies Theorem \ref{thm1.1} by a compactness argument as in \cite{CMM11}.  

Let 
\begin{align*}
	X_{k}(t,x):= x-v_{k}t -x_{k},
	\quad 
	\Theta_{k}(t,x):= \frac{1}{2}v_{k}x - \frac{1}{4}|v_{k}|^{2}t +\omega_{k}t + \theta_{k}.
\end{align*}
We define
\begin{align*}
	\mathcal{Y}_{k}^{\pm}(t,x) &:=  Y_{k}^{\pm}(X_{k}(t,x)) e^{i\Theta_{k}(t,x)}
	=Y_{\omega_{k},\gamma_{k}}^{\pm}(x-v_{k}t -x_{k}) e^{i\left(\frac{1}{2}v_{k}x - \frac{1}{4}|v_{k}|^{2}t +\omega_{k}t + \theta_{k}\right)},
	\\
	\mathcal{Z}_{k}^{\pm}(t,x) &:=  Z_{k}^{\pm}(X_{k}(t,x)) e^{i\Theta_{k}(t,x)}
	= Z_{\omega_{k},\gamma_{k}}^{\pm}(x-v_{k}t -x_{k}) e^{i\left(\frac{1}{2}v_{k}x - \frac{1}{4}|v_{k}|^{2}t +\omega_{k}t + \theta_{k}\right)}.
\end{align*}
We let $T_{n}$ be an increasing time sequence with $T_{n} \to \infty$.
We will show the following uniform backward estimate. 

\begin{proposition}[Uniform backward estimate]
\label{ube}
There exist $n_{0} \in \mathbb{N}$, $c_{0}>0$, $T_{0}>0$, $C>0$ (independent of $n$) such that the following holds: for each $n\geq n_{0}$, there exists $\bm{\lambda_{n}}:=(\bm{\alpha_{n}},\bm{\beta_{n}})=(\alpha_{1,n}^{\pm},..., \alpha_{K,n}^{\pm}, \beta_{n}^{\pm})_{\pm} \in \mathbb{R}^{2K+2}$ with $|\bm{\lambda_{n}}|\leq e^{-c_{0}T_{n}}$ and such that the solution $u_{n}$ to 
\begin{align}
\label{eq}
	\begin{cases}
	i \partial_{t} u_{n} + \partial_{x}^{2} u_{n} + \gamma \delta u_{n} +|u_{n}|^{p-1}u_{n}=0,
	\\
	u_{n}(T_{n})=\mathcal{R}(T_{n}) + i \sum_{k\in \llbracket 1,K \rrbracket, \pm} \alpha_{k,n}^{\pm} \mathcal{Y}_{k}^{\pm}(T_{n}) + i \sum_{\pm} \beta_{n}^{\pm} \mathcal{Z}_{k_{0}}^{\pm}(T_{n})
	\end{cases}
\end{align}
is defined on $[T_{0},T_{n}]$, and satisfies 
\begin{equation}\label{eq2.2}
	\|u_{n}(t) - \mathcal{R}(t)\|_{H^{1}} \leq C e^{-c_{0}t}
\end{equation}
for any $t \in [T_{0},T_{n}]$. 
\end{proposition}

\begin{proof}[Proof of Theorem \ref{thm1.1} assuming Proposition \ref{ube}]
The proof is very similar to \cite{CMM11}, so we only give a sketch. 
By the uniform backward estimate, we obtain the following compactness property
\begin{align*}
	\lim_{A \to \infty} \sup_{n \in \mathbb{N}} \int_{|x|\geq A} |u_{n}(T_{0},x)|^{2}dx =0.
\end{align*}
Moreover, it gives us that $\{u_{n}(T_0)\}_{n \in \mathbb{N}}$ is a bounded sequence in $H^1(\mathbb{R})$ and thus we may assume that $u_{n}(T_0)$ converges to a function $\varphi_0$ weakly in $H^1(\mathbb{R})$. 
By the above compactness property, the compact Sobolev embedding $H^{1}(|x| \leq A) \hookrightarrow L^{2}(|x|\leq A)$, and the Sobolev interpolation, it holds that $u_n(T_0)$ converges to $\varphi_0$ strongly in $H^s$ for $0<s<1$. Let $u_*$ be the $H^1$-solution to \eqref{deltaNLS} with $u_*(T_0)=\varphi_0$ and $T_*$ be the forward maximal existence time. By the local well-posedness in $H^s$ for $1/2<s<1$ of \eqref{deltaNLS} (see Appendix \ref{appB}),  $u_n(t) \to u_*(t)$ in $H^s$ and thus $u_n(t) \rightharpoonup u_*(t)$ in $H^1(\mathbb{R})$. By this convergence and \eqref{eq2.2}, we get 
\begin{align*}
	\|u_*(t)\|_{H^1} \lesssim \liminf_{n \to \infty} \|u_n(t) - \mathcal{R}(t)\|_{H^1} + \|\mathcal{R}(t)\|_{H^1} \lesssim 1
\end{align*} 
for $t \in [T_0,T_*)$. This means that $u_*$ is global and \eqref{eq2.2} gives the desired estimate. 
\end{proof}

%%%%%%%%%%%%%%%%%%%%%

\subsection{Modulation}
\label{sec3.2}

%%%%%%%%%%%%%%%%%%%%%

\subsubsection{Time independent modulation}

Given $y_1^0,..., y_{k_0-1}^0,y_{k_0+1}^0,..., y_K^0\in\R$ and $\mu_1^0,...,\mu_K^0\in\R$, 
let 
$$
\boR^0 (x) := \sum_{k=1}^K Q_{k}(x -y_k^0) e^{i\left(
\frac 12 v_kx + \mu_k^0
\right)}
$$
with $y_{k_0}^0 =0$. 
For $\varepsilon>0$, we set 
\begin{align*}
	\mathcal{B}_{\varepsilon}:=\{ f \in H^{1}: \|f-\boR^0\|_{H^{1}} \leq \varepsilon\}.
\end{align*}
%where we recall $R(t,x)=\sum_{k=1}^{K}R_{k}(t,x)$. 

\begin{lemma}[Modulation for function independent of time]\label{Lemma:ModInd}
There exist $\varepsilon_0,L_{1},C_{1}>0$ such that the following holds; if  $\min\{|y_{k}^0-y_{j}^0|: k\neq j\}\geq L_{1}$, $u \in \mathcal{B}_{\varepsilon_0}$, 
then there exists a unique $C^1$ function $(y_{1}(u),...,y_{k_{0}-1}(u),y_{k_{0}+1}(u),...,y_{K}(u);\mu_{1}(u),...,\mu_{K}(u)): \mathcal{B}_{\varepsilon_0} \to \mathbb{R}^{K-1} \times \mathbb{R}^{K}$ %, which is $C^{1}$, 
such that 
%and, for all $k  \in \llbracket 1,K\rrbracket \setminus\{k_{0}\}$, 
\begin{align*}
	&\Re \int_{\mathbb{R}} (\partial_{x}Q_{k})(x-y_k^0 -y_{k})e^{i\left(\frac{1}{2}v_{k}x+\mu_k^0 +\mu_{k}\right)}  \overline{w(x)} dx=0\qquad (k\in \llbracket 1,K\rrbracket \setminus\{k_{0}\}),
	\\
	&\Im \int_{\mathbb{R}}  Q_{k}(x-y_k^0 - y_{k})e^{i\left(\frac{1}{2}v_{k}x+ \mu_k^0 +\mu_{k}\right)} \overline{w(x)} dx=0\qquad 
(k\in \llbracket 1,K\rrbracket \setminus\{k_{0}\}),
	\\
	&\Im \int_{\mathbb{R}}  Q_{k_0}(x)e^{i\left(\frac{1}{2}v_{k_0}x+ \mu_{k_0}^0 +\mu_{k_0}\right)} \overline{w(x)} dx=0,
\end{align*}
where
\begin{align*}
	w(x) :=u(x) - \sum_{k\in \llbracket 1, K \rrbracket \setminus\{ k_{0}\}} Q_{k}(x-y_k^0 -y_{k})e^{i\left(\frac{1}{2}v_{k}x+\mu_k^0+\mu_{k}\right)} - Q_{k_0}(x)e^{i\left(\frac{1}{2}v_{k_0}x+\mu_{k_0}^0+\mu_{k_0}\right)},
\end{align*}
%Moreover, we have%for $\alpha \in (0,\alpha_{1})$, we have
with the property 
\begin{align*}
	\|w\|_{H^{1}} +\sum_{k\in \llbracket 1, K \rrbracket \setminus\{ k_{0}\}} |y_k| + \sum_{k\in \llbracket 1, K \rrbracket} |\mu_{k}| \leq C_{1} \nor{u- \boR^0}{H^1}. 
\end{align*}
\end{lemma}

\begin{proof}
For simplicity of notation, we suppose $k_0=K$. For $\boy = (y_k)_{k=1}^{K-1}$ and $\bmu = (\mu_k)_{k=1}^{K}$, let
$$
w(u,\boy,\bmu) := u - \sum_{k=1}^K Q_k (\cdot -y_k^0-y_k) e^{i\left(\frac 12 v_k x +\mu_k^0+\mu_k \right)}
$$
with identification $y_K=0$. Obviously, we have 
$w(\boR^0,\bm0,\bm0 ) =0.
%,\qquad \boy^0 := (\boy^0_k)_{k=1}^{K-1}, \quad 
%\bmu^0 := (\bmu_k^0)_{k=1}^{K}.
$ 
Define $\Phi:B_{\varepsilon_0}\times \R^{K-1}\times \R^K \to \R^{K-1}\times \R^K$ by
$$
\Phi (u,\boy,\bmu) =
\Phi_{\boy^0,\bmu^0} (u,\boy,\bmu) := \left( (\rho_k^1 (u,\boy,\bmu)_{k=1}^{K-1}, (\rho_k^2 (u,\boy,\bmu)_{k=1}^{K} )
\right)
$$
with $\boy^0 := (\boy^0_k)_{k=1}^{K-1}$, $\bmu^0 := (\mu_k^0)_{k=1}^{K}$ and
$$
\rho_k^1 (u,\boy,\bmu) := 
\Re \int_{\mathbb{R}} (\partial_{x}Q_{k})(x-y_k^0-y_{k})e^{i\left(\frac{1}{2}v_{k}x+\mu_k^0+\mu_{k}\right)}  \overline{w(u,\boy,\bmu)}dx,\qquad k\in \llbracket 1,K-1\rrbracket,
$$
$$
\rho_j^2(u,\boy,\bmu) := 
\Im \int_{\mathbb{R}} Q_j (x -y_j^0 - y_j) e^{i\left(\frac{1}{2}v_{j}x+\mu_j^0+\mu_{j}\right)}  \overline{w(u,\boy,\bmu)}dx, 
\qquad j\in \llbracket 1,K\rrbracket.
$$
First we note $\Phi(\boR^0,\bm{0},\bm{0})=\bm{0}$. By calculation, for $k,k'\in \llbracket 1,K-1\rrbracket$ and $j,j'\in \llbracket 1,K\rrbracket$, we have
$$
\begin{aligned}
\frac{\rd \rho^1_{k}}{\rd y_{k'}} (\boR^0,\bm{0},\bm{0}) 
&=
%\Re \int_{\mathbb{R}} (\partial_{x}Q_{k})(x-y_{k}^0)e^{i\left(\frac{1}{2}v_{k}x+\mu_{k}^0\right)} \overline{\partial_{x}Q_{k'}(x-y_{k'}^0)e^{i\left(\frac{1}{2}v_{k'}x+\mu_{k'}^0\right)}}dx
\Re \int_{\mathbb{R}} (\partial_{x}Q_{k})(x-y_{k}^0)
(\partial_{x}Q_{k'})(x-y_{k'}^0)
e^{i\left(\frac{1}{2} (v_{k} -v_{k'})x+(\mu_{k}^0- \mu_{k'}^0)\right)} dx
%\overline{\partial_{x}Q_{k'}(x-y_{k'}^0)e^{i\left(\frac{1}{2}v_{k'}x+\mu_{k'}^0\right)}}dx
\\
&=
\left\{
\begin{aligned}
&\nor{\rd_x Q_k}{L^2}^2 & \text{if } k=k',\\
&O\left(e^{- \min\{ \sqrt{\omega_{k}},\sqrt{\omega_{k'}}\}|y_{k}^0-y_{k'}^0|}\right) & 
\text{if } k\neq k';
\end{aligned}
\right.
\end{aligned}
$$
%For $j\in \llbracket 1,K-1\rrbracket$ and $k\in \llbracket 1,K\rrbracket$,
$$
\begin{aligned}
\frac{\rd \rho^1_{k}}{\rd \mu_{j'}} (\boR^0,\bm{0},\bm{0}) 
&=
\left\{
\begin{aligned}
&0 & \text{if } k=j'\\
&-\Im \int_{\mathbb{R}} \rd_x Q_{k}(x-y_{k}^0) Q_{j'}(x-y_{j'}^0)e^{i\left(\frac{1}{2}(v_{k}-v_{j'})x+(\mu_{k}^0-\mu_{j'}^0)\right)}dx & \text{if } j\neq k'
\end{aligned}
\right.\\
&=
O\left(e^{- \min\{ \sqrt{\omega_{k}},\sqrt{\omega_{j'}}\}|y_{k}^0-y_{j'}^0|}\right);
\end{aligned}
$$
%For $j\in \llbracket 1,K\rrbracket$ and $k\in \llbracket 1,K-1\rrbracket$, 
$$
\begin{aligned}
\frac{\rd \rho^2_{j}}{\rd y_{k'}} (\boR^0,\bm{0},\bm{0}) 
&=
\left\{
\begin{aligned}
&0 & \text{if } j=k'\\
&\Im \int_{\mathbb{R}} Q_{j}(x-y_{j}^0)\partial_{x}Q_{k'}(x-y_{k'}^0)e^{i\left(\frac{1}{2}(v_{j}-v_{k'})x+(\mu_{j}^0-\mu_{k'}^0)\right)}dx & \text{if } j\neq k'
\end{aligned}
\right.
\\
&= O\left(e^{- \min\{ \sqrt{\omega_{j}},\sqrt{\omega_{k'}}\}|y_{j}^0-y_{k'}^0|}\right);
\end{aligned}
$$
$$
\begin{aligned}
\frac{\rd \rho^2_{j}}{\rd \mu_{j'}} (\boR^0,\bm{0},\bm{0}) 
&=
%\Re \int_{\mathbb{R}} Q_{j}(x-y_{j}^0)e^{i\left(\frac{1}{2}v_{j}x+\mu_{j}^0\right)} \overline{\partial_{x}Q_{j'}(x-y_{j'}^0)e^{i\left(\frac{1}{2}v_{j'}x+\mu_{j'}^0\right)}}dx
\Re \int_{\mathbb{R}} Q_{j}(x-y_{j}^0) 
Q_{j'}(x-y_{j'}^0) 
e^{i\left(\frac{1}{2}(v_{j} -v_{j'})x+(\mu_{j}^0 -\mu_{j'}^0)\right)} dx
%\overline{\partial_{x}Q_{j'}(x-y_{j'}^0)e^{i\left(\frac{1}{2}v_{j'}x+\mu_{j'}^0\right)}}dx
\\
&=
\left\{
\begin{aligned}
&\nor{Q_j}{L^2}^2 & \text{if } j=j',\\
&O\left(e^{- \min\{ \sqrt{\omega_{j}},\sqrt{\omega_{j'}}\}|y_{j}^0-y_{j'}^0|}\right) & 
\text{if } j\neq j'.
\end{aligned}
\right.
\end{aligned}
$$
Hence, if $\min\{|y_{k}^0-y_{j}^0|: k\neq j\}\geq L_{1}$, then
$$
\nab_{\boy,\bmu} \Phi(\boR^0,\bm{0},\bm{0}) = D + O(e^{-C L_1})
$$
for some $C>0$, where
$$
D= 
\begin{pmatrix}
\diag \left( \nor{\rd_x Q_k}{L^2}^2 \right)_{k\in \llbracket 1,K-1\rrbracket} & 0\\
0 & \diag \left( \nor{Q_j}{L^2}^2 \right)_{j\in \llbracket 1,K\rrbracket}
\end{pmatrix}
.
$$
In particular, $\nab_{\boy,\bmu} \Phi%_{\boy^0,\bmu^0}
(\boR^0,\bm{0},\bm{0})$ is invertible if $L_1$ is sufficiently large, and
\begin{equation}\label{mod1}
\nor{(\nab_{\boy,\bmu} \Phi%_{\boy^0,\bmu^0}
 (\boR^0,\bm{0},\bm{0}))^{-1}}{} \le C\quad 
\text{with } C \text{ independent of } \boy^0,\bmu^0.
\end{equation}
Thus, by the implicit function theorem, there exist %an open ball $B_{\alpha_{1}} \in H^{1}(\mathbb{R})$ 
$\varepsilon_0>0$ and a unique $C^{1}$-class function $(\boy (u), \bmu (u)): \mathcal{B}_{\varepsilon_0} \to \mathbb{R}^{K-1} \times \mathbb{R}^{K}$ such that 
%\begin{align*}
$
	(\boy(\boR^0);\bmu(\boR^0))=(\bm{0},\bm{0})
%(x_{1},...,x_{k_{0}-1},x_{k_{0}+1},...,x_{K};\theta_{1},...,\theta_{K})
$ 
%\end{align*}
and $\Phi(u,\boy(u),\bmu(u)) = \bm{0}$ 
%\begin{align*}
%	\rho_{k}^{1}(\bm{\widetilde{x}}(u);\bm{\widetilde{\theta}}(u);u)=\rho_{k}^{2}(\bm{\widetilde{x}}(u);\bm{\widetilde{\theta}}(u);u)=0
%\end{align*}
for all $u \in \mathcal{B}_{\varepsilon_0}$. 
Then $\frac{d\boy}{du}$, $\frac{d\bmu}{du}$ is bounded in $\mathcal{B}_{\varepsilon_0}$, with uniform bound in $\boy^0,\bmu^0$, since we have \eqref{mod1} and since 
$
\nor{\frac{\rd \Phi}{\rd u} (u,\boy,\bmu)}{H^{-1}} 
$ 
is uniformly bounded in $H^1\times \R^{K-1}\times \R^K$. 
Hence, for $\varepsilon>0$ we have
$$
|(\boy(u), \bmu(u)) - (\boy(\boR^0), \bmu(\boR^0))| \le C\varepsilon
$$
for $u\in \mathcal{B}_\varepsilon$. %Therefore the proof is complete.
The rest of the proof is straightforward.
\end{proof}

%%%%%%%%%%%%%%%%%%%%%%%%%%%%%%%%%%%%%%%%%%
\subsubsection{Time dependent modulation}

We will apply the modulation to the solution of \eqref{eq}. By the well-posedness, the solution $u_{n}$ to \eqref{eq} exists for $t$ close to $T_{n}$. 
Moreover, if $\bm{\lambda_{n}}:=(\bm{\alpha_{n}},\bm{\beta_{n}})$ is sufficiently small, then the final data $u_n(T_n)$ satisfies
\begin{align*}
	\|u_n(T_n) - \mathcal{R}(T_n)\|_{H^1} 
	&\leq 
	\|u_n(T_n) - \mathcal{R}(T_n) 
	+ i \sum_{k,\pm}\alpha_{k,n}^\pm \mathcal{Y}_k^\pm (T_n)
	+ i \sum_{\pm}\beta_{n}^\pm \mathcal{Z}_k^\pm (T_n)\|_{H^1} 
	\\
	&\quad +C \sum_{k,\pm}|\alpha_{k,n}^\pm| 
	+ C \sum_{\pm} |\beta_{n}^\pm|
	\leq \frac{\varepsilon_0}{2}.
\end{align*}
and thus Lemma \ref{Lemma:ModInd} can be applied to $u_n(t)$ around $\boR (t)$ for $t$ close to $T_n$.
%it can be modulated around $\mathcal{R}(t)$ for $t$ close to $T_{n}$ since $u_n(t)$ is continuous in $H^1$. 

In what follows we omit the subscript $n$ if there is no risk of confusion. 
%We sometimes omit the subscript $n$ from now on. 
For $t$ such that $\nor{u(t) - \boR (t)}{H^1} < \varepsilon_0$, 
Lemma \ref{Lemma:ModInd} can be applied with $y_k^0 = v_k t + x_k$ and $\mu_k^0 = 
-\frac 14 |v_k|^2 t + \om_k t + \te_k$, which defines parameter functions: 
\begin{align*}
	\bm{y}(t)&:=(y_{1}(t),...,y_{k_{0}-1}(t), y_{k_{0}+1}(t), ...y_{K}(t)) \in \mathbb{R}^{K-1},
	\\
	\bm{\mu}(t)&:=(\mu_{1}(t),...,\mu_{K}(t)) \in \mathbb{R}^{K}.
\end{align*}
We set 
\begin{align*}	
	\widetilde{X}_{k} (t,x)&:= x-v_{k}t -x_{k}-y_{k}(t),
	\\
	\widetilde{\Theta}_{k}(t,x)&:= \frac{1}{2}v_{k} x - \frac{1}{4}|v_{k}|^{2}t +\omega_{k}t + \theta_{k} + \mu_{k}(t), 
\end{align*}
and
\begin{align*}
	\widetilde{\mathcal{R}}_{k}(t,x)&:=
	%\mathcal{R}_{k}(t,x-y_{k}(t))e^{i\mu_{k}(t)}=
	Q_{k}(\widetilde{X}_{k}(t,x))e^{i\widetilde{\Theta}_{k}(t,x)},
	\\
	\widetilde{\mathcal{R}}(t,x) &:= \sum_{k=1}^{K}\widetilde{\mathcal{R}}_{k}(t,x),
	\\
	\widetilde{\mathcal{Y}}_{k}^{\pm}(t,x) &:= 
	%\mathcal{Y}_{k}^{\pm}(t,x-y_{k}(t))e^{i\mu_{k}(t)} 
	Y_{k}^{\pm}(\widetilde{X}_{k}(t,x))e^{i\widetilde{\Theta}_{k}(t,x)},
	\\
	\widetilde{\mathcal{Z}}_{k_{0}}^{\pm}(t,x) &:= \mathcal{Z}_{k_{0}}^{\pm}(t,x)e^{i\mu_{k_{0}}(t)}
	=Z_{k_{0}}^{\pm}(\widetilde{X}_{k_0}(t,x))e^{i\widetilde{\Theta}_{k_0}(t,x)}. 
\end{align*}
Let $w(t):=u(t) - \widetilde{\mathcal{R}}(t)$. Then $w$ satisfies
\begin{align*}
	\Im \int_{\mathbb{R}} \overline{w(t,x)} \widetilde{\mathcal{R}}_{k}(t,x)dx =0
\end{align*}
for any $k \in \llbracket 1, K \rrbracket$ and
\begin{align*}
	\Re \int_{\mathbb{R}} \overline{w(t,x)} D_x \widetilde{\mathcal{R}}_k(t,x)dx 
	= 0
\end{align*}
for any $k \in \llbracket 1, K \rrbracket \setminus \{k_{0}\}$, where we let 
$D_x \widetilde{\mathcal{R}}_k(t,x):=(\partial_{x}Q_{k})(\widetilde{X}_{k}(t,x))e^{i\widetilde{\Theta}_{k}(t,x)}$ (which is different from $\partial_x \widetilde{\mathcal{R}}_k$). 

We define $\bm{a}^{\pm}(t)=(a_{k}^{\pm}(t))_{k=1}^{K}$ by
\begin{align*}
	a_{k}^{\pm}(t):= \Im \int_{\mathbb{R}} \widetilde{\mathcal{Y}}_{k}^{\mp}(t,x)\overline{w(t,x)}dx
\end{align*}
and $b^{\pm}(t)$ by
\begin{align*}
	b^{\pm}(t):= \Im \int_{\mathbb{R}} \widetilde{\mathcal{Z}}_{k_{0}}^{\mp}(t,x)\overline{w(t,x)}dx.
\end{align*}
Let 
\begin{align*}
	\sqrt{c_0}:= \frac{1}{10}\min\{ \sqrt{\omega_{1}}, \cdots, \sqrt{\omega_{K}}, v_{K}-v_{K-1}, \cdots , v_{2}-v_{1},\sqrt{\mathfrak{y}_1},...,\sqrt{\mathfrak{y}_K},\sqrt{ \mathfrak{z}_{k_0}} \},
\end{align*}
where $1/10$ is not essential.

%%%%%%%%%%%%%%%%%%%%%%%%%%%%%%%%%%%%%%%%%%
\subsubsection{Modulated Final Data}

For a normed vector space $X$, we denote the closed ball with center $a$ and radius $r$ by $B_X(a,r)$. If the center is the origin of $X$, then we denote it by $B_X(r)$. We also denote the boundary of $B_X(r)$ by $\mathbb{S}_X(r)$.  

We will use the following modulated final data result to rewrite the problem. 

\begin{proposition}[modulated final data]
\label{P3.3}
If $T_0$ is sufficiently large, the following holds: 
There exists $C>0$, independent of $n$, such that for any $n\in \N$ and for all 
$\bm{\mathfrak{l}}^{+} = (\bm{\mathfrak{a}}^+, \mathfrak{b}^+)\in B_{\R^{K+1}}(e^{-\frac{3}{2} c_0 T_n})$, there exists 
a unique $\bm{\lambda}_n =  (\bm{\alpha}_n, \bm{\beta}_n)$ with $|\bm{\lambda}_n| \le C |\bm{\mathfrak{l}}^{+}|$, which depends continuously on $\bm{\mathfrak{l}}^{+}$, such that 
the modulation $(w(T_n) , \boy(T_n), \bmu(T_n))$ of $u(T_n)$ satisfies
\begin{equation}\label{3.2}
\bm{l}^+(T_n):=(\boa^+ (T_n) , b^+ (T_n)) = \bm{\mathfrak{l}}^{+},\qquad 
\bm{l}^-(T_n):=(\boa^- (T_n) , b^- (T_n)) = \bm{0}.
\end{equation}
\end{proposition}

\begin{proof}
In the proof, we often drop the time variable $t=T_n$ of a space-time function for simplicity. We do not worry about the order of vectors if it is not crucial. 
Also, we may suppose $k_0=K$ without loss of generality. 
We divide the proof into several steps:\medskip\par
\textbf{Step 1}. We first clarify the structure of one-to-one correspondence which we focus on. Consider the maps:
$$
\Phi : \R^{2K}\times \R^{2} \to H^1,\qquad
\Phi(\bal, \bbe) := 
i \sum_{k\in 
\llbracket 1,K \rrbracket ,\pm} \al^\pm_{k,n} \boY_k^\pm (T_n) 
+i \sum_{\pm} \be_n^\pm \boZ_{k_0}^\pm (T_n),
$$
$$
\Psi : \boV \to H^1 \times \R^{K-1} \times \R^K,\qquad
\Psi(f) := (f+ \boR(T_n) - \widetilde{\boR} (T_n) , \psi_1(f) , \psi_2(f)),
$$
$$
\begin{aligned}
&\Xi : H^1\times \R^{K-1} \times \R^{K} \to \R^{2K-2} \times \R^2 \times \R^2,\qquad \\
&\Xi (w,\boy, \bmu) := 
\left(
\left(
\Im \int_\R \ovl{w} \widetilde{\boY}^\pm_{k}
\right)_{k\in \llbracket 1,K-1 \rrbracket,\pm}
,
\left(
\Im \int_\R \ovl{w} \widetilde{\boY}^\pm_{K}
\right)_{\pm}
,
\left(
\Im \int_\R \ovl{w} \widetilde{\boZ}^\pm_{K}
\right)_{\pm}
\right).
\end{aligned}
$$
Here, $\boV$ is a small ball centered at $0$ in $H^1(\mathbb{R})$. The functions $\psi_1$ and $\psi_2$ are defined as follows; applying Lemma \ref{Lemma:ModInd} with $\mathcal{R}^0=\mathcal{R}(T_n)$, then $\psi_1(f):= \bm{y}(f+\mathcal{R}(T_n))$ and $\psi_2(f):= \bm{\mu}(f+\mathcal{R}(T_n))$.
Define $\Upsilon := \Xi \circ \Psi \circ \Phi$. 
Then \eqref{3.2} can be written as
$$
\Upsilon (\bla_n) = (\bm{\mathfrak{l}}^{+},\bm{0}^-).
$$
Thus in order to show Proposition \ref{P3.3}, it suffices to show that $\Upsilon$ is invertible in a neighborhood of $\bm{0}$ with uniform bound in $n$.\medskip\par
\textbf{Step 2}. First note that $\Phi$ is linear, thus $d\Phi = \Phi$. 
Next, the derivative of $\Xi$ is
$$
\begin{aligned}
&d\Xi (w,\boy,\bmu). (h,\boz,\bnu)\\
&= 
\frac{d}{d\eps} 
\left(
\left(
\Im \int_\R
\ovl{(w+\eps h)(x)} \boY_{k}^\pm (x - y_{k} -\eps z_{k}) 
e^{i(\frac 12 v_{k} x + \mu_{k} + \eps \nu_{k})}
dx
\right)_{k,\pm} , \right.
\\
&\hspace{100pt}
\left(
\Im \int_\R
\ovl{(w+\eps h)(x)} \boY_{k_0}^\pm (x) e^{i(\frac 12 v_{k} x + \mu_{k_0} + \eps \nu_{k_0})}
dx
\right)_{\pm},
\\
&\hspace{100pt}
\left.
\left. 
\left(
\Im \int_\R
\ovl{(w+\eps h)(x)} \boZ_{k_0}^\pm (x) e^{i(\frac 12 v_{k} x + \mu_{k_0} + \eps \nu_{k_0})}
dx
\right)_{\pm}
\right)
\right|_{\eps=0}\\
&=
\left(
\left(
\Im \int_\R
\ovl{h} \widetilde{\boY}_k^\pm  dx
-
\Im \int_\R
z_\ka \ovl{w} %\rd_x \boY_k^\pm ( \cdot - y_k) e^{i(\frac 12 v_{k} x + \mu_k)} 
D_x \boY_k^\pm dx
+
\Im \int_\R
i\nu_k 
\ovl{w} \widetilde{\boY}_k^\pm dx
\right)_{k,\pm} , \right.
\\
&\hspace{100pt}
\left(
\Im \int_\R
\ovl{h} \widetilde{\boY}_{k_0}^\pm  
dx
+
\Im \int_\R
i\nu_{k_0} 
\ovl{w} \widetilde{\boY}_{k_0}^\pm dx
\right)_{\pm},\\
&\hspace{100pt}
\left. 
\left(
\Im \int_\R
\ovl{h} \widetilde{\boZ}_{k_0}^\pm 
dx
+
\Im \int_\R
i\nu_{k_0} 
\ovl{w} \widetilde{\boZ}_{k_0}^\pm  dx
\right)_{\pm}
\right)\\
\end{aligned}
$$
\medskip\par
\textbf{Step 3}. Let us calculate the derivative of $\Psi$. 
The explicit definition of $\Psi$ is
$$
\begin{aligned}
\Psi(f) 
&= 
\left(f+ \boR -
\sum_{k=1}^{K-1}
Q_{k} (\cdot - %\tilde{x}_{\ka(k)}
\psi_{1k}(f)) e^{i(\frac 12 v_{k} x + \psi_{2k}(f))}
- 
Q_{k_0} e^{i(\frac 12 v_{k_0} x + \psi_{2k_0}(f))}
%\right.\\
%&\hspace{100pt} 
%\left.
,\psi_1(f), \psi_2(f)
\right).
\end{aligned}
$$
Hence we have
$$
\begin{aligned}
&d \Psi (f).h
=
\left.\frac{d}{d\eps} \Psi(w+\eps h)\right|_{\eps =0}\\
&=
\left(
h+ \sum_{k=1}^{K-1} D_x \widetilde{\mathcal{R}}_k %(\cdot -\phi_{1k}(f))
(d\psi_{1k}(f).h) %e^{i(\frac 12 v_k x + \phi_{2k}(f))}
%\right.\\
%&\hspace{80pt}
%\left.
-
\sum_{k=1}^{K} \widetilde{\mathcal{R}}_k 
i(d\psi_2(f).h)
,\quad 
d\psi_1(f).h,\quad d\psi_2(f).h
\right)\\
&=
\left(
h+ \sum_{k=1}^{K-1} 
D_x \widetilde{\mathcal{R}}_k
%\rd_x Q_{k} (\cdot -\phi_{1k}(f)) 
% e^{i(\frac 12 v_k x + \phi_{2k}(f))}
\Re \int_\R 
\frac{D_x \widetilde{\mathcal{R}}_k \ovl{h}}{\nor{\rd_x Q_{k}}{L^2}^2} 
\, dx 
-
\sum_{k=1}^{K} i \widetilde{\mathcal{R}}_k 
\Im \int_\R 
\frac{ \widetilde{\mathcal{R}}_k \ovl{h}}{\nor{Q_k}{L^2}^2} 
\, dx,
\right.\\
&\hspace{180pt}
\left.
\Re \int_\R 
\frac{D_x \widetilde{\mathcal{R}}_k \ovl{h}}{\nor{\rd_x Q_{k}}{L^2}^2} 
\, dx 
,
\quad 
\Im \int_\R 
\frac{ \widetilde{\mathcal{R}}_k \ovl{h}}{\nor{Q_k}{L^2}^2} 
\, dx 
\right)\\
&\hspace{50pt} +
O(\nor{h}{H^1} (\nor{f}{H^1} + e^{-\si_0 T_n} )).
\end{aligned}
$$

\textbf{Step 4}. 
Let $\bla, \brho\in \R^{2K}\times \R^2$, and assume $\nor{\bla}{} \ll 1$ so that $\Upsilon(\bla)$ is well defined. 
Then we have
$$
d\Upsilon(\bla). \brho =  
d\Xi(\Psi\circ \Phi(\bla)) \circ d\Psi (\Phi(\bla)). \Phi(\brho)
$$
%
%Now we substitute $f=\Phi(\bla)$ and $h= \Phi(\nmu)$ for $\bla,\bmu \in $ 
%Then 
Let us first calculate $d \Psi (\Phi(\bla)). \Phi(\brho)$. 
We first claim that
\begin{equation}\label{3.3}
\sum_{k=1}^{K-1}\left|\Re \int_\R D_x \widetilde{\mathcal{R}}_k \ovl{\Phi(\brho)} dx \right|
+
\sum_{k=1}^{K}\left|\Im \int_\R \widetilde{\mathcal{R}}_k \ovl{\Phi(\brho)} dx \right|
= O(\nor{\brho}{} (\nor{f}{H^1} + e^{-\si_0 T_n}) ) .
\end{equation}
Indeed, if we write $\boy = \psi_1(f)$, $\bmu = \psi_2(f)$, $\brho=((\al_j^\pm)_{j,\pm}, (\be^\pm)_\pm)$, then for $k\in \llbracket 1,K-1 \rrbracket$ we have
$$
\begin{aligned}
&\Re \int_\R D_x \widetilde{\mathcal{R}}_k \ovl{\Phi(\brho)} dx\\
&=
\Re \int_\R 
%\rd_x Q_k (x) e^{i\frac 12 v_k x} 
D_x \mathcal{R}_k
\ovl{\Phi(\brho)} dx
+ O(\nor{\Phi(\brho)}{L^2} (\nor{\boy}{} + \nor{\bmu}{}) ) \\
&=
\Re \int_\R 
%\rd_x Q_k (x) e^{i\frac 12 v_k x } 
D_x \mathcal{R}_k
\ovl{
i \sum_{j\in 
\llbracket 1,K \rrbracket ,\pm} \al^\pm_{j} \boY_j^\pm 
+i \sum_{\pm} \be^\pm \boZ_{K}^\pm 
}dx 
+ O(\nor{\brho}{} \nor{f}{H^1} ) 
\\
%&= 
%-\sum_{\pm} \al_{k}^\pm \Im \int_\R 
%\rd_x Q_k \ovl{\boY_k^\pm} dx
%-
%\sum_{\substack{j\in 
%\llbracket 1,K \rrbracket ,\pm \\ j\neq k}} 
%\al_{j}^\pm \Im \int_\R 
%\rd_x Q_k \ovl{\boY_j^\pm} dx
%-
%\sum_{\pm} \be^\pm \Im \int_\R 
%\rd_x Q_k \ovl{\boZ_K^\pm} dx\\
%&\hspace{200pt} + O(\nor{\brho}{} \nor{f}{H^1} ) \\
&=
O(\nor{\brho}{} (\nor{f}{H^1} + e^{-\si_0 T_n}) ).
\end{aligned}
$$
%where the first term vanishes because of the inner product of odd and even functions, and the second and third terms are $O(e^{-\si_0T_n}\nor{\brho}{})$ by Lemma A.3. 
Similarly, for $k\in \llbracket 1,K\rrbracket$,
$$
\begin{aligned}
&\Im \int_\R \widetilde{\mathcal{R}}_k \ovl{\Phi(\brho)} dx\\
&=
\Im \int_\R 
%\rd_x Q_k (x) e^{i\frac 12 v_k x} 
\mathcal{R}_k
\ovl{\Phi(\brho)} dx
+ O(\nor{\Phi(\brho)}{L^2} (\nor{\boy}{} + \nor{\bmu}{}) ) \\
&=
\Im \int_\R 
% Q_k (x) e^{i\frac 12 v_k x } 
 \mathcal{R}_k
 \ovl{
i \sum_{j\in 
\llbracket 1,K \rrbracket ,\pm} \al^\pm_{j} \boY_j^\pm 
+i \sum_{\pm} \be^\pm \boZ_{K}^\pm 
}dx 
+ O(\nor{\brho}{} \nor{f}{H^1} ) 
\\
%&= 
%-\sum_{\pm} \al_{k}^\pm \Re \int_\R 
%Q_k \ovl{\boY_k^\pm} dx
%-
%\sum_{\substack{j\in 
%\llbracket 1,K \rrbracket ,\pm \\ j\neq k}} 
%\al_{j}^\pm \Re \int_\R 
% Q_k \ovl{\boY_j^\pm} dx
%-
%\sum_{\pm} \be^\pm \Re \int_\R 
%Q_k \ovl{\boZ_K^\pm} dx\\
%&\hspace{200pt} + O(\nor{\brho}{} \nor{f}{H^1} ) \\
&=
O(\nor{\brho}{} (\nor{f}{H^1} + e^{-\si_0 T_n}) ).
\end{aligned}
$$
%where the first term vanishes due to orthogonality, the second term is $O(e^{-\si_0 T_n})$, and the third term is $O(e^{-\si_0 T_n})$ when $k\in \llbracket 1,K-1\rrbracket$, or $0$ when $k=K$. 
Hence \eqref{3.3} holds.

By \eqref{3.3}, we have
$$
\begin{aligned}
d \Psi (\Phi(\bla)). \Phi(\brho) 
=
(\Phi(\rho), \bm{0},\bm{0}) + O(\nor{\brho}{} (\nor{f}{H^1} + e^{-\si_0 T_n})).
\end{aligned}
$$

Thus, if we write $(w,\boy,\bmu) = \Psi\circ \Phi (\bla)$, we obtain
$$
\begin{aligned}
d\Upsilon (\bla). \brho 
&=
\left(
\left(
\Im \int_\R
\ovl{\Phi(\brho)} \widetilde{\boY}_k^\pm  dx
\right)_{k,\pm} , 
%\right.
%\\
%&\hspace{100pt}
\left(
\Im \int_\R
\ovl{\Phi(\brho)} \widetilde{\boY}_{K}^\pm  
dx
\right)_{\pm},
\left(
\Im \int_\R
\ovl{\Phi(\brho)} \widetilde{\boZ}_{K}^\pm 
dx
\right)_{\pm}
\right)\\
&\hspace{150pt} +O(\nor{\brho}{} ( \nor{w}{H^1} +\nor{\boy}{} + \nor{\bmu}{} + e^{-\si_0 T_n}))\\
&=
\left(
\left(
\Im \int_\R
\ovl{\Phi(\brho)} \boY_k^\pm  dx
\right)_{k,\pm} , 
%\right.
%\\
%&\hspace{100pt}
\left(
\Im \int_\R
\ovl{\Phi(\brho)} \boY_{K}^\pm  
dx
\right)_{\pm},
\left(
\Im \int_\R
\ovl{\Phi(\brho)} \boZ_{K}^\pm 
dx
\right)_{\pm}
\right)\\
&\hspace{150pt} +O(\nor{\brho}{} ( \nor{w}{H^1} +\nor{\boy}{} + \nor{\bmu}{} + e^{-\si_0 T_n}))\\
%&=
%\left(
%\left(
%\sum \Re \int_\R
%\ovl{\Phi(\brho)} \boY_k^\pm  dx
%%
%\right)_{k,\pm} , 
%%\right.
%%\\
%%&\hspace{100pt}
%\left(
%\Im \int_\R
%\ovl{\Phi(\brho)} \boY_{K}^\pm  
%dx
%\right)_{\pm},
%\left(
%\Im \int_\R
%\ovl{\Phi(\brho)} \boZ_{K}^\pm 
%dx
%%
%\right)_{\pm}
%\right)\\
&= 
%\begin{pmatrix}
%\mathrm{diag} \left(
%\nor{\boY_k^\pm}{L^2}^2 
%\right)_{k\in \llbracket 1,K \rrbracket , \pm} 
%& 0 \\
%0 & \mathrm{diag} \left(
%\nor{\boZ_K^\pm}{L^2}^2 
%\right)_{\pm} 
%\end{pmatrix}
\left(
\left(
-\Re \int_\R
\ovl{\al^+_k \boY^+_k + \al^-_k \boY^-_k } \boY_k^\pm  dx
\right)_{k,\pm} , 
%\right.
%\\
%&\hspace{100pt}
\left(
-\Re \int_\R
\ovl{\al^+_K \boY^+_K + \al^-_K \boY^-_K} \boY_{K}^\pm  
dx
\right)_{\pm},\right.\\
&\hspace{60pt}\left.
\left(
-\Re \int_\R
\ovl{\be^+ \boY^+_K + \be^- \boY^-_K} \boZ_{K}^\pm 
dx
\right)_{\pm}
\right)
\brho
%\mathrm{Gramm} ( (\boY_k^{\pm_1}, \boZ_K^{\pm_2})_{k\in \llbracket 1,K\rrbracket , \pm_1, \pm_2}) \brho
+O(\nor{\brho}{} ( \nor{\bla}{} + e^{-\si_0 T_n}))
\end{aligned}
$$
where we used 
$$
\sum_{j\neq k, \pm_1,\pm_2} \left|\int_\R  \boY_k^{\pm_1} \boY_j^{\pm_2} dx \right| 
+  \sum_{\substack{k\in \llbracket 1,K-1\rrbracket\\ \pm_1,\pm_2}}
 \left|\int_\R  \boY_k^{\pm_1} \boZ_K^{\pm_2} dx \right|
= O(e^{-\si_0 T_n}), 
$$
$$
\sum_{\pm_1, \pm_2} \Re \int_\R \ovl{\boY_K^{\pm_1}} \boZ_K^{\pm_2} dx =
 0.
$$
Therefore, we have
\begin{equation}\label{3.4}
d\Upsilon (\bla) = 
P + O(\nor{\bla}{} + e^{-\si_0 T_n})
\end{equation}
where
$$
P := 
\begin{pmatrix}
\mathrm{Gramm} (Y^\pm_1)_\pm & 0 &\cdots & 0 & 0\\
0 & \mathrm{Gramm} (Y^\pm_2)_\pm & \cdots & 0 & 0\\
\vdots &\vdots & \ddots & \vdots & \vdots\\
0 & 0 & \cdots & \mathrm{Gramm} (Y^\pm_K)_\pm & 0\\
0 & 0 & \cdots & 0 & \mathrm{Gramm} (Z^\pm_K)_\pm
\end{pmatrix}
,
$$
$$
\mathrm{Gramm} (A^\pm)_\pm := 
\begin{pmatrix}
%\Re\int_{\R}Y^+ \ovl{Y^+} dx & \int_{\R}Y^+ \ovl{Y^+} dx
\inp{A^+}{A^+}_{L^2} & \inp{A^+}{A^-}_{L^2} \\
\inp{A^-}{A^+}_{L^2} & \inp{A^-}{A^-}_{L^2}  
\end{pmatrix}
\quad 
\text{for } A^+, A^-\in L^2(\R).
$$
Since $Y^\pm_k = Y_1 \pm iY_2 $, we have
\begin{equation}\label{3.5}
\det \mathrm{Gramm} (Y^\pm_k)_\pm = 4 \nor{Y_1}{L^2}^2 \nor{Y_2}{L^2}^2 >0.
\end{equation}
The same is true for $(Z^\pm_k)_\pm$. Thus $d\Upsilon$ is invertible if $|\bla| \ll 1$ and $T_0\gg 1$.\medskip\\
\textbf{Step 5}. Let us conclude the proposition. 
By \eqref{3.4} and \eqref{3.5}, 
%$d\Upsilon$ is invertible in a neighborhood of $U$ of $\bm{0}\in \R^{2K+2}$ which is independent of $n$. 
%Moreover, 
%, there is a neighborhood $U$ s.t.
there exists $\del_0, \del_1>0$, independent of $n$, such that 
$$
|\det d\Upsilon (\la) | \ge \del_0 \qquad \text{for } |\bla| \in \del_1.
$$
In particular, $\Upsilon$ is invertible on $B_{\R^{2K+2}} (\del_1)$. Moreover, there is a constant $C$ independent of $n$ such that
$$
\nor{\Upsilon^{-1} (\bm{\mathfrak{l}})}{} \le C \nor{\bm{\mathfrak{l}}}{} \qquad \text{for } \bm{\mathfrak{l}} \in \Upsilon(B_{\R^{2K+2}} (\del_1)).
$$
In particular, if $\bm{\mathfrak{l}}^{+}\in B_{\R^{K+1}}(e^{-\frac 32 c_0  T_n})$, and if $T_0$ is sufficiently large such that $B_{\R^{2K+2}} (e^{-\frac 32 c_0 T_n}) \subset \Upsilon(B_{\R^{2K+2}} (\del_1))$, then $\bla := \Upsilon^{-1}(\bm{\mathfrak{l}}^{+} ,0)$ is a unique element satisfying \eqref{3.2} with $\nor{\bla}{} \le C \nor{\bm{\mathfrak{l}}^{+}}{}$. 
Hence the proof is complete.
\end{proof}

%%%%%%%%%%%%%%%%%%%%%
\subsection{Proof of uniform backward estimate}
\label{sec3.3}

%%%%%%%%%%%%%%%%%%%%%
\subsubsection{Rewriting the problem}

Now our problem is to find the final data at $T_n$, that is, $\mathfrak{a}_{k,n}^\pm$ and $\mathfrak{b}_n^\pm$ satisfying the uniform backward estimate. By the modulated final data, i.e., Proposition \ref{P3.3}, we can rewrite the problem into finding an appropriate $\bm{\mathfrak{l}}^{+}$. To show it, we give the following definition. 

\begin{definition}
Let $n$ be fixed. 
Let $T_{0}$ be determined later (independently of $n$). 
For $\bm{\mathfrak{l}}^{+} \in \mathbb{R}^{K+1}$, we define $T(\bm{\mathfrak{l}}^{+})$ as the infimum of $T \in [T_{0},T_{n}]$ such that  the following properties hold for all $t\in [T,T_{n}]$: 
\begin{enumerate}[(1)]
\item Closeness to $ \mathcal{R}(t)$: 
\begin{align*}
	\|u(t) - \mathcal{R}(t)\|_{H^{1}} \leq \varepsilon_{0}.
\end{align*}

\item Estimates of the modulation parameters: 
\begin{align*}
	&e^{c_{0}t}w(t) \in B_{H^{1}}(1), 
	\\
	&e^{c_{0}t}\bm{y}(t) \in B_{\mathbb{R}^{K-1}}(1), \quad  
	e^{c_{0}t}\bm{\mu}(t) \in B_{\mathbb{R}^{K}}(1),
	\\
	&e^{\frac{3}{2}c_{0}t}\bm{l}^{+}(t) \in B_{\mathbb{R}^{K+1}}(1),  \quad
	e^{\frac{3}{2}c_{0}t}\bm{l}^{-}(t) \in B_{\mathbb{R}^{K+1}}(1).
\end{align*}
\end{enumerate}
\end{definition}

Now we check that the conditions are satisfied when $t=T_n$ if $|\mathfrak{l}^+| \leq e^{-\frac{3}{2}c_0T_n}$ and $T_0 \gg 1$. In fact, we show better estimates except for $\bm{l}^{+}$. 
Let $T_0$ be sufficiently large (independent of $n$). 
As seen before, by taking $\bm{\mathfrak{l}}^{+}$ such that $|\bm{\mathfrak{l}}^{+}|\leq e^{-\frac{3}{2}c_0T_n}(\leq e^{-\frac{3}{2}c_0T_0})$, we have 
\begin{align*}
	\|u_n(T_n) - \mathcal{R}(T_n)\|_{H^1} 
	\leq \sum_{k,\pm}|\alpha_{k,n}^\pm| + \sum_{\pm}| \beta_{n}^\pm|
	\leq C |\bm{\mathfrak{l}}^{+}| 
	\leq C e^{-\frac{3}{2}c_0T_0}\leq \frac{\varepsilon_0}{2}.
\end{align*}
By modulation, we also have
\begin{align*}
	\|w(T_n)\|_{H^1} + |\bm{y}(T_n)| + |\bm{\mu}(T_n)| 
	\leq C |\bm{\mathfrak{l}}^{+}|
	\leq \frac{1}{2}e^{-c_0T_n}.
\end{align*}
Moreover, by the modulated final data, we get $|\bm{l}^+(T_n)|=|\bm{\mathfrak{l}}^{+}| \leq e^{-\frac{3}{2}c_0T_n}$ and $|\bm{l}^-(T_n)|=0\leq \frac{1}{2} e^{-\frac{3}{2}c_0T_n}$. Thus the conditions are satisfied at $t=T_n$. By the continuity of the solution and parameters, $T(\bm{\mathfrak{l}}^{+})$ is well-defined for $\bm{\mathfrak{l}}^{+}$ such that $|\bm{\mathfrak{l}}^{+}|\leq e^{-\frac{3}{2}c_0T_n}$.

To show the uniform backward estimate, i.e., Proposition \ref{ube}, it is enough to prove that for any $n$ there exists $\bm{\mathfrak{l}}^{+} \in B_{\mathbb{R}^{K+1}}(e^{-\frac{3}{2}c_0T_n})$ such that $T(\bm{\mathfrak{l}}^{+})=T_{0}$.

%%%%%%%%%%%%%%%%%%%%%%%%%%%%%%%%%%%%%%%%%%
\subsubsection{Estimates}

\begin{lemma}[Estimates of time derivative of modulated parameters]
For any $t \in [T(\bm{\mathfrak{l}}^{+}),T_{n}]$, the following estimates hold: 
\begin{align*}
	\left| \dot{\bm{y}}(t) \right| + \left| \dot{\bm{\mu}}(t) \right| \leq C\|w(t)\|_{H^{1}} + C e^{-2c_{0}t},
\end{align*}
\end{lemma}

\begin{proof}
Let $w=u-\sum_{k=1}^{K}\widetilde{\mathcal{R}}_{k}$. 
By a simple calculation, $w$ satisfies the following equation: 
\begin{align}
	\notag
	&(i\partial_{t} + \partial_{x}^{2} +\gamma \delta) w +\mathcal{N}(w+\widetilde{\mathcal{R}}) - \sum_{k=1}^{K}\mathcal{N}(\widetilde{\mathcal{R}}_{k})
	\\ \label{eq:w}
	&=i \sum_{k\neq k_{0}} \dot{y_{k}} D_x\widetilde{\mathcal{R}}_{k}
	+\sum_{k=1}^{K} \dot{\mu_{k}} \widetilde{\mathcal{R}}_{k}
	-\gamma \sum_{k\neq k_{0}} \delta \widetilde{\mathcal{R}}_{k}
\end{align}
in the distribution sense. 
Letting $j \in \llbracket 1, K \rrbracket$,
multiplying $\overline{\widetilde{\mathcal{R}}_{j}}$ with \eqref{eq:w}, integrating it, and taking the real part, we obtain
\begin{align}
\label{eq2.8}
	|\dot{\mu_{j}}| 
	\leq C \|w\|_{H^{1}} + C e^{-2c_{0}t}
	+e^{-c_{0}t} \sum_{k=1}^K |\dot{y}_{k}|
	+e^{-c_{0}t} \sum_{k=1}^K |\dot{\mu}_{k}|
%	\leq C \|w\|_{H^{1}} + C |\dot{y}_{j}| + C e^{-2c_{0}t}
%	+e^{-c_{0}t} \sum_{k\neq j}|\dot{y}_{k}|
\end{align}
by using the orthogonality $\Im \int  \overline{w}\widetilde{\mathcal{R}_{j}} dx=0$. 
Letting $j \neq k_{0}$, multiplying $\overline{D_x \widetilde{\mathcal{R}}_{j}}$ with \eqref{eq:w}, integrating it, and taking the imaginary part, we obtain
\begin{align}
\label{eq2.9}
	|\dot{y_{j}}| \leq C \|w\|_{H^{1}} + C e^{-2c_{0}t}
	+e^{-c_{0}t} \sum_{k=1}^K |\dot{\mu}_{k}|
	+e^{-c_{0}t} \sum_{k=1}^K |\dot{y}_{k}|
	%+e^{-c_{0}t} \sum_{k\neq j}|\dot{\mu_{k}}|
\end{align}
by using the orthogonality $\Re \int \overline{w} D_x \widetilde{\mathcal{R}}_{j} dx=0$. 
Combining \eqref{eq2.8} and \eqref{eq2.9}, we get the statement.
\end{proof}

\begin{lemma}[Estimates of time derivatives of $\bm{a}^{\pm}$ and $b^{\pm}$]
\label{lem2.8}
We have the following estimates: 
\begin{align*}
	&\left| \frac{da_{k}^{\pm}}{dt} \pm \mathfrak{y}_{k} a_{k}^{\pm} \right|
	\leq C \|w(t)\|_{H^{1}}^{2} + Ce^{-3c_{0}t},
	\\
	&\left| \frac{d b^{\pm}}{dt} \pm \mathfrak{z}_{k} b^{\pm} \right|
	\leq C \|w(t)\|_{H^{1}}^{2} + Ce^{-3c_{0}t}.
\end{align*}
\end{lemma}

\begin{proof}
This follows from the similar calculation to \cite{CMM11}. We omit the proof. 
\end{proof}

%%%%%%%%%%%%%%%%%%%%%%%%%%%%%%%%%%%%%%%%%%
\subsubsection{Control of unstable directions}

First, we give the following control estimates except for $\bm{l}^\pm$. 

\begin{lemma}
\label{lem2.9}
For large $T_{0}$ (independent of $n$) and for all $\bm{\mathfrak{l}}^{+}\in B_{\mathbb{R}^{K+1}}(e^{-\frac{3}{2}c_{0}T_{n}})$, the following hold for all $t \in [T(\bm{\mathfrak{l}}^{+}),T_{n}]$: 
\begin{enumerate}[(1)]
\item $\|u_{n}(t)-\mathcal{R}(t)\|_{H^{1}} \leq C e^{-c_{0}t} \leq \varepsilon_{0}/2$.
\item 
\begin{align*}
	&e^{c_{0}t}w(t) \in B_{H^{1}}\left(\frac{1}{2}\right), 
	\\
	&e^{c_{0}t}\bm{y}(t) \in B_{\mathbb{R}^{K-1}}\left(\frac{1}{2}\right), \quad  
	e^{c_{0}t}\bm{\mu}(t) \in B_{\mathbb{R}^{K}}\left(\frac{1}{2}\right),
\end{align*}
\end{enumerate}
\end{lemma}

As shown before, the estimates are fulfilled at $t=T_n$ by taking large $T_0$ (independent of $n$). %Lemma \ref{lem2.9} means that the estimates hold for all $t \in [T(\bm{\mathfrak{l}}^{+}),T_{n}]$. 

To show Lemma \ref{lem2.9}, we prepare some notations and lemmas. We will use an energy estimate and the coercivity result to prove Lemma \ref{lem2.9}.

\begin{lemma}[localized virial identity]
\label{lvi}
Let $u$ be a solution to \eqref{deltaNLS}, $f\in C^{1}(\mathbb{R};\mathbb{R})$, and $g\in C^{3}(\mathbb{R};\mathbb{R})$ with $g(x)=0$ near the origin. Then we have
\begin{align*}
	\frac{d}{dt} \int f(x) |u(t,x)|^{2}dx =2\Im \int f'(x) \overline{u(t,x)} \partial_{x}u(t,x) dx
\end{align*} 
and 
\begin{align*}
	\frac{d}{dt} \Im \int g(x) \overline{u(t,x)} \partial_{x}u(t,x) dx
	&=\int g'(x)\left( 2|\partial_{x}u(t,x)|^{2} - \frac{p-1}{p+1} |u(t,x)|^{p+1}\right)dx 
	\\
	&\quad -\frac{1}{2} \int g'''(x) |u(t,x)|^{2}dx. 
\end{align*}
\end{lemma}

\begin{proof}
See \cite{BaVi16}. 
\end{proof}

We define a function $\psi$ by
\begin{align*}
	\psi(x) :=
	\begin{cases}
	0 & (x\leq -1),
	\\
	c' \int_{-1}^{x} e^{-\frac{1}{1-y^{2}}}dy & (-1 <x<1),
	\\
	1 & (x \geq 1),
	\end{cases}
\end{align*}
where $c':= (\int_{-1}^{1} e^{-\frac{1}{1-y^{2}}}dy)^{-1}$. Then the function $\psi$ satisfies that $\psi \in C^{3}(\mathbb{R})$, $0 \leq \psi \leq 1$,  and $\psi'(x) \geq 0$ for $x \in \mathbb{R}$. Moreover, we also find that there exists a constant $C>0$ such that 
\begin{align*}
	(\psi'(x))^{2} \leq C \psi (x) 
	\text{ and } (\psi''(x))^{2} \leq C \psi' (x),
\end{align*}
and it satisfies $\psi(x)+\psi(-x)=1$ for $x\in \mathbb{R}$. 

For $k \in \llbracket 2,K\rrbracket$, we set 
\begin{align*}
	\sigma_{k}:=\frac{1}{2}(v_{k-1}+v_{k}). 
\end{align*}

Since $v_{1}<...<v_{K}$, we have $\sigma_{2}<...<\sigma_{K}$. Moreover, $v_{k-1}<\sigma_{k}<v_{k}$ for any $k \in \llbracket 2,K\rrbracket$. 

For $L>0$, we define
\begin{align*}
	\psi_{k}(t,x):= \psi\left( \frac{x -\sigma_{k}t}{L}\right)-\psi\left( \frac{x -\sigma_{k+1}t}{L}\right)
	\text{ for }  k \in \llbracket 2,K-1 \rrbracket
\end{align*}
and 
\begin{align*}
	\psi_{1}(t,x):=1-\psi\left( \frac{x -\sigma_{2}t}{L}\right),
	\quad 
	\psi_{K}(t,x):=\psi\left( \frac{x -\sigma_{K}t}{L}\right). 
\end{align*}

We set
\begin{align*}
	\mathcal{M}_{k}(t)&=\mathcal{M}_{k}(u(t)):= \int_{\mathbb{R}} \psi_{k}(t,x)|u(t,x)|^{2}dx,
	\\
	\mathcal{P}_{k}(t)&=\mathcal{P}_{k}(u(t)):=\Im \int_{\mathbb{R}} \psi_{k}(t,x) \overline{u(t,x)} \partial_{x}u(t,x) dx.
\end{align*}

\begin{lemma}
\label{lem3.9}
Let $L>0$ be sufficiently large and $T_0$ be chosen larger according to $L$. 
For any $k\in\llbracket 1,K \rrbracket$, we have the following estimates:
\begin{align*}
	\left|\frac{d}{dt} \mathcal{M}_{k}(t)\right|
	\leq \frac{C}{L} \|w\|_{H^{1}}^{2} +  C e^{-3c_{0}t}
\end{align*}
for any $t \in [T(\bm{\mathfrak{l}}^{+}),T_{n}]$
We also have, for any $k\in\llbracket 1,K \rrbracket \setminus\{k_{0}\}$,
\begin{align*}
	\left|\frac{d}{dt} \mathcal{P}_{k}(t)\right|
	\leq  \frac{C}{L} \|w\|_{H^{1}}^{2} +  C e^{-3c_{0}t}
\end{align*}
for any $t \in [T(\bm{\mathfrak{l}}^{+}),T_{n}]$. 
\end{lemma}

\begin{remark}
\label{rmk3.1}
It is worth emphasizing that $\frac{d}{dt}\mathcal{P}_{k_0}(t)$ cannot be estimated as above since the support of $\psi_{k_0}$ includes the origin and the delta potential interacts at the origin. Fortunately, we do not need to estimate it since $\mathcal{P}_{k_0}$ will appear together with $v_{k_0}=0$ such as $v_k \mathcal{P}_{k}$. In the case that $v_k\neq 0$ for any $k$, we need to pay more attention since the support of $\psi_{k}$ includes the origin for some $k$. We give a trick to modify the proof of Lemma \ref{lem3.12} in that case in Appendix \ref{appC}. 
\end{remark}

\begin{proof}[Proof of Lemma \ref{lem3.9}]
This can be shown in the same way as in \cite{MaMe06} and \cite{CMM11} and thus we omit the proof.
\end{proof}

We define a boosted energy $\mathcal{G}$ by
\begin{align*}
	\mathcal{G}(t)=\mathcal{G}(u(t)):=E_{\gamma}(u(t)) + \sum_{k=1}^{K}\left\{ \frac{1}{2}\left( \omega_{k} + \frac{|v_{k}|^{2}}{4}\right)\mathcal{M}_{k}(u(t)) - \frac{v_{k}}{2}  \mathcal{P}_{k}(u(t))\right\}.
\end{align*}

\begin{corollary}
\label{cor3.10}
Let $L$ and $T_0$ be as in Lemma \ref{lem3.9}. 
%Let $L>0$ sufficiently large. 
We have the following estimate:
\begin{align*}
	\left| \frac{d}{dt} \mathcal{G}(t) \right| \leq \frac{C}{L}\|w\|_{H^{1}}^{2}+ C e^{-3c_{0}t}. 
\end{align*}
\end{corollary}

\begin{proof}
This follows from the energy conservation law and Lemma \ref{lem3.9}. 
\end{proof}

For $w\in H^1(\mathbb{R})$, we define 
\begin{align*}
	\mathcal{H}_{\gamma}(w)&:=
	\int_{\mathbb{R}} |\partial_{x}w|^{2}dx  -\gamma |w(0)|^{2}
	\\
	&\quad - \sum_{k=1}^{K} \int_{\mathbb{R}} |\widetilde{\mathcal{R}}_{k}|^{p-1}|w|^{2} + (p-1) |\widetilde{\mathcal{R}}_{k}|^{p-3}\{\Re (\overline{\widetilde{\mathcal{R}}_{k}}w)\}^{2} dx
	\\
	&\quad  + \sum_{k=1}^{K} \left\{ \left( \omega_{k} + \frac{|v_{k}|^{2}}{4} \right) \mathcal{M}_k(w)
	- v_{k} \mathcal{P}_k(w)  \right\}. 
\end{align*}

\begin{lemma}
\label{lem3.11}
Let $L$ and $T_0$ be as in Lemma \ref{lem3.9}. 
For $t\in [T(\bm{\mathfrak{l}}^{+}),T_n]$, we have
\begin{align*}
	\mathcal{G}(u(t)) 
	= \mathcal{G}(\widetilde{\mathcal{R}}(t)) + \frac{1}{2}\mathcal{H}_{\gamma}(w)
	+O(e^{-3c_{0}t})
	%+O\left(\frac{1}{L} e^{-c_{0}t} \|w\|_{H^1}\right)
	+o(\|w\|_{H^1}^{2}),
\end{align*}
where $w= u- \widetilde{\mathcal{R}}$. 
\end{lemma}

\begin{proof}
By simple calculations, we get
\begin{align*}
	E_{\gamma}(u)&= E_{\gamma}(\widetilde{\mathcal{R}})
	 +\Re \int \partial_{x} \widetilde{\mathcal{R}} \overline{\partial_{x}w}dx
	- \gamma \Re \widetilde{\mathcal{R}}(0)\overline{w(0)}
	-\Re \int |\widetilde{\mathcal{R}}|^{p-1}\widetilde{\mathcal{R}}\overline{w}  dx
	\\
	&\quad  +\frac{1}{2} \|\partial_{x}w\|_{L^{2}}^{2}
	- \frac{\gamma}{2} |w(0)|^{2}
	-\int \left[ \frac{p-1}{2}|\widetilde{\mathcal{R}}|^{p-3}\{ \Re(\widetilde{\mathcal{R}}\overline{w})\}^{2} +\frac{1}{2}|\widetilde{\mathcal{R}}|^{p-1}|w|^{2}   \right]dx
	\\
	&\quad +o(\|w\|_{L^{2}}^{2})
\end{align*}
Moreover, we also have
\begin{align*}
	\mathcal{M}_{k}(u)
	= \mathcal{M}_{k}(\widetilde{\mathcal{R}}) + \mathcal{M}_{k}(w)
	+2\Re \int \psi_{k} \widetilde{\mathcal{R}}\overline{w} dx
\end{align*}
and
\begin{align*}
	\mathcal{P}_{k}(u)
	= \mathcal{P}_{k}(\widetilde{\mathcal{R}})+\mathcal{P}_{k}(w)
	+2 \Im \int \psi_{k}\partial_{x}\widetilde{\mathcal{R}}\overline{w}  dx
	+ \Im \int  \partial_{x}\psi_{k} \widetilde{\mathcal{R}}\overline{w} dx. 
\end{align*}
The last term is estimated as follows: 
\begin{align*}
	\left|\Im \int  \partial_{x}\psi_{k} \widetilde{\mathcal{R}}\overline{w} dx\right|
	\leq \frac{C}{L} \int_{\Omega_{k}\cup \Omega_{k+1}} |\widetilde{\mathcal{R}}| |w| dx
	\leq \frac{C}{L} e^{-2c_{0}t} \|w\|_{L^{2}},
\end{align*}
where $\Omega_k:= [\sigma_k t -L, \sigma_k +L]$. 
Since we have $\widetilde{\mathcal{R}}=\sum_{k=1}^{K}\widetilde{\mathcal{R}}_{k}$ and $\widetilde{\mathcal{R}}_{k}$ ($k\neq k_0$) satisfies the equation
\begin{align*}
	-\partial_{x}^{2}\widetilde{\mathcal{R}}_{k}
	+\left( \omega_{k} + \frac{|v_{k}|^{2}}{4} \right)\widetilde{\mathcal{R}}_{k}
	+iv_{k}\partial_{x}\widetilde{\mathcal{R}}_{k}
	-|\widetilde{\mathcal{R}}_{k}|^{p-1}\widetilde{\mathcal{R}}_{k}
	=0,
\end{align*}
and $\widetilde{\mathcal{R}}_{k_0}$ satisfies the equation
\begin{align*}
	-\partial_{x}^{2} \widetilde{\mathcal{R}}_{k_{0}} -\gamma \delta_{0} \widetilde{\mathcal{R}}_{k_{0}}
	+\omega_{k_{0}} \widetilde{\mathcal{R}}_{k_{0}}
	 -|\widetilde{\mathcal{R}}_{k_{0}}|^{p-1}\widetilde{\mathcal{R}}_{k_{0}}
	 =0,
\end{align*}
and $|\widetilde{\mathcal{R}}_k(0)\overline{w}(0)|\lesssim e^{-2c_0t} \|w\|_{H^1}$ for $k \neq k_0$, we obtain the statement. 
\end{proof}

By the coercivity of the linearized operator, we obtain the following coercivity for $\mathcal{H}_{\gamma}$. 

\begin{lemma}
\label{lem3.12}
Let $L$ and $T_0$ be as in Lemma \ref{lem3.9}. 
We have
\begin{align*}
	c\|w\|_{H^{1}}^{2} \leq \mathcal{H}_{\gamma}(w) +O(e^{-3c_0t})
\end{align*}
for $t\in [T(\bm{\mathfrak{l}}^{+}),T_n]$. 
\end{lemma}

\begin{proof}
Since $\sum_k \psi_k =1$, we get 
\begin{align*}
	\mathcal{H}_{\gamma}(w)
	= \sum_{k=1}^{K} \mathcal{I}_k + O(e^{-3c_{0}t}\|w\|_{H^{1}}^{2}),
\end{align*}
where
\begin{align*}
	\mathcal{I}_{k} &:=\int_{\mathbb{R}} \psi_{k} |\partial_{x}w|^{2}dx  - \gamma_{k} |w(0)|^{2}
	\\
	&\quad -  \int_{\mathbb{R}} \psi_{k}\left[ |\widetilde{\mathcal{R}}_{k}|^{p-1}|w|^{2} + (p-1) |\widetilde{\mathcal{R}}_{k}|^{p-3}\{\Re (\overline{\widetilde{\mathcal{R}}_{k}}w\}^{2}\right] dx
	\\
	&\quad  + \left\{ \left( \omega_{k} + \frac{|v_{k}|^{2}}{4} \right)  \int_{\mathbb{R}} \psi_{k} |w|^{2} dx 
	- v_{k} \Im \int_{\mathbb{R}}  \psi_{k}\overline{w} \partial_{x}w dx   \right\}.
\end{align*}
Since we have $\sqrt{\psi_{k}} \partial_{x}w=\partial_{x} (\sqrt{\psi_{k}}w)-\frac{1}{2} \frac{\partial_{x}\psi_{k}}{\sqrt{\psi_{k}}}w$
and $|\partial_{x}\psi_{k}/\sqrt{\psi_{k}}|\leq C/L$ from the definition of $\psi$, we get 
\begin{align*}
	\mathcal{I}_{k}
	\geq \mathcal{J}_{k} -\frac{C}{L} \|w\|_{H^{1}}^{2},
\end{align*}
where we set
\begin{align*}
	\mathcal{J}_{k}&:=\int_{\mathbb{R}}  |\partial_{x}(\sqrt{\psi_{k}}w(t,x))|^{2}dx  - \gamma_{k} |\sqrt{\psi_{k}(0)}w(0)|^{2}
	\\
	&\quad -  \int_{\mathbb{R}} \left[ |\widetilde{\mathcal{R}}_{k}|^{p-1}|\sqrt{\psi_{k}}w|^{2} + (p-1) |\widetilde{\mathcal{R}}_{k}|^{p-3}\{\Re (\overline{\widetilde{\mathcal{R}}_{k}}\sqrt{\psi_{k}}w)\}^{2}\right] dx
	\\
	& \quad + \left\{ \left( \omega_{k} + \frac{|v_{k}|^{2}}{4} \right)  \int_{\mathbb{R}} |\sqrt{\psi_{k}}w|^{2} dx 
	- v_{k} \Im \int_{\mathbb{R}}  \overline{\sqrt{\psi_{k}(t,x)}w(t,x)} \partial_{x}(\sqrt{\psi_{k}}w) dx   \right\}.
\end{align*}
Setting $F_k(t,x):= e^{-i\widetilde{\Theta}_k(t,\chi_k(t,x))}\sqrt{\psi_k(t,\chi_k(t,x))}w(t,\chi_k(t,x))$, where $\chi_k(t,x):=x+v_k t+x_k +y_k(t)$, we get
\begin{align*}
	\mathcal{J}_{k} 
	&= \int_{\mathbb{R}}  |\partial_{x}F_k(t,x)|^{2}dx 
	 - \gamma_{k} |F_k(0)|^{2} +\omega_{k}\int_{\mathbb{R}} |F_k|^{2} dx 
	\\
	&\quad -  \int_{\mathbb{R}} \left[ |Q_{k}(x)|^{p-1}|F_k|^{2} + (p-1) |Q_{k}(x)|^{p-3}\{\Re (Q_{k}(x)F_k)\}^{2}\right] dx
	\\
	&=:B_{k}[F_k]
\end{align*}
By the coercivity results, i.e., Propositions \ref{P2.11} and \ref{P2.10}, we obtain 
\begin{align*}
	c\|F_k\|_{H^1}^2 \leq  B_{k}[F_k] +c^{-1} \text{Ortho}_k,
\end{align*}
where
\begin{align*}
	 \text{Ortho}_{k_0} &= \langle iQ_{k_0} , F_{k_0} \rangle_{L^2}^2
	 +\sum_{\pm} \langle iY_{k_0}^\pm , F_{k_0} \rangle_{L^2}^2
	  +\sum_{\pm} \langle iZ_{k_0}^\pm , F_{k_0} \rangle_{L^2}^2
%	 \left(\Im  \int  Q_{k_0} F_{k_0} dx \right)^2
%	+\sum_{\pm}\left( \Im \int Y_{k_0}^\pm \overline{F_{k_0}} dx \right)^2
%	+\sum_{\pm}\left( \Im \int Z_{k_0}^\pm \overline{F_{k_0}} dx \right)^2
\end{align*}
and 
\begin{align*}
	 \text{Ortho}_{k} &=
	 \langle \partial_x Q_{k}, F_k \rangle_{L^2}^2
	 +\langle iQ_{k}, F_k \rangle_{L^2}^2
	  +\sum_{\pm} \langle iY_{k}^\pm , F_{k} \rangle_{L^2}^2
%	 \left(\Re  \int \partial_x Q_{k} F_k dx \right)^2
%	+\left( \Im \int  Q_{k} F_k dx \right)^2
%	+\sum_{\pm}\left( \Im \int Y_{k}^\pm \overline{F_k} dx \right)^2
\end{align*}
for $k\neq k_0$. By the inverse transformation, i.e., transformation from $F_k$ into $\sqrt{\psi_k}w$, we get 
\begin{align*}
	\|F_k\|_{H^1}^2 \geq c\int \psi_{k} (|\partial_{x}w|^{2}+|w|^{2}) dx - \frac{C}{L^{2}} \int |w|^{2}dx.
\end{align*}
On the other hand, by the  inverse transformation and the orthogonality condition of $w$, we have
\begin{align*}
	\Im  \int  Q_{k} F_k dx 
	=- \Im  \int  \widetilde{\mathcal{R}}_k \sqrt{\psi_k} \overline{w} dx 
	=O(e^{-c_0t}\|w\|_{H^1}).
\end{align*}
By the similar calculation, we also get
\begin{align*}
	\Re  \int \partial_x Q_{k} F_k dx = O(e^{-c_0t}\|w\|_{H^1}).
\end{align*}
Moreover, we have
\begin{align*}
	 \Im \int Y_{k}^\pm \overline{F_{k}} dx
	 = \Im \int \widetilde{\mathcal{Y}}_{k}^\pm \sqrt{\psi_k} \overline{w} dx
	 = a_k^\mp (t) + O(e^{-c_0t}\|w\|_{H^1})
\end{align*}
for any $k$ and 
\begin{align*}
	 \Im \int Z_{k_0}^\pm \overline{F_{k}} dx
	 = b^\mp (t) + O(e^{-c_0t}\|w\|_{H^1}). 
\end{align*}
Combining the above calculations (and recalling $\sum_{k=1}^{K} \psi_{k}=1$), we get
\begin{align*}
	C \mathcal{H}_\gamma(w) 
	\geq \|w\|_{H^1}^2
	 -  \sum_{k\in \llbracket 1,K\rrbracket,\pm} (a_k^\pm (t) )^2 - \sum_{\pm}( b^\pm (t) )^2 
	- \frac{1}{L} \|w\|_{H^1}^2 - O(e^{-2c_0t} \|w\|_{H^1}^2)
\end{align*}
Taking $L$ sufficiently large and large $T_0$ (independent of $n$), 
we obtain the statement since $(a_k^\pm (t) )^2 + ( b^\pm (t) )^2 \leq C e^{-3c_0t}$ for $t \in [T(\bm{\mathfrak{l}}^{+}),T_n]$ by the definition of $T(\bm{\mathfrak{l}}^{+})$. 
\end{proof}

\begin{corollary}\label{cor3.12.0}
Let $L$ and $T_0$ be as in Lemma \ref{lem3.9}. 
For $t \in [T(\bm{\mathfrak{l}}^{+}),T_n]$, we have
\begin{align*}
	c\|w\|_{H^{1}}^{2} \leq \frac{C}{L} \int_{t}^{T_{n}} \|w(s)\|_{H^{1}}^{2}ds + C e^{-3c_{0}t}.
\end{align*}
\end{corollary}

\begin{proof}
By Lemmas \ref{lem3.12} and \ref{lem3.11}, we get
\begin{align*}
	c\|w\|_{H^{1}}^{2} &\leq  \frac{1}{2}\mathcal{H}_{\gamma}(w) + O(e^{-3c_0t})
	\\
	&=\mathcal{G}(u) - \mathcal{G}(\widetilde{\mathcal{R}}) +O(e^{-3c_{0}t})
	%+O\left(\frac{1}{L} e^{-c_{0}t} \|w\|_{H^1}\right)
	+o(\|w\|_{H^1}^{2}).
\end{align*}
Now, by Corollary \ref{cor3.10}, we have
\begin{align*}
	|\mathcal{G}(u(t)) - \mathcal{G}(\widetilde{\mathcal{R}}(t))|
	&\leq 
	|\mathcal{G}(u(t)) -\mathcal{G}(u(T_{n}))| + |\mathcal{G}(u(T_n)) -\mathcal{G}(\widetilde{\mathcal{R}}(t))|
	\\
	&\leq \frac{C}{L} \int_{t}^{T_{n}} \|w(s)\|_{H^{1}}^{2}ds + C e^{-3c_{0}t} + |\mathcal{G}(u(T_n)) -\mathcal{G}(\widetilde{\mathcal{R}}(t))|.
\end{align*}
Since $u(T_n)= \mathcal{R}(T_{n}) +  i \sum_{k\in \llbracket 1,K \rrbracket, \pm} \alpha_{k,n}^{\pm} \mathcal{Y}_{k}^{\pm}(T_{n}) + i \sum_{\pm} \beta_{n}^{\pm} \mathcal{Z}_{k_{0}}^{\pm}(T_{n})$, we get
\begin{align*}
	\mathcal{G}(u(T_n)) 
	= \mathcal{G}(\mathcal{R}(T_{n})) + O(e^{-3c_0t}).
\end{align*}
We have
\begin{align*}
	\mathcal{G}(\mathcal{R}(t))= \mathcal{G}(\widetilde{\mathcal{R}}(t))
	&= \sum_{k=1}^{K}\left\{ E_{\gamma_k}(Q_{k}) +\frac{\omega_k}{2}M(Q_{k}) \right\} +O( e^{-3c_{0}t})
\end{align*}
for any $t \in [T(\bm{\mathfrak{l}}^{+},T_0)]$. Therefore, we get
\begin{align*}
	|\mathcal{G}(u(T_n)) -\mathcal{G}(\widetilde{\mathcal{R}}(t))|
	= O(e^{-3c_0t}). 
\end{align*}
This gives  the statement. 
\end{proof}

We are now ready to prove Lemma \ref{lem2.9}. 

\begin{proof}[Proof of Lemma \ref{lem2.9}]
Taking large $L>0$ and $T_{0}>0$, we get
\begin{align*}
	\|w\|_{H^{1}}^{2}
	&\leq \frac{C}{L} \int_{t}^{T_{n}} e^{-2c_{0}s}ds + C e^{-3c_{0}t}
	\\
	&\leq  \frac{C}{L} e^{-2c_{0}t} + C e^{-3c_{0}t}
	\\
	&\leq  \frac{1}{4} e^{-2c_{0}t}
\end{align*}
by Corollary \ref{cor3.12.0} and the boot strap assumption. 
Thus we get
\begin{align*}
	\|w\|_{H^{1}} \leq \frac{1}{2}e^{-c_{0}t}. 
\end{align*}
Moreover, we have
\begin{align*}
	|\bm{y}(t)|
	&\leq |\bm{y}(T_{n})-\bm{y}(t)|+|\bm{y}(T_{n})|
	\\
	&\leq \frac{C}{L} \int_{t}^{T_{n}} e^{-c_{0}s}ds + C e^{-2c_{0}t}
	+ |\bm{y}(T_{n})|
	\\
	&\leq \frac{C}{L}e^{-c_{0}t} + C e^{-2c_{0}t}
	+ |\bm{y}(T_{n})|.
\end{align*}
Now, $|\bm{y}(T_{n})| \leq C|\bm{\mathfrak{l}}^{+}|\leq C e^{-\frac{3}{2}c_0t}$. Thus we get
\begin{align*}
	|\bm{y}(t)| \leq \frac{C}{L}e^{-c_{0}t} + C e^{-2c_{0}t}+C e^{-\frac{3}{2}c_0t}
	\leq \frac{1}{2} e^{-c_{0}t}
\end{align*}
for large $L$ and $T_0$. We can get the estimate for $\bm{\mu}$ in the same way. Moreover, 
\begin{align*}
	\|u(t)-\mathcal{R}(t)\|_{H^1} 
	&\leq \|\widetilde{\mathcal{R}}(t)-\mathcal{R}(t)\|_{H^1} +\|w(t)\|_{H^1}
	\\
	&\leq C|\bm{y}(t)| +C |\bm{\mu}(t)| + \|w(t)\|_{H^1} +O(e^{-2c_0t})
	\\
	&\leq Ce^{-c_0t} \leq \varepsilon_0/2,
\end{align*} 
where we used $Q_k \in H^2(\mathbb{R})$ for $k \neq k_0$ (indeed, $\|f(\cdot -y)-f\|_{H^1} \leq |y|\|f\|_{H^2}$).  We note that we do not use any such smoothness for $k =k_0$ (indeed $Q_{k_0} \not\in H^2(\mathbb{R})$) since it does not have any translation parameter, i.e., $\widetilde{X}_{k_0} =X_{k_0}=x$ and $\widetilde{\Theta}_{k_0}-\mu_{k_0}(t)=\Theta_{k_0}=\omega_{k_0}t +\theta_{k_0}$.
\end{proof}

We can control $\bm{l}^{-}(t)$ as follows. 

\begin{lemma}[Control of $\bm{l}^{-}(t)$]
\label{lem3.14}
For large $T_{0}$ (independent of $n$) and for all $\bm{\mathfrak{l}}^{+}\in B_{\mathbb{R}^{K+1}}(e^{-\frac{3}{2}c_{0}T_{n}})$, $e^{\frac{3}{2}c_{0}t}\bm{l}^{-}(t) \in B_{\mathbb{R}^{K+1}}(1/2)$ holds for all $t \in [T(\bm{\mathfrak{l}}^{+}),T_{n}]$. 
\end{lemma}

\begin{proof}
The proof is similar to \cite[Lemma 5]{CMM11}. Thus we omit it. 
\end{proof}

The following is sufficient to show the uniform backward estimate. 

\begin{lemma}[Control of $\bm{l}^{+}(t)$]
For large $T_{0}$ (independent of $n$), there exists $\bm{\mathfrak{l}}^{+}\in B_{\mathbb{R}^{K+1}}(e^{-\frac{3}{2}c_{0}T_{n}})$ such that $T(\bm{\mathfrak{l}}^{+})=T_{0}$. 
\end{lemma}

\begin{proof}
The proof is similar to \cite[Lemma 6]{CMM11}, so we omit it. 
\end{proof}

%%%%%%%%%%%%%%%%%%%%%%%%%%%%%%%%%%%%%%%%%
\appendix

%%%%%%%%%%%%%%%%%%%%%%%%%%%%%%%%%%%%%%%%%%

\section{$H^{s}$-solvability}
\label{appB}

In this section, we prove that \eqref{deltaNLS} is locally well-posed in $H^s(\mathbb{R})$ for $s\in (1/2,1)$. The proof is based on the equivalency between the spaces $H^s$ and $\mathcal{D}((-\Delta_\gamma+\lambda)^{s/2})$. This fact may be well-known, but we are not aware of a proof and so we give one for the reader's convenience.
Our proof works even when $\gamma>0$, that is, the potential is attractive. Thus, we consider $\gamma \in \mathbb{R}\setminus\{0\}$ in this section.

\subsection{Equivalency of $\mathcal{D}((-\Delta_\gamma+\lambda)^{s/2})$ and $H^s$}

\subsubsection{Statement}

Let $\lambda >0$ when $\gamma<0$ and $\lambda>\gamma^{2}/4$ when $\gamma>0$. 

In this subsection, we will show the following equivalency between $\mathcal{D}((-\Delta_\gamma+\lambda)^{s/2})$ and the usual fractional Sobolev space $H^s(\mathbb{R})$ for $s\in (0,3/2)$. Our argument relies on \cite{GMS18}. 

\begin{theorem}
\label{thmC.1}
Let $s\in (0,3/2)$. 
We have $\mathcal{D}((-\Delta_\gamma+\lambda)^{s/2})=H^s(\mathbb{R})$. 
Moreover, we have
\begin{align*}
	\|f\|_{H_\gamma^s}:=\|(-\Delta_\gamma+\lambda)^{s/2}f\|_{L^2} \approx \|f\|_{H^s}.
\end{align*}
\end{theorem}

\subsubsection{Elementary lemmas}

Before considering Theorem \ref{thmC.1}, we give some elementary lemmas. We set 
\begin{align*}
	G_k(x)= i(2k)^{-1}e^{ik|x|}
\end{align*}
for $k\in \mathbb{C}$. Then easy calculations imply the following. 

\begin{lemma}
\label{lemC.2}
For $x\geq 0$ and $s\in (0,2)$, we have 
\begin{align*}
	x^{s/2}=\frac{\sin\left(\frac{\pi}{2}s\right)}{\pi} \int_{0}^{\infty} t^{\frac{s}{2}-1} \frac{x}{t+x} dt.
\end{align*}
\end{lemma}

\begin{lemma}
\label{lemC.4}
We have
\begin{align*}
	\widehat{G_{i\sqrt{\lambda+t}}}(\xi)
	=\frac{1}{\sqrt{2\pi}(|\xi|^2+\lambda+t)},
\end{align*}
where $\widehat{f}(\xi)=(2\pi)^{-1/2}\int f(x)e^{-ix\xi}dx$. 
\end{lemma}

The formula of the resolvent of $-\Delta_\gamma - k^2$ is known as follows. 

\begin{lemma}[{\cite[p.77, Theorem 3.1.2]{AGHKH88}}]
\label{lemC.3}
Let $ \rho (T)$ denote the resolvent set of an operator $T$. 
For $k^2 \in \rho (-\Delta_{\gamma})$, $\Im k >0$, it holds that
\begin{align*}
	(-\Delta_\gamma - k^2)^{-1}=(-\Delta_0 - k^2)^{-1} + 2\gamma k (-i\gamma +2k)^{-1} \langle  \cdot , \overline{G_k} \rangle G_k,
\end{align*}
where 
%$G_k = i(2k)^{-1}e^{ik|x|}$ and 
we recall $\langle f , g \rangle=\int_{\mathbb{R}} f(y)\overline{g(y)}dy$.
\end{lemma}

%{\color{gray}
%\begin{proof}
%See \cite[p.77, Theorem 3.1.2]{AGHKH88}. 
%\end{proof}
%}

\subsubsection{Proof of Theorem \ref{thmC.1}}
To show Theorem \ref{thmC.1}, we prove some lemmas. 

\begin{lemma}
\label{lemB.5}
Let $s \in (0,2)$. 
A function $g\in \mathcal{D}((-\Delta_\gamma+\lambda)^{s/2})$ is written by
\begin{align*}
	g
	=\mathcal{A}(g) + \mathcal{B}(g),
\end{align*}
where 
\begin{align*}
	\mathcal{A}(g) &:=  (-\Delta_0 + \lambda)^{-s/2}(-\Delta_\gamma + \lambda)^{s/2} g,
	\\
	\mathcal{B}(g)
	&:=\frac{2\gamma\sin\left(\frac{\pi}{2}s\right)}{\pi} \int_{0}^{\infty} t^{-\frac{s}{2}}  \frac{\sqrt{\lambda+t} c_g^\gamma(t)}{-\gamma +2\sqrt{\lambda+t} } 
	 \frac{e^{-\sqrt{\lambda+t}|x|}}{2\sqrt{\lambda+t}} dt,
%	- \frac{2\gamma\sin\left(\frac{\pi}{2}s\right)}{\pi} \int_{0}^{\infty} t^{\frac{s}{2}-2}  \sqrt{\lambda+t^{-1}} (\gamma +2\sqrt{\lambda+t^{-1}})^{-1} 
%	\langle \overline{G_{i\sqrt{\lambda+t^{-1}}}} , (-\Delta_\gamma + \lambda)^{s/2} g \rangle G_{i\sqrt{\lambda+t^{-1}}} dt.
	\\
	c_g^\gamma(t) &:=\langle  (-\Delta_\gamma + \lambda)^{s/2} g, \overline{G_{i\sqrt{\lambda+t}}}  \rangle.
\end{align*}
Moreover, $\mathcal{A}(g) \in H^s(\mathbb{R})$, $\mathcal{B}(g) \in L^2(\mathbb{R})$, and $\mathcal{A}:\mathcal{D}((-\Delta_\gamma+\lambda)^{s/2})\to H^s(\mathbb{R})$ is surjective. 
\end{lemma}

\begin{proof}

By fractional calculus and Lemma \ref{lemC.2}, we have 
\begin{align*}
	(-\Delta_\gamma + \lambda)^{-s/2} 
	&= \frac{\sin\left(\frac{\pi}{2}s\right)}{\pi} \int_{0}^{\infty} t^{\frac{s}{2}-1} (-\Delta_\gamma + \lambda)^{-1}\{t+(-\Delta_\gamma + \lambda)^{-1}\}^{-1} dt
	\\
	& = \frac{\sin\left(\frac{\pi}{2}s\right)}{\pi} \int_{0}^{\infty} t^{\frac{s}{2}-2} \{-\Delta_\gamma +( \lambda+t^{-1})\}^{-1} dt.
\end{align*}

Applying Lemma \ref{lemC.3} as $k=i\sqrt{\lambda+t^{-1}}$, we get 
\begin{align*}
	(-\Delta_\gamma + \lambda)^{-s/2} 
	&= \frac{\sin\left(\frac{\pi}{2}s\right)}{\pi} \int_{0}^{\infty} t^{\frac{s}{2}-2} \{-\Delta_\gamma +( \lambda+t^{-1})\}^{-1} dt
	\\
	&= \frac{\sin\left(\frac{\pi}{2}s\right)}{\pi} \int_{0}^{\infty} t^{\frac{s}{2}-2} \{-\Delta_0 +( \lambda+t^{-1})\}^{-1} dt
	\\
	&\quad + \frac{2\gamma\sin\left(\frac{\pi}{2}s\right)}{\pi} \int_{0}^{\infty} t^{\frac{s}{2}-2} \frac{ \sqrt{\lambda+t^{-1}} \langle  \cdot , \overline{G_{i\sqrt{\lambda+t^{-1}}}} \rangle }{-\gamma +2\sqrt{\lambda+t^{-1}}}  G_{i\sqrt{\lambda+t^{-1}}} dt
	\\
	&= (-\Delta_0 + \lambda)^{-s/2}  
	\\
	&\quad + \frac{2\gamma\sin\left(\frac{\pi}{2}s\right)}{\pi} \int_{0}^{\infty} t^{\frac{s}{2}-2} \frac{ \sqrt{\lambda+t^{-1}} \langle  \cdot , \overline{G_{i\sqrt{\lambda+t^{-1}}}} \rangle }{-\gamma +2\sqrt{\lambda+t^{-1}}}  G_{i\sqrt{\lambda+t^{-1}}} dt.
\end{align*}

Applying this to $(-\Delta_\gamma + \lambda)^{s/2} g$ for $g\in \mathcal{D}((-\Delta_\gamma+\lambda)^{s/2})$, we get
\begin{align*}
	g
	=\mathcal{A}(g) + \mathcal{B}(g),
\end{align*}
where it holds that
\begin{align*}
	\mathcal{B}(g)=\frac{2\gamma\sin\left(\frac{\pi}{2}s\right)}{\pi} \int_{0}^{\infty} t^{\frac{s}{2}-2} \frac{ \sqrt{\lambda+t^{-1}} \langle (-\Delta_\gamma + \lambda)^{s/2} g , \overline{G_{i\sqrt{\lambda+t^{-1}}}} \rangle }{-\gamma +2\sqrt{\lambda+t^{-1}}}  G_{i\sqrt{\lambda+t^{-1}}} dt
\end{align*}
by $G_{i\sqrt{\lambda+t^{-1}}}(x)=\frac{e^{-\sqrt{\lambda+t^{-1}}|x|}}{2\sqrt{\lambda+t^{-1}}}$, the definition of $c_g^\gamma$, and change of variables. Therefore, the first statement is shown. 
The statement $\mathcal{A}(g) \in H^s(\mathbb{R})$ is obvious and 
\begin{align*}
	\mathcal{B}(g)=g-\mathcal{A}(g) \in L^2(\mathbb{R}).
\end{align*}
Take $\widetilde{f}\in H^s(\mathbb{R})$. Then
\begin{align*}
	\widetilde{g}:=(-\Delta_\gamma + \lambda)^{-s/2}(-\Delta_0 + \lambda)^{s/2}\widetilde{f} \in \mathcal{D}((-\Delta_\gamma+\lambda)^{s/2})
\end{align*}
and $\mathcal{A}(\widetilde{g}) =\widetilde{f}$. This shows that $\mathcal{A}:\mathcal{D}((-\Delta_\gamma+\lambda)^{s/2})\to H^s(\mathbb{R})$ is surjective. 
\end{proof}

\begin{lemma}
\label{lemB.6}
Let $s \in (0,2)$. 
For $\varphi \in L^2(\mathbb{R})$, we have
\begin{align*}
	(-\Delta_\gamma + \lambda)^{s/2} \varphi 
	=(-\Delta_0 + \lambda)^{s/2} \varphi 
	+\frac{\sin\left(\frac{\pi}{2}s\right)}{\pi} \int_{0}^{\infty} \frac{-\gamma t^{\frac{s}{2}} \kappa_\varphi(t) }{-\gamma+2\sqrt{t+\lambda}}  e^{-\sqrt{t+\lambda}|x|} dt
\end{align*}
in the distributional sense, where
\begin{align*}
	\kappa_{\varphi}(t):=\langle \varphi , \overline{G_{i\sqrt{t+\lambda}}} \rangle.
\end{align*}
\end{lemma}

\begin{proof}
By Lemma \ref{lemC.2} and fractional calculus, we have
\begin{align*}
	(-\Delta_\gamma + \lambda)^{s/2} \varphi
	&= \frac{\sin\left(\frac{\pi}{2}s\right)}{\pi} \int_{0}^{\infty} t^{\frac{s}{2}-1} (-\Delta_\gamma + \lambda) \{t+(-\Delta_\gamma + \lambda)\}^{-1} \varphi dt
	\\
	&=\frac{\sin\left(\frac{\pi}{2}s\right)}{\pi} \int_{0}^{\infty} t^{\frac{s}{2}-1} \{(-\Delta_\gamma + t+ \lambda) -t\} (-\Delta_\gamma +t+ \lambda)^{-1} \varphi dt
	\\
	&=\frac{\sin\left(\frac{\pi}{2}s\right)}{\pi} \int_{0}^{\infty} t^{\frac{s}{2}-1} \{1 - t(-\Delta_\gamma +t+ \lambda)^{-1}\} \varphi dt.
\end{align*}
We also have the similar equation for $(-\Delta_0 + \lambda)^{s/2} \varphi$. Thus we get
\begin{align*}
	&(-\Delta_\gamma + \lambda)^{s/2} \varphi - (-\Delta_0 + \lambda)^{s/2} \varphi
	\\
	&=\frac{-\sin\left(\frac{\pi}{2}s\right)}{\pi} \int_{0}^{\infty} t^{\frac{s}{2}} \{ (-\Delta_\gamma +t+ \lambda)^{-1} -(-\Delta_0 +t+ \lambda)^{-1}\}  \varphi dt.
\end{align*}
By using Lemma \ref{lemC.3} as $k=i(t+\lambda)^{1/2}$, we obtain
\begin{align*}
	&(-\Delta_\gamma + \lambda)^{s/2} \varphi - (-\Delta_0 + \lambda)^{s/2} \varphi
	\\
	&=\frac{\sin\left(\frac{\pi}{2}s\right)}{\pi} \int_{0}^{\infty} \frac{-\gamma t^{\frac{s}{2}}}{-\gamma+2\sqrt{t+\lambda}} \langle \varphi , \overline{G_{i\sqrt{t+\lambda}}} \rangle e^{-\sqrt{t+\lambda}|x|} dt.
\end{align*}
%We set
%\begin{align*}
%	\kappa_{\varphi}(t):=\int_{\mathbb{R}} \frac{e^{-\sqrt{t+\lambda}|y|}\varphi(y)}{2\sqrt{t+\lambda}} dy = \langle G_{i\sqrt{t+\lambda}} , \varphi \rangle.
%\end{align*}
This completes the proof.
\end{proof}

We set
\begin{align*}
	c_f^0(t):=\langle \overline{G_{i\sqrt{\lambda+t}}} , (-\Delta_0 + \lambda)^{s/2} f \rangle
	=\int_{\mathbb{R}}  \frac{e^{-\sqrt{\lambda+t}|y|}}{2\sqrt{\lambda+t}}(-\Delta_0+ \lambda)^{s/2} f(y)dy,
\end{align*}
for $f \in H^s(\mathbb{R})$. 

\begin{lemma}
\label{lemB.7}
Let $s>0$. For $f \in H^s(\mathbb{R})$, we have
\begin{align*}
	|c_f^0(t)|
	\lesssim_\lambda  (1+t)^{-3/4}  \|f\|_{H^s}.
\end{align*}
\end{lemma}

\begin{proof}
By the Cauchy--Schwarz inequality, we get
\begin{align*}
	|c_f^0(t)|
	\leq \|G_{i\sqrt{\lambda+t}}\|_{L^2} \|(-\Delta_0+ \lambda)^{s/2} f\|_{L^2}=\|G_{i\sqrt{\lambda+t}}\|_{L^2} \|f\|_{H^s}.
\end{align*}
Now we have
\begin{align*}
	\|G_{i\sqrt{\lambda+t}}\|_{L^2}^2
	=\int_{\mathbb{R}} \frac{e^{-2\sqrt{\lambda+t}|y|}}{4(\lambda+t)} dy
	\approx (\lambda+t)^{-3/2}. 
\end{align*}
This shows the result. 
\end{proof}

For $F\in H^s(\mathbb{R})$, we set
\begin{align*}
	I_F :=\frac{\sin\left(\frac{\pi}{2}s\right)}{\pi} \int_{0}^{\infty} \frac{-\gamma t^{\frac{s}{2}}\kappa_F(t)}{-\gamma+2\sqrt{t+\lambda}}  e^{-\sqrt{t+\lambda}|x|} dt 
\end{align*}
and then
$(-\Delta_\gamma + \lambda)^{s/2} F = (-\Delta_0 + \lambda)^{s/2} F +I_F$ by Lemma \ref{lemB.6}.
We show that $I_F \in L^2(\mathbb{R})$. 

\begin{lemma}
\label{lemC.8}
Let $s\in (0,3/2)$. 
We have $\|I_F\|_{L^2} \lesssim \|F\|_{H^s}$.
\end{lemma}

\begin{proof}
We have
\begin{align*}
	\|I_F\|_{L^2}^2
	&=\|\widehat{I_F}\|_{L^2}^2
	\\
	&=\int_{\mathbb{R}}\left| \frac{\sin\left(\frac{\pi}{2}s\right)}{\pi} \int_{0}^{\infty} \frac{-\gamma t^{\frac{s}{2}}}{-\gamma+2\sqrt{t+\lambda}} \kappa_F(t) \frac{2\sqrt{\lambda+t}}{\sqrt{2\pi}(|\xi|^2+\lambda+t)} dt \right|^2d\xi
	\\
	&\lesssim \int_{\mathbb{R}} \left( \int_{0}^{\infty} \frac{ t^{\frac{s}{2}}}{-\gamma+2\sqrt{t+\lambda}} |\kappa_F(t)| \frac{\sqrt{\lambda+t}}{|\xi|^2+\lambda+t} dt \right)^2 d\xi
\end{align*}
Here, for $a \in [0,1]$, we have
\begin{align*}
	|\kappa_F(t)| 
	&= |\langle F , \overline{G_{i\sqrt{t+\lambda}}} \rangle| 
	\\
	&\approx \left| \int_{\mathbb{R}} \widehat{F}(\eta)  \frac{1}{|\eta|^2+\lambda+t} d\eta \right|
	\\
	&\lesssim  \int_{\mathbb{R}} (|\eta|^2+\lambda)^{s/2}|\widehat{F}(\eta)|  \frac{(|\eta|^2+\lambda)^{-s/2}}{|\eta|^2+\lambda+t} d\eta 
	\\
	&\lesssim  \int_{\mathbb{R}} (|\eta|^2+\lambda)^{s/2}|\widehat{F}(\eta)|  \frac{(|\eta|^2+\lambda)^{-s/2}}{(|\eta|^2+\lambda)^a(\lambda+t)^{1-a}} d\eta 
	\\
	&\lesssim  \frac{1}{(\lambda+t)^{1-a}} \int_{\mathbb{R}} (|\eta|^2+\lambda)^{s/2}|\widehat{F}(\eta)| (|\eta|^2+\lambda)^{-s/2-a} d\eta 
	\\
	&\lesssim  \frac{1}{(\lambda+t)^{1-a}} \|F\|_{H^s} \left( \int_{\mathbb{R}}(|\eta|^2+\lambda)^{-s-2a} d\eta \right)^{\frac{1}{2}}.
\end{align*}
The integral of the last term is finite if $a> \frac{1-2s}{4}$. By taking $a\in [\max\{0,\frac{1-2s}{4}\}, 1]$ (since $s>0$, we can choose such a power $a$), then we get
\begin{align*}
	&\int_{\mathbb{R}} \left( \int_{0}^{\infty} \frac{ t^{\frac{s}{2}}}{-\gamma+2\sqrt{t+\lambda}} |\kappa_F(t)| \frac{\sqrt{\lambda+t}}{|\xi|^2+\lambda+t} dt \right)^2 d\xi
	\\
	&\lesssim \int_{\mathbb{R}} \left( \int_{0}^{\infty} \frac{ t^{\frac{s}{2}} }{-\gamma+2\sqrt{t+\lambda}}  \frac{\|F\|_{H^s}}{(\lambda+t)^{1-a}}  \frac{\sqrt{\lambda+t}}{|\xi|^2+\lambda+t}  dt \right)^2  d\xi
	\\
	&\lesssim \int_{\mathbb{R}} \left( \int_{0}^{\infty} \frac{ t^{\frac{s}{2}}  }{-\gamma+2\sqrt{t+\lambda}}  \frac{\|F\|_{H^s}}{(\lambda+t)^{1-a}}  \frac{\sqrt{\lambda+t}}{(|\xi|^2+\lambda)^b(\lambda+t)^{1-b}}  dt \right)^2  d\xi
	\\
	&\lesssim  \left( \int_{0}^{\infty} \frac{ t^{\frac{s}{2}} }{-\gamma+2\sqrt{t+\lambda}}  \frac{\sqrt{\lambda+t}}{(\lambda+t)^{2-a-b}}    dt 
	\right)^2 \int_{\mathbb{R}} \frac{1}{(|\xi|^2+\lambda)^{2b}}  d\xi \|F\|_{H^s}^2
\end{align*}
for $b \in [0,1]$. If $b>1/4$, then the integral with respect to $\xi$ is finite. Moreover, if $\frac{s}{2}-(2-a-b)<-1$, then  the integral with respect to $t$ is finite. 
Taking $a=\max\{0,\frac{1-2s}{4}\} +\varepsilon$ and $b=\frac{1}{4}+\varepsilon$, where $\varepsilon >0$ is sufficiently small, then we get
\begin{align*}
	\frac{s}{2}-(2-a-b)
	%=\frac{s-4}{2} + \min \left\{0,\frac{1-2s}{4}\right\} +\frac{1}{4}+2\varepsilon
	=\max\left\{\frac{2s-7}{4}, -\frac{3}{2} \right\} +2\varepsilon
	<-1
\end{align*}
since $s<3/2$. Thus, we get the statement.

\end{proof}

Finally, we give a proof of Theorem \ref{thmC.1}. 

\begin{proof}[Proof of Theorem \ref{thmC.1}]
Let $F \in H^s(\mathbb{R}).$ Lemma \ref{lemC.8} shows that 
\begin{align*}
	(-\Delta_\gamma + \lambda)^{s/2} F = (-\Delta_0 + \lambda)^{s/2} F +I_F \in L^2(\mathbb{R})
\end{align*}
and thus $F\in\mathcal{D}((-\Delta_\gamma + \lambda)^{s/2})$. This means that $H^s(\mathbb{R}) \subset \mathcal{D}((-\Delta_\gamma + \lambda)^{s/2})$. 

Next we show $H^s(\mathbb{R}) \supset \mathcal{D}((-\Delta_\gamma + \lambda)^{s/2})$. Let $g \in \mathcal{D}((-\Delta_\gamma + \lambda)^{s/2})$. Then, by Lemma \ref{lemB.5}, we have $g=\mathcal{A}(g)+\mathcal{B}(g)$. Now, letting 
\begin{align*}
	\mathcal{C}(f):=\frac{2\gamma\sin\left(\frac{\pi}{2}s\right)}{\pi} \int_{0}^{\infty} t^{-\frac{s}{2}}  \frac{\sqrt{\lambda+t}}{-\gamma +2\sqrt{\lambda+t}}
	c_f^0(t) G_{i\sqrt{\lambda+t}} dt
\end{align*}
for $f \in H^s(\mathbb{R})$, we have $\mathcal{B}(g)= \mathcal{C}(\mathcal{A}(g))$. Indeed, this follows from 
\begin{align*}
	c_{g}^\gamma(t)
	=\langle (-\Delta_\gamma + \lambda)^{s/2} g, \overline{G_{i\sqrt{\lambda+t^{-1}}}}  \rangle
	=\langle  (-\Delta_0 + \lambda)^{s/2}\mathcal{A}(g), \overline{G_{i\sqrt{\lambda+t^{-1}}}}  \rangle
	=c_{\mathcal{A}(g)}^0(t).
\end{align*}
Thus we have $g=\mathcal{A}(g)+\mathcal{C}(\mathcal{A}(g))$, where we note that $\mathcal{A}(g) \in H^s(\mathbb{R})$. For $f \in H^s(\mathbb{R})$, we have $\mathcal{C}(f) \in H^s(\mathbb{R})$. Indeed,  by Lemma \ref{lemB.7}, we have 
\begin{align*}
	\|\mathcal{C}(f)\|_{H^s}^2
	&=\int_{\mathbb{R}} |(|\xi|^2+\lambda)^{s/2}\widehat{\mathcal{C}(f)}(\xi)|^2 d\xi
	\\
	&\lesssim \int_{\mathbb{R}} (|\xi|^2+\lambda)^{s}\left|\int_{0}^{\infty} t^{-\frac{s}{2}}  \frac{\sqrt{\lambda+t}}{ -\gamma +2\sqrt{\lambda+t}}
	c_f^0(t) \widehat{G_{i\sqrt{\lambda+t}}}(\xi) dt\right|^2 d\xi
	\\
	&\lesssim \int_{\mathbb{R}} (|\xi|^2+\lambda)^{s}\left|\int_{0}^{\infty} t^{-\frac{s}{2}}  \frac{\sqrt{\lambda+t}}{ -\gamma +2\sqrt{\lambda+t}}
	c_f^0(t) \frac{1}{(|\xi|^2+\lambda+t)} dt\right|^2 d\xi
	\\
	&\lesssim \int_{\mathbb{R}} (|\xi|^2+\lambda)^{s}  \left|\int_{0}^{\infty} t^{-\frac{s}{2}}  \frac{\sqrt{\lambda+t}}{ -\gamma +2\sqrt{\lambda+t}}
	c_f^0(t) \frac{1}{(|\xi|^2+\lambda)^a(\lambda+t)^{1-a}} dt\right|^2 d\xi
	\\
	&\lesssim \int_{\mathbb{R}} (|\xi|^2+\lambda)^{s-2a} d\xi\left|\int_{0}^{\infty} t^{-\frac{s}{2}} \frac{\sqrt{\lambda+t}}{ -\gamma +2\sqrt{\lambda+t}}
	c_f^0(t) (\lambda+t)^{-1+a}dt\right|^2
	\\
	&\lesssim  \int_{\mathbb{R}} (|\xi|^2+\lambda)^{s-2a} d\xi \left|\int_{0}^{\infty} t^{-\frac{s}{2}}c_f^0(t)(\lambda+t)^{-1+a} dt\right|^2 
	\\
	&\lesssim  \int_{\mathbb{R}} (|\xi|^2+\lambda)^{s-2a} d\xi \left|\int_{0}^{\infty} t^{-\frac{s}{2}}(1+t)^{-\frac{3}{4}-1+a} dt\right|^2 \|f\|_{H^s}^2,
\end{align*}
where $0\leq a\leq 1$. 
Take $a=\frac{2s+9}{12}$, which is in the interval $(3/4,1)$ for $s\in (0,3/2)$. Then we have
\begin{align*}
	2s-4a
%	= 2s -\frac{2s+9}{3} 
	= \frac{4}{3}s -3 < -1,
	\quad
	-\frac{s}{2}-\frac{3}{4}-1+a 
%	= -\frac{s}{2}-\frac{3}{4}-1 + \frac{2s+9}{12}
	=-\frac{1}{3}s -1 <-1,
	\text{ and }
	 -\frac{s}{2} > -1
%	\\
%	& -\frac{s}{2} > -1,
\end{align*}
for any $s\in (0,3/2)$ and thus the integrals for $\xi$ and $t$ are finite. This implies 
\begin{align*}
	\|\mathcal{C}(f)\|_{H^s}^2 \lesssim \|f\|_{H^s}^2.
\end{align*}
Therefore, we get $g = \mathcal{A}(g) +\mathcal{C}(\mathcal{A}(g))  \in H^s$. 
This means $ \mathcal{D}((-\Delta_\gamma + \lambda)^{s/2}) \subset H^s$.

%Therefore we obtain $ \mathcal{D}((-\Delta_\gamma + \lambda)^{s/2}) = H^s$. 

At last, we show the norm equivalency. 
Let $g\in \mathcal{D}((-\Delta_\gamma + \lambda)^{s/2}) = H^s$. 
Since we have $g=\mathcal{A}(g)+\mathcal{C}(\mathcal{A}(g))$ and $\|\mathcal{C}(f)\|_{H^s} \lesssim \|f\|_{H^s}$, we obtain
\begin{align*}
	\|g\|_{H^s} \leq \|\mathcal{A}(g)\|_{H^s} + \|\mathcal{C}(\mathcal{A}(g))\|_{H^s} \lesssim  \|\mathcal{A}(g)\|_{H^s}.
\end{align*}
We have 
\begin{align*}
	\|\mathcal{A}(g)\|_{H^s}
	=\|(-\Delta_0 + \lambda)^{s/2}\mathcal{A}(g)\|_{L^2}
	=\|(-\Delta_\gamma + \lambda)^{s/2}g\|_{L^2}
	=\|g\|_{H_\gamma^s}
\end{align*}
and thus we get $\|g\|_{H^s} \lesssim \|g\|_{H_\gamma^s}$. 
%and by the above calculation we have
%\begin{align*}
%	\|\mathcal{B}(g)\|_{H^s}=\|\mathcal{C}(\mathcal{A}(g))\|_{H^s}\lesssim \|\mathcal{A}(g)\|_{H^s}.
%\end{align*}
%Thus we get
%\begin{align*}
%	\|g\|_{H^s} \leq \|\mathcal{A}(g)\|_{H^s} + \|\mathcal{B}(g)\|_{H^s}
%	\lesssim \|\mathcal{A}(g)\|_{H^s} \approx \|g\|_{H_\gamma^s}. 
%\end{align*}
Moreover, by Lemma \ref{lemC.8}, we also have
\begin{align*}
	\|g\|_{H_\gamma^s}
	=\|(-\Delta_\gamma + \lambda)^{s/2}g\|_{L^2}
	\leq \|(-\Delta_0 + \lambda)^{s/2}g\|_{L^2} + \|I_g\|_{L^2}
	\lesssim  \|g\|_{H^s}.
\end{align*}
This completes the proof. 
\end{proof}

\subsection{Sketch of the proof of local well-posedness}

We give a sketch proof of the local well-posedness of \eqref{deltaNLS} in $H^{s}$ for $s\in (1/2,1)$. 

By Duhamel's formula, we have
\begin{align*}
	\|u(t)\|_{H^s}
	\lesssim \|e^{it\Delta_\gamma}u_0\|_{H^s} + \|\int_{0}^{t} e^{i(t-\tau)\Delta_\gamma}(|u|^{p-1}u)(\tau)d\tau\|_{H^s}
\end{align*}
By the equivalency, i.e., Theorem \ref{thmC.1}, we get
\begin{align*}
	\|u(t)\|_{H^s}
	&\lesssim \|e^{it\Delta_\gamma}u_0\|_{H_\gamma^s} + \|\int_{0}^{t} e^{i(t-\tau)\Delta_\gamma}(|u|^{p-1}u)(\tau)d\tau\|_{H_\gamma^s}
	\\
	&\lesssim \|e^{it\Delta_\gamma}(-\Delta_\gamma + \lambda)^{s/2}u_0\|_{L^2} + \|\int_{0}^{t} e^{i(t-\tau)\Delta_\gamma}(-\Delta_\gamma + \lambda)^{s/2}(|u|^{p-1}u)(\tau)d\tau\|_{L^2}
	\\
	&\lesssim \|u_0\|_{H^s} + \|(-\Delta_\gamma + \lambda)^{s/2}(|u|^{p-1}u)\|_{L_t^1L^2}
	\\
	&\lesssim \|u_0\|_{H^s} + \||u|^{p-1}u\|_{L_t^1H^s}.
\end{align*}
By the fractional chain rule (see e.g. \cite[Lemma A.9]{Tao06} %and \cite[Corollary 6.4.5]{Hor97}, 
and \cite[Lemma 2.3 (2.15)]{HiWa19}) and the Sobolev embedding $H^s \subset L^\infty$ for $s>1/2$, 
\begin{align*}
	\|u(t)\|_{H^s}&\lesssim \|u_0\|_{H^s} + T\|u\|_{L_t^\infty L^\infty}^{p-1}\|u\|_{L_t^\infty H^s}
	\\
	&\lesssim \|u_0\|_{H^s} + T\|u\|_{L_t^\infty H^s}^{p}.
\end{align*}
This gives the local well-posedness for $H^s$, where $1/2<s<1$. 

%%%%%%%%%%%%%%%%%%%%%%%%%%%%%%%%%%%%%%%%%%
\section{$v_{k}\neq0$ for any $k$}
\label{appC}

As stated in Remark \ref{rmk3.1}, we need a small modification in the case that $v_{k}\neq0$ for any $k$. Now, let $K\geq 1$ and $v_{k}\neq0$ for any $k \in \llbracket 1,K\rrbracket$. We assume that 
\begin{align*}
	v_1<...<v_k<0<v_{k+1}<...<v_K. 
\end{align*}
%We regard $0$ as a (virtual) velocity. 
To the velocity family $\{v_1,...,v_K\}$, we add virtual velocity $0$. 
By renumbering, we may assume that $K \geq 2$ and $v_1<...<v_{k_0-1}<v_{k_0}=0<v_{k_0+1}<...<v_K$. Then we apply the same cut-off argument with the velocity family $\{v_1,...,v_{k_0},...,v_K\}$. 
By this trick, we do not need to calculate the time derivative of $\mathcal{P}_
{k_0}$ as in the case that there exists $k_0$ such that $v_{k_0}=0$. In the coercivity argument we modify the proof as follows: Now, we have 
\begin{align*}
	\mathcal{H}_{\gamma}(w)
	& =\sum_{k\neq k_{0}} \mathcal{I}_{k} 
	-\gamma |w(t,0)|^{2} 
	+\int_{\mathbb{R}} \psi_{k_{0}} (|\partial_{x}w(t,x)|^{2} + |w(t,x)|^{2} )dx + O(e^{-3c_0t}\|w\|_{H^1}^2)
	\\
	& \geq \sum_{k\neq k_{0}} \mathcal{I}_{k} 
	+\int_{\mathbb{R}} \psi_{k_{0}} (|\partial_{x}w(t,x)|^{2} + |w(t,x)|^{2} )dx+ O(e^{-3c_0t}\|w\|_{H^1}^2),
\end{align*}
where $\mathcal{I}_k$ is same as in the proof of Lemma \ref{lem3.12}. Since $\mathcal{I}_k$ can be estimated in the same way, we get
\begin{align*}
	\mathcal{I}_k \geq c \int \psi_k(|\partial_x w|^2+|w|^2) dx - \frac{C}{L}\|w\|_{L^2}^2 -c^{-1} \text{Ortho}_k. 
\end{align*}
Therefore, by $\sum_{k}\psi_k=1$, we obtain 
\begin{align*}
	\mathcal{H}_{\gamma}(w)
	&\geq c\sum_{k\neq k_{0}}\int \psi_k(|\partial_x w|^2+|w|^2) dx - \frac{C}{L}\|w\|_{L^2}^2 -c^{-1} \sum_{k\neq k_{0}}\text{Ortho}_k 
	\\
	&\quad + \int_{\mathbb{R}} \psi_{k_{0}} (|\partial_{x}w(t,x)|^{2} + |w(t,x)|^{2} )dx+ O(e^{-3c_0t}\|w\|_{H^1}^2)
	\\
	&\geq  c\|w\|_{H^1}^2 - \frac{C}{L}\|w\|_{L^2}^2 +O(e^{-c_0t}\|w\|_{H^1}) + O(e^{-3c_0t}\|w\|_{H^1}^2).
\end{align*}
This means that we have
\begin{align*}
	c\|w\|_{H^1}^2 \leq \mathcal{H}_{\gamma}(w) +  O(e^{-3c_0t}). 
\end{align*}
The rest of the arguments are same as in the case that there exists $k_0$ such that $v_{k_0}=0$.

%%%%%%%%%%%%%%%%%%%%%%%%%%%%%%%%%%%%%%%%%%

%%%%%Acknowledgement%%%%%

\section*{Acknowledgement}

Research of the first author is partially supported
by an NSERC Discovery Grant. 
The second author is supported by JSPS KAKENHI Grant-in-Aid for Early-Career Scientists No. JP18K13444. He also expresses deep appreciation toward JSPS Overseas Research Fellowship and the University of British Columbia for stay by the fellowship. 
The third author is supported by Grant-in-Aid for JSPS Fellows JP23KJ1416.

%%%%%References%%%%%

%\bibliographystyle{myrefstyle.bst}
%\bibliography{references.bib}

\begin{thebibliography}{99}

\bibitem{AGHKH88}
S.~Albeverio, F.~Gesztesy, R.~H{\o}egh-Krohn, and H.~Holden, \emph{Solvable
  models in quantum mechanics}, Texts and Monographs in Physics,
  Springer-Verlag, New York, 1988.

\bibitem{ArIn22}
A.~H. Ardila and T.~Inui, \emph{Threshold scattering for the focusing {NLS}
  with a repulsive {D}irac delta potential}, J. Differential Equations
  \textbf{313} (2022), 54--84.

\bibitem{BaVi16}
V.~Banica and N.~Visciglia, \emph{Scattering for {NLS} with a delta potential},
  J. Differential Equations \textbf{260} (2016), no.~5, 4410--4439.

\bibitem{CGNT07}
S.-M. Chang, S.~Gustafson, K.~Nakanishi, and T.-P. Tsai, \emph{Spectra of
  linearized operators for {NLS} solitary waves}, SIAM J. Math. Anal.
  \textbf{39} (2007/08), no.~4, 1070--1111.

\bibitem{Com14}
V.~Combet, \emph{Multi-existence of multi-solitons for the supercritical
  nonlinear {S}chr\"{o}dinger equation in one dimension}, Discrete Contin. Dyn.
  Syst. \textbf{34} (2014), no.~5, 1961--1993.

\bibitem{CMM11}
R.~C\^{o}te, Y.~Martel, and F.~Merle, \emph{Construction of multi-soliton
  solutions for the {$L^2$}-supercritical g{K}d{V} and {NLS} equations}, Rev.
  Mat. Iberoam. \textbf{27} (2011), no.~1, 273--302.

\bibitem{DuMe08}
T.~Duyckaerts and F.~Merle, \emph{Dynamics of threshold solutions for
  energy-critical wave equation}, Int. Math. Res. Pap. IMRP (2008), Art ID
  rpn002, 67.

\bibitem{DuRo10}
T.~Duyckaerts and S.~Roudenko, \emph{Threshold solutions for the focusing 3{D}
  cubic {S}chr\"{o}dinger equation}, Rev. Mat. Iberoam. \textbf{26} (2010),
  no.~1, 1--56.

\bibitem{FuJe08}
R.~Fukuizumi and L.~Jeanjean, \emph{Stability of standing waves for a nonlinear
  {S}chr\"{o}dinger equation with a repulsive {D}irac delta potential},
  Discrete Contin. Dyn. Syst. \textbf{21} (2008), no.~1, 121--136.

\bibitem{FOO08}
R.~Fukuizumi, M.~Ohta, and T.~Ozawa, \emph{Nonlinear {S}chr\"{o}dinger equation
  with a point defect}, Ann. Inst. H. Poincar\'{e} C Anal. Non Lin\'{e}aire
  \textbf{25} (2008), no.~5, 837--845.

\bibitem{GMS18}
V.~Georgiev, A.~Michelangeli, and R.~Scandone, \emph{On fractional powers of
  singular perturbations of the {L}aplacian}, J. Funct. Anal. \textbf{275}
  (2018), no.~6, 1551--1602.

\bibitem{GHW04}
R.~H. Goodman, P.~J. Holmes, and M.~I. Weinstein, \emph{Strong {NLS}
  soliton-defect interactions}, Phys. D \textbf{192} (2004), no.~3-4, 215--248.

\bibitem{Gri88}
M.~Grillakis, \emph{Linearized instability for nonlinear {S}chr\"{o}dinger and {K}lein-{G}ordon equations}, Comm. Pure Appl. Math. \textbf{41} (1988), no.~6, 747--774. 

\bibitem{GSS90}
M.~Grillakis, J.~Shatah, and W.~Strauss, \emph{Stability theory of solitary
  waves in the presence of symmetry. {II}}, J. Funct. Anal. \textbf{94} (1990),
  no.~2, 308--348.

\bibitem{GuIn23p2}
S.~Gustafson and T.~Inui, \emph{Scattering and Blow-up for threshold even solutions to the nonlinear Schr\"{o}dinger equation with repulsive delta potential at low frequencies}, preprint.

\bibitem{GuIn22p}
S.~Gustafson and T.~Inui, \emph{Threshold even solutions to the nonlinear
  schr{\"o}dinger equation with delta potential at high frequencies}, preprint,
  arXiv:2211.15591.

\bibitem{HiWa19}
K.~Hidano and C.~Wang, \emph{Fractional derivatives of composite functions and
  the {C}auchy problem for the nonlinear half wave equation}, Selecta Math.
  (N.S.) \textbf{25} (2019), no.~1, Paper No. 2, 28.

\bibitem{Hor97}
L.~H\"{o}rmander, \emph{Lectures on nonlinear hyperbolic differential
  equations}, Math\'{e}matiques \& Applications (Berlin) [Mathematics \&
  Applications], vol.~26, Springer-Verlag, Berlin, 1997.

\bibitem{IkIn17}
M.~Ikeda and T.~Inui, \emph{Global dynamics below the standing waves for the
  focusing semilinear {S}chr\"{o}dinger equation with a repulsive {D}irac delta
  potential}, Anal. PDE \textbf{10} (2017), no.~2, 481--512.

\bibitem{Inu23p}
T.~Inui, \emph{Remark on blow-up of the threshold solutions to the nonlinear
  schr\"{o}dinger equation with the repulsive dirac delta potential}, to appear
  in RIMS K\^{o}ky\^{u}roku Bessatsu.

\bibitem{Kat95}
T.~Kato, \emph{Perturbation theory for linear operators}, Classics in
  Mathematics, Springer-Verlag, Berlin, 1995.

\bibitem{LCFFKS08}
S.~Le~Coz, R.~Fukuizumi, G.~Fibich, B.~Ksherim, and Y.~Sivan, \emph{Instability
  of bound states of a nonlinear {S}chr\"{o}dinger equation with a {D}irac
  potential}, Phys. D \textbf{237} (2008), no.~8, 1103--1128.

\bibitem{MaMe06}
Y.~Martel and F.~Merle, \emph{Multi solitary waves for nonlinear
  {S}chr\"{o}dinger equations}, Ann. Inst. H. Poincar\'{e} C Anal. Non
  Lin\'{e}aire \textbf{23} (2006), no.~6, 849--864.

\bibitem{Mer90}
F.~Merle, \emph{Construction of solutions with exactly {$k$} blow-up points for
  the {S}chr\"{o}dinger equation with critical nonlinearity}, Comm. Math. Phys.
  \textbf{129} (1990), no.~2, 223--240.

\bibitem{Oht11}
M.~Ohta, \emph{Instability of bound states for abstract nonlinear
  {S}chr\"{o}dinger equations}, J. Funct. Anal. \textbf{261} (2011), no.~1,
  90--110.

\bibitem{Sch06}
W.~Schlag, \emph{Spectral theory and nonlinear partial differential equations: a survey}, Discrete Contin. Dyn. Syst. \textbf{15} (2006), no.~3, 703--723.

\bibitem{Sch12}
K.~Schm\"{u}dgen, \emph{Unbounded self-adjoint operators on {H}ilbert space},
  Graduate Texts in Mathematics, vol. 265, Springer, Dordrecht, 2012.

\bibitem{Tao06}
T.~Tao, \emph{Nonlinear dispersive equations}, CBMS Regional Conference Series
  in Mathematics, vol. 106, Published for the Conference Board of the
  Mathematical Sciences, Washington, DC; by the American Mathematical Society,
  Providence, RI, 2006.

\end{thebibliography}

\end{document}